\documentclass[twoside,a4paper]{amsart}
\usepackage[bookmarksnumbered,plainpages,driverfallback=dvipdfm,backref=page]{hyperref}
\usepackage[english]{babel}
\usepackage[T1]{fontenc}
\usepackage{xcolor}
\usepackage{graphicx}
\usepackage{float}
\usepackage{amsfonts,amsmath,amsthm,amssymb,latexsym}
\usepackage{vmargin}
\usepackage{amscd}
\usepackage[all]{xy}
\usepackage[final]{showlabels} 
\usepackage{enumerate}
\usepackage{version}
\usepackage[normalem]{ulem}
\usepackage{tikz-cd}
\usepackage{orcidlink}

\newtheorem{thm}{Theorem}[section]
\newtheorem{cor}[thm]{Corollary}
\newtheorem{lem}[thm]{Lemma}
\newtheorem{prop}[thm]{Proposition}
\theoremstyle{definition}
\newtheorem{defn}[thm]{Definition}
\newtheorem{es}[thm]{Example}

\newtheorem{rmk}[thm]{Remark}

\newcommand{\id}{\mathrm{Id}}
\newcommand{\im}{\mathrm{Im}}

\newcommand{\Nat}{\mathrm{Nat}}
\newcommand{\Hom}{\mathrm{Hom}}

\newcommand{\cc}{\mathcal{C}}
\newcommand{\dd}{\mathcal{D}}
\newcommand{\e}{\mathcal{E}}
\newcommand{\f}{\mathcal{F}}

\newcommand{\m}{\mathcal{M}}

\newcommand{\p}{\mathcal{P}}

\newcommand{\rd}[1]{{\color{red}{#1}}}

\newcommand{\Set}{{\sf Set}}

\newcommand{\vect}{{\sf Vec}}

\def\Coalg{{\sf Coalg}}

\allowdisplaybreaks[4]

\newenvironment{invisible}{{\noindent\sc \colorbox{yellow}{Invisible:}\;}\color{gray}}{\medskip}
\excludeversion{invisible}

\begin{document}

\title{Semiseparability of induction functors in a monoidal category}

\author{Lucrezia Bottegoni \orcidlink{0009-0005-8172-5605}}
\address{%
\parbox[b]{\linewidth} {(Lucrezia Bottegoni)}}
\email{lucrezia.bottegoni@edu.unito.it}
\author{Zhenbang Zuo~\orcidlink{0009-0008-9013-0365}}
\address{%
\parbox[b]{\linewidth}{(Zhenbang Zuo) University of Turin, Department of Mathematics ``G. Peano'', via
Carlo Alberto 10, I-10123 Torino, Italy}}
\email{zhenbang.zuo@edu.unito.it }

\date{}
\begin{abstract}
For any algebra morphism in a monoidal category, we provide sufficient conditions (which are also necessary if the unit is a left tensor generator) for the attached induction functor being semiseparable. Under mild assumptions, we prove that the semiseparability of the induction functor is preserved if one applies a lax monoidal functor.  Similar results are shown for the coinduction functors attached to coalgebra morphisms in a monoidal category. As an application, we study the semiseparability of combinations of (co)induction functors in the context of duoidal categories. 

\end{abstract}

\maketitle

\tableofcontents

\section{Introduction}
The notion of separable functor was introduced in \cite{NVV89}, and several results and applications are described e.g.\ in
\cite{CMZ02}. As main examples, given a ring morphism $\varphi:R\to S$, the restriction of scalars functor $\varphi_*:\m_S\to\m_R$ is separable if and only if the corresponding ring extension $S/R$ is
separable, 
while its left adjoint, the induction functor $\varphi^*=-\otimes_RS:\m_R\to\m_S$, is separable if and only if there exists an $R$-bimodule morphism $E:S\to R$ such that $E\circ \varphi=\id_R$. In \cite[Theorem 3.6]{BT15}, the previous  characterizations were extended to an algebra extension in a
monoidal category where the unit object is a left $\otimes$-generator.

Semiseparable functors are defined in \cite{AB22} as a weaker notion of separable and naturally full functors \cite{ACMM06}. 
In fact, a functor results to be separable (resp., naturally full) if and only if it is faithful (resp., full) and semiseparable, see \cite[Proposition 1.3]{AB22}. 
Besides, the Rafael-type Theorem \cite[Theorem 2.1]{AB22} characterizes the semiseparability of a left (resp., right) adjoint functor in terms of a regularity condition for the (co)unit of the adjunction.

In this paper, we study semiseparability (and natural fullness) for (co)induction functors attached to a (co)algebra morphism in a monoidal category $(\cc, \otimes, 1)$ with coequalizers (resp., equalizers). We prove that, given an algebra morphism $\varphi:R\to S$ in $\cc$, the induction functor $\varphi^*=-\otimes_RS:\cc_R\to\cc_S$ is semiseparable (resp., naturally full) if $\varphi$ is regular (resp., split-epi) as an $R$-bimodule morphism in $\cc$, see Theorem \ref{thm:phi*semisep} (resp. Proposition \ref{prop:phi*natfull}). The converse holds true as well, if the monoidal unit $1$ is a left $\otimes$-generator. As a particular case, we recover \cite[Proposition 3.1]{AB22} in which $\cc$ is the category $(\mathsf{Ab}, \otimes_\mathbb{Z}, \mathbb{Z})$ of abelian groups. 
We then provide examples for some  monoidal categories, such as the category $(\Set, \times, \{0\})$ of sets, the category $({}_H\m,\otimes, \Bbbk)$ of left modules over a bialgebra $H$, the category $({}_{H}\m_{H}, \otimes_{H}, H)$ of $H$-bimodules where $H$ is a separable $\Bbbk$-algebra. 

As expected, for an algebra $A$ in $\cc$ with unit $u_A$, the induction functor $u_A^*=-\otimes _1 A: \cc_1\to \cc_A$ is exactly the tensor functor $-\otimes A:\cc\to\cc_A$. In Proposition \ref{prop:algebra-induct-functor} we characterize the semiseparability of the latter functor without requiring extra assumptions on the monoidal category $\cc$. 
Besides, in an abelian monoidal category where the unit is a left $\otimes$-generator, for an algebra morphism $f$ with image factorization $f= \varphi \circ \psi$, we show that, if the induction functor $f^*$ is semiseparable, then $\psi^*=-\otimes_R \im(f)$ is naturally full and $\varphi^*= -\otimes_{\im(f)} S$ is separable, see Corollary \ref{semisimple factorization}.

Next we explore variations under which the semiseparability of induction functors is preserved or reflected. We are mainly interested in the variation obtained through monoidal functors, see Proposition \ref{prop:algebra-monoidalfunctor}, Proposition \ref{lax monoidal variation unit}, Proposition \ref{prop:monoidalfunct} and Proposition \ref{prop:refl-semisep-ind}.

Dual results hold true for coinduction functors attached to coalgebras in a monoidal category $\cc$ with equalizers. More explicitly, given a morphism of coalgebras $\psi:C\to D$ in $\cc$, the coinduction functor $\psi^*=-\square_DC:\cc^D\to\cc^C$ is semiseparable if $\psi$ is a regular morphism of $D$-bicomodules in $\cc$, see Theorem \ref{prop:semisep-coind-nocogen}. 
Moreover, in Proposition  \ref{prop:sep-coind-nocogen} (resp., Proposition \ref{prop:natfull-coind}), we show that $\psi^*$ is separable (resp., naturally full) if $\psi$ is a split-epi (resp., split-mono) as a $D$-bicomodule morphism in $\cc$. 
The converses hold true if the unit object $1$ of $\cc$ is a left $\otimes$-cogenerator. In particular, if $\cc$ is the category of $\Bbbk$-vector spaces we recover the known \cite[Proposition 3.8]{AB22}, \cite[Theorem 2.7]{CGN97}, and \cite[Examples 3.23 (1)]{ACMM06}. We give examples in the monoidal categories of sets, of bimodules over a ring $R$, of left modules over a bialgebra $H$, of $n$-by-$n$ matrices with entries in $\vect_f$, etc.

We then turn our attention to the setting of duoidal categories $(\cc,\circ, \bullet)$. Given two algebra morphisms $f_1$, $f_2$ in $\cc^\circ$ (with unit being a $\circ$-generator) such that the attached induction functors are semiseparable (resp., separable, naturally full), we prove that $(f_1\bullet f_2)^*$ is also semiseparable (resp., separable, naturally full), see Proposition \ref{prop:duoidal}. 
This result can be applied, for instance, to a monoidal category with finite products (Corollary \ref{additivemonoidal}), and to a pre-braided monoidal category (Corollary \ref{cor:prebraid}). Moreover, for algebras $R$, $S$ in $\cc^\circ$ such that the functors $-\circ R:\cc\to \cc_R$, $-\circ S: \cc\to\cc_S$ are semiseparable, we consider the semiseparability of  $-\circ (R\bullet S): \cc\to\cc_{R\bullet S}$, see Proposition \ref{prop:duoidal-delta}. Similar results can be shown for the coinduction functors, see e.g.\ Proposition \ref{prop:duoidal-coind}, Proposition \ref{prop:duoidal-coalgebras}.

\medskip

The paper is organized as follows. In Section \ref{background}, we recall some basic definitions and results about semiseparable, separable, and naturally full functors. Section \ref{sub:induct} investigates the semiseparability for the induction functor attached to an algebra morphism in a monoidal category. 
Several examples are provided in Subsection \ref{subsect:examples}. As a particular case, in Subsection \ref{subsect:alg-mon} we consider the tensor functors attached to algebras in a monoidal category. Subsection \ref{subsect:monoidal-functor} collects results for induction functors varied by a monoidal functor. In Section \ref{sect:coinduct}, we study the semiseparability for the coinduction functor attached to a coalgebra morphism in a monoidal category. Finally, in Section \ref{sect:duoidal} we show results for (co)induction functors in duoidal categories.
\medskip

\emph{Notations.} Given an object $X$ in a category $\cc$, the identity morphism on $X$ will be denoted either by $\id_X$ or $X$ for short. For categories $\cc$ and $\dd$, a functor $F:\cc\to \dd$ means a covariant functor. We denote the identity functor on $\cc$ by $\id_{\cc}$. 
Besides, we denote by $K$ a commutative unital ring and by $\Bbbk$ a field. By a ring (resp., $K$-algebra) we mean a unital associative ring (resp., $K$-algebra). The category of left modules over a ring $R$ is denoted by ${}_R\m$. Given a coalgebra $C$ over a field, the category of right $C$-comodules and their morphisms is denoted by $\m^{C}$. For an algebra $A$ in a monoidal category $\cc$, we denote by $m_A$ and $u_A$ its multiplication and unit, and by ${}_A\cc$ (resp., $\cc_A$) the category of left (resp., right) modules over $A$. For a coalgebra $C$ in a monoidal category $\cc$, we denote by $\Delta_C$ and $\varepsilon_C$ its comultiplication and counit, and by ${}^C\cc$ (resp., $\cc^C$) the category of left (resp., right) comodules over $C$. A morphism (resp., natural transformation) $f$ is called \emph{regular} if there is a morphism (resp., natural transformation) $g$ such that $f\circ g\circ f=f$.

\section{Background on semiseparable functors}\label{background}

Let $F: \cc \rightarrow \dd$ be a functor and consider the natural transformation
\begin{equation*}\label{nat_transf}
	\f : \Hom_{\cc}(-,-)\rightarrow \Hom_{\dd}(F-, F-),
\end{equation*}
defined by setting $\f_{X,X'}(f)= F(f)$, for any $f:X\rightarrow X'$ in $\cc$.

If there is a natural transformation $\p : \Hom_{\dd}(F-, F-)\rightarrow \Hom_{\cc}(-,-)$ such that
\begin{itemize}
	\item $\p\circ\f = \id $, then $F$ is called \emph{separable} \cite{NVV89};
	\item $\f\circ\p = \id $, then $F$ is called \emph{naturally full} \cite{ACMM06};
	\item $\f\circ\p\circ\f = \f $, then $F$ is called \emph{semiseparable} \cite{AB22}.
\end{itemize}

Semiseparability allows to characterize separable and naturally full functors as follows.
\begin{prop}\label{prop:sep} \cite[Proposition 1.3]{AB22}
	Let $F: \cc \rightarrow \dd$ be a functor. Then,
	\begin{itemize}
		\item[$i)$]$F$ is separable if, and only if, $F$ is semiseparable and faithful;
		\item[$ii)$]$F$ is naturally full if, and only if, $F$ is semiseparable and full.
	\end{itemize}
\end{prop}

Semiseparable functors are not closed under composition, see e.g. \cite[Example 3.3]{AB22}, but in the following cases the closeness is available.

\begin{lem}\label{lem:comp-semisep}\cite[Lemma 1.12]{AB22}
Let $F: \cc \rightarrow \dd$ and $G:\dd\rightarrow\e$ be functors and consider the composite $G\circ F:\cc\rightarrow \e$.
\begin{itemize}
\item[$i)$] If $F$ is semiseparable and $G$ is separable, then $G\circ F$ is semiseparable.
\item[$ii)$] If $F$ is naturally full and $G$ is semiseparable, then $G\circ F$ is semiseparable.
\end{itemize}
\end{lem}

A suitable idempotent natural transformation can be attached to any semiseparable functor, and it determines when the functor is separable.

\begin{prop}\label{prop:idempotent} \cite[Proposition 1.4]{AB22}
	Let $F:\cc\rightarrow \dd$ be a semiseparable functor through a natural transformation $\p$. Then,
	there is a unique idempotent natural transformation $e:\id_{\cc%
	}\rightarrow \id_{\cc}$ such that $Fe=\id_{F}$ with the following universal property: if $f,g:X\to Y$ are morphisms, then $Ff=Fg$ if and only if $e_Y\circ f=e_Y\circ g$. Explicitly, $e$ is defined on components by $e_{X}:=\mathcal{P}_{X,X}\left( \id_{FX}\right)$. 
\end{prop}

\begin{cor}\cite[Corollary 1.7]{AB22}\label{cor:sep-idemp}
    Let $F:\cc\rightarrow \dd$ be a semiseparable functor and let $e:\id_\cc\to \id_\cc$ be the associated natural transformation. Then, $F$ is separable if, and only if, $e=\id$.
\end{cor}

A key result for separable functors is the  \emph{Rafael Theorem} \cite[Theorem 1.2]{Raf90}, which provides a characterization of separability for functors that have an adjoint in terms of splitting properties for the (co)unit. 
A Rafael-type Theorem still holds for semiseparability. 

\begin{thm}\label{thm:rafael}\emph{\cite[Theorem 2.1]{AB22}}
	Let $F\dashv G:\dd\to\cc$ be an adjunction, with unit $\eta:\id_\cc\to GF$ and counit $\epsilon:FG\to\id_\dd$. Then,
	\begin{itemize}
		\item[$i)$] $F$ is semiseparable if, and only if, $\eta$ is regular;
		
		\item[$ii)$] $G$ is semiseparable if, and only if, $\epsilon$ is regular.
		
	\end{itemize}
\end{thm}

Several examples of semiseparable functors can be found in \cite[Section 3]{AB22}. We recall the main characterizations of semiseparability for the extension of scalars functors attached to a ring morphism, for the coinduction functor associated to a morphism of coalgebras, for the induction functor attached to a coring.

\begin{prop}\label{prop:inducfunc}\emph{\cite[Proposition 3.1]{AB22}}
	Let $\varphi :R\to S$ be a morphism of rings. Then, the extension of scalars functor $\varphi^*= S\otimes_{R}(-):{}_R\m\rightarrow {}_S\m$ is semiseparable if, and only if, $\varphi$ is a regular morphism of $R$-bimodules, i.e., there is $E\in {}_{R}\Hom_{R}(S,R)$ such that $\varphi\circ E\circ\varphi =\varphi$, i.e., such that $\varphi E(1_S)=1_S$.
\end{prop}

\begin{prop}\cite[Proposition 3.8]{AB22}
    Let $\psi :C\to D$ be a morphism of coalgebras. Then, the coinduction functor $\psi^* =(-)\square_D C:\m^D\to\m^C$ is semiseparable if, and only if, $\psi$ is a regular morphism of $D$-bicomodules if, and only if, there is a $D$-bicomodule morphism $\chi:D\to C$ such that $\varepsilon_C \circ\chi\circ\psi=\varepsilon_C$.
\end{prop}

\begin{prop}\cite[Theorem 3.10]{AB22}\label{prop:coring}
   Let $\cc$ be an $R$-coring. Then, the following are equivalent.
\begin{enumerate}
  \item[$(i)$] The induction functor $G:=(-)\otimes_R\cc : \m_R\to\m^{\cc}$ is semiseparable. 
  \item[$(ii)$] The coring counit $\varepsilon_{\cc}:\cc\to R$ is regular as a morphism of $R$-bimodules.
  \item[$(iii)$] There exists an invariant element $z\in\cc^R$ such that $\varepsilon_{\cc}(z)\varepsilon_{\cc}(c)=\varepsilon_{\cc}(c)$ (equivalently such that $\varepsilon_{\cc}(z)c=c$), for every $c\in \cc$.
\end{enumerate} 
\end{prop}

In the following sections, we will investigate semiseparability for (co)induction functors attached to (co)algebra morphisms in a monoidal category. For further details on (co)algebras in a monoidal category we refer to \cite{AM10}. Without loss of generality, we usually assume monoidal categories are strict, in view of Mac Lane coherence theorem. 

\section{Extension of scalars functors}\label{sub:induct} 

We first remind preliminary notions from \cite{BT15}, \cite[Section 3]{PaI}. 
Let $\cc$ be a monoidal category with coequalizers. Consider an algebra $A$ in $\cc$, $M\in \cc_{A}$
and $X\in {}_A\cc$, with structure morphisms $\nu^A_M:M\otimes A\to M$ and $\mu_X^A:A\otimes X\to X$, respectively. 
Let $(M\otimes _AX,q^{A}_{M,X})$ be the coequalizer of the parallel morphisms
$\nu^A_M\otimes\id_X$
and $\id_M\otimes \mu_X^A$ in $\cc$:

\begin{equation}\label{eq:coeq}
\xymatrix@C=1.5cm{M\otimes A\otimes X\ar@<-1ex>[r]_-{\id_M\otimes \mu_X^A}\ar@<1ex>[r]^-{\nu^A_{M}\otimes\id_X}& M\otimes X\ar[r]^-{q^A_{M,X}}& M\otimes _AX.}
\end{equation}

For a left $A$-linear morphism $f:X\to Y$ in $\cc$, let $\tilde{f}:=M\otimes_A f:M\otimes _AX\to M\otimes_AY$ be the unique morphism in $\cc$ such that
\[
\tilde{f}q^A_{M,X}=q^A_{M,Y} (\id_M\otimes f).\]
\begin{invisible}
    Indeed, $q^A_{M,Y}(\id_M\otimes f)(\id_M\otimes \mu^A_X)=q^A_{M,Y}(\id_M\otimes \mu^A_Y)(\id_M\otimes\id_A\otimes f)=q^A_{M,Y}(\nu^A_M\otimes\id_Y)(\id_M\otimes\id_A \otimes f)=q^A_{M,Y}(\nu^A_M\otimes f)=q^A_{M,Y}(\id_M\otimes f)(\nu^A_M\otimes\id_X)$.
\end{invisible}

Let $g:M\to N$ in $\cc_A$ and $Y\in {}_A\cc$. By $\hat{g}:= g\otimes_A Y :M\otimes _AY\to N\otimes _A Y$ we denote the unique morphism in $\cc$ such that 
\[
\hat{g}q^A_{M,Y}=q^A_{N,Y}(g\otimes\id_Y).
\]
For $M\in\cc_A$, $X\in\cc$, $Y\in {}_A\cc$, there are canonical natural isomorphisms $\Upsilon_M$, $\Upsilon_{M,X}$, $\Upsilon'_Y$:
\begin{itemize}
    \item $\Upsilon_M:M\otimes_AA\to M$, uniquely determined by $\Upsilon_Mq^A_{M,A}=\nu^A_M$; 
    \item $\Upsilon_{M,X}:M\otimes _A(A\otimes X)\to M\otimes X$, uniquely determined by  $\Upsilon_{M,X}q^A_{M,A\otimes X}=\nu^A_M\otimes \id_X$;
    \item $\Upsilon'_Y:A\otimes_AY\to Y$, uniquely determined by  $\Upsilon'_Yq^A_{A,Y}=\mu^A_Y$. 
\end{itemize}

One can check that 
\[
\Upsilon^{-1}_M=q^A_{M,A}(\id_M\otimes u_A);\quad \Upsilon'^{-1}_Y=q^A_{A,Y}(u_A\otimes \id_Y);\quad \Upsilon^{-1}_{M,X}=q^A_{M,A\otimes X}(\id_M\otimes u_A\otimes \id_X). 
\]
\begin{invisible}
We only prove the second equality, as the others follow similarly.

We have\[
\Upsilon'_Yq^A_{A,Y}(u_A\otimes \id_Y)=\mu^A_Y(u_A\otimes \id_Y)=\id,\]
and \[\begin{split}
q^A_{A,Y}&(u_A\otimes \id_Y)\Upsilon'_Yq^A_{A,Y}=q^A_{A,Y}(u_A\otimes \id_Y)\mu^A_Y=q^A_{A,Y}(u_A\otimes\mu^A_Y)\\&=q^A_{A,Y}(\id_A\otimes\mu^A_Y)(u_A\otimes \id_{A\otimes Y})=q^A_{A,Y}(m_A\otimes\id_Y)(u_A\otimes \id_{A\otimes Y})=q^A_{A,Y},
\end{split}
\]
so $q^A_{A,Y}(u_A\otimes \id_Y)\Upsilon'_Y=\id$, as $q^A_{A,Y}$ is an epimorphism.
\end{invisible}

Recall that an object $X$ of $\cc$ is \emph{right} (resp., \emph{left}) \emph{coflat} if the functor $X\otimes -$ (resp., $-\otimes X$) preserves coequalizers. An object of $\cc$ is \emph{coflat} if it is both left and right coflat.

Let $A, R$ be algebras in $\cc$. By \cite[Lemma 2.4]{BC07} we have:
\begin{itemize}
    \item[$i)$] if $A$ is right coflat, then for any $X\in {}_A\cc_R$ and $Y\in {}_R\cc$, the morphism $\mu^A_{X\otimes_R Y}:A\otimes (X\otimes_R Y)\to X\otimes_RY$, uniquely determined by $\mu^A_{X\otimes _RY}(\id_A\otimes q^R_{X,Y})=q^R_{X,Y}(\mu^A_X\otimes\id_Y)$, defines on $X\otimes_R Y$ a left $A$-module structure in $\cc$; 
\item[$ii)$] if $A$ is left coflat, then for any $X\in\cc_R$ and $Y\in{}_R\cc_A$, the morphism $\nu^A_{X\otimes _RY}:(X\otimes_R Y)\otimes A\to X\otimes_RY$, uniquely determined by $\nu^A_{X\otimes_R Y}(q^R_{X,Y}\otimes\id_A)=q^R_{X,Y}(\id_X\otimes \nu^A_Y)$, 
defines on $X\otimes_RY$ a right $A$-module structure in $\cc$.
\end{itemize}

\begin{defn}\label{def:ogenerator} \cite[Definition 3.1]{BT15}
	Let $\cc$ be a monoidal category. An object $P$ of $\cc$ is a \emph{left $\otimes$-generator} of $\cc$ if wherever we consider two morphisms $f,g:Y\otimes Z\to W$ in $\cc$ such that $f(\varepsilon\otimes \id_Z)=g(\varepsilon\otimes \id_Z)$, for all $\varepsilon :P\to Y$ in $\cc$, then $f=g$.
  \end{defn}  
  The notion of \emph{right $\otimes$-generator} can be defined symmetrically. We call \emph{two-sided $\otimes$-generator} an object which is both a left and right $\otimes$-generator.
	By taking $Z=1$ in the above definition we get that a left $\otimes$-generator of a monoidal category $\cc$ is necessarily a generator for $\cc$. 

In this section, unless stated otherwise, we always assume that monoidal categories have  coequalizers and their objects are coflat.

\subsection{Semiseparability of induction functors}\label{subsect:induct}

Let $\varphi:R\to S$ be  an algebra morphism in $\cc$.
Observe that $S$ becomes an $R$-bimodule with $R$-module structures given by $m_S(\varphi\otimes \id_S)$ and $m_S(\id_S\otimes \varphi)$.
Consider the restriction of scalars functor $\varphi_*:\cc_S\to\cc_R$ and its left adjoint, the induction functor 
\[
\varphi^*:=-\otimes_RS:\cc_R\to\cc_S, \quad X\mapsto X\otimes_RS, \quad f\mapsto \hat{f},
\]
where the right $S$-module structure of $X\otimes_RS$ is induced by $m_S$.
The unit and the counit of the adjunction are given by
\[
\eta_X:=q^R_{X,S}(\id_X\otimes u_S):X\to X\otimes_RS,\quad \epsilon_M:=\overline{\nu^S_M}:M\otimes_RS\to M,
\]
where the latter is uniquely determined by $\overline{\nu^S_M}q^R_{M,S}=\nu^S_M$. 
\begin{invisible}
Indeed, $\nu^S_M(\id_M\otimes \mu^R_S)=\nu^S_M(\id_M\otimes m_S(\varphi\otimes\id_S))=\nu^S_M(\nu^S_M\otimes\id_S)(\id_M\otimes\varphi\otimes\id_S)=\nu^S_M(\nu^R_M\otimes\id_S)$.\end{invisible}

\begin{rmk}\label{rmk:comp-phi-psi}
   If $\psi:L\to R$ and $\varphi:R\to S$ are algebra morphisms in $\cc$, then $\varphi^*\circ\psi^*\cong(\varphi\circ\psi)^*$. Indeed, $(\varphi\circ\psi)_*=\psi_*\circ\varphi_*$ is the right adjoint of $\varphi^*\circ\psi^*$ 
   by \cite[Proposition 3.2.1]{Bor94}.
Meanwhile, we have the adjunction $(\varphi\circ\psi)^*\dashv(\varphi\circ\psi)_*$, hence $\varphi^*\circ\psi^*\cong(\varphi\circ\psi)^*$. 
\end{rmk}

Since $\varphi_*$ is faithful, by Proposition \ref{prop:sep}, it results to be semiseparable if and only if it is separable. Next we want to investigate the  semiseparability of $\varphi^*$. We recall that the separability of   $\varphi^*$ was characterized in \cite{BT15} as follows. 
\begin{thm}\label{thm:sepmon}\cite[Theorem 3.6]{BT15} 
    Let $\cc$ be a monoidal category such that $1$ is a left $\otimes$-generator. Consider an algebra morphism $\varphi:R\to S$ in $\cc$. 
Then, the induction functor $\varphi^*=-\otimes_RS:\cc_R\to\cc_S$ is separable if, and only if, there exists an $R$-bimodule morphism $E:S\to R$ in $\cc$ such that $Eu_S=u_R$ (equivalently $E\circ\varphi=\id_R$).     
\end{thm}

One of the main ingredients in the proof of the previous result is the following lemma.

\begin{lem}\cite[Lemma 3.3]{BT15}\label{lem:bij}
    Assume that the unit object $1$ is a left $\otimes$-generator for $\cc$ and let $\varphi:R\to S$ be an algebra morphism in $\cc$. Consider the restriction of scalars functor $\varphi_*:\cc_S\to\cc_R$  and the induction functor $\varphi^*:\cc_R\to\cc_S$, induced by $\varphi$. Then, there exists a bijection
    \[
\mathrm{Nat} (\varphi_*\varphi^*, \id_{\cc_R})\cong {}_R\mathrm{Hom}_R(S,R).
        \]
\end{lem}
\begin{invisible}
\begin{proof}
    It follows similarly to \cite[Theorem 27, p. 100]{CMZ02}. We consider \[
\alpha:\mathrm{Nat}(\varphi_*\varphi^*, \id_{\cc_R})\to {}_R\Hom_R(S,R)
    \]
 sending $t=(t_X:X\otimes_RS\to X)_{X\in\cc_R}$ in  $\mathrm{Nat}(\varphi_*\varphi^*, \id_{\cc_R})$, into $\alpha(t)=t_R\Upsilon'^{-1}_S$. In order to show that $\alpha(t)$ is right $R$-linear, we only need to show that $\Upsilon'^{-1}_S$ is a right $R$-linear morphism, as $t_R$ is in $\cc_R$. 
 In fact, we have 
\[
\begin{split}
\nu^R_{R\otimes_RS}(\Upsilon'^{-1}_S\otimes\id_R)&=\nu^R_{R\otimes_RS}(q^R_{R,S}(u_R\otimes\id_S)\otimes\id_R)=q^R_{R,S}(\id_R\otimes \nu^R_S)(u_R\otimes\id_S\otimes\id_R)\\&=q^{R}_{R,S}(u_R\otimes\id_S)\nu^R_S=\Upsilon'^{-1}_S\nu^R_S,
\end{split}
\]
so $\Upsilon'^{-1}_S$ is right $R$-linear. 

Consider an arbitrary morphism $\xi:1\to R$ in $\cc$ and define $f_\xi:R\to R$ by $f_{\xi}:=m_R(\xi\otimes\id_R)$. From associativity of $m_R$ it follows that $m_R(f_{\xi}\otimes \id_R)=m_R(m_R(\xi\otimes\id_R)\otimes\id_R)=m_R(m_R\otimes\id_R)(\xi\otimes\id_R\otimes\id_R)=m_R(\id_R\otimes m_R)(\xi\otimes \id_R\otimes\id_R)=m_R(\xi\otimes\id_R)m_R=f_{\xi}m_R$, so $f_{\xi}$ is right $R$-linear. By naturality of $t$, we get that $f_{\xi}t_R=t_R\varphi_*\varphi^*(f_{\xi})=t_R\hat{f_{\xi}}$. This is equivalent to $f_{\xi}t_R\Upsilon'^{-1}_S=t_R\hat{f_{\xi}}\Upsilon'^{-1}_S$. Since 
$t_R\hat{f_{\xi}}\Upsilon'^{-1}_S=t_R\hat{f}_\xi q^{R}_{R,S}(u_R\otimes\id_S)=t_Rq^{R}_{R,S}(f_{\xi}\otimes \id_S)(u_R\otimes \id_S)=t_Rq^{R}_{R,S}(\xi\otimes \id_S)$, we get
\[
m_R(\id_R\otimes \alpha(t))(\xi\otimes\id_S)=t_Rq^{R}_{R,S}(\xi\otimes \id_S).
\]

Moreover, we have \[\begin{split}\alpha(t)m_S(\varphi\otimes \id_S)&=t_R\Upsilon'^{-1}_Sm_S(\varphi\otimes \id_S)=t_Rq^{R}_{R,S}(u_R\otimes\id_S)m_S(\varphi\otimes \id_S)\\&=t_Rq^R_{R,S}(m_R(u_R\otimes\id_R)\otimes\id_S)=t_Rq^R_{R,S},
\end{split}
\]
where the second latter equality follows as $(q^R_{R,S}, R\otimes_RS)$ is a coequalizer, so we get that
\[
m_R(\id_R\otimes \alpha(t))(\xi\otimes\id_S)=\alpha(t)m_S(\varphi\otimes \id_S)(\xi\otimes \id_S),
\]
for all $\xi:1\to R$. Since $1$ is a left $\otimes$-generator for $\cc$, we have that $\alpha(t)$ is left $R$-linear, and thus it is an $R$-bimodule morphism in $\cc$.

We show that $\alpha$ is a bijection with inverse given by 
\[
\alpha^{-1}(E)=(\beta_X:=\Upsilon_X\tilde{E}:X\otimes_RS\to X)_{X\in\cc_R},
\]
for every $E\in{}_R\mathrm{Hom}_R(S,R)$. We have that $\beta:=(\beta_X)_{X\in\cc_R}$ is completely determined by the property $\beta_X q^R_{X,S}=\nu^R_{X}(\id_X\otimes E)$, for any $X\in\cc_R$, as $\beta_Xq^R_{X,S}=\Upsilon_X\tilde{E}q^R_{X,S}=\Upsilon_X q^R_{X,R}(\id_X\otimes E)=\nu^R_X(\id_X\otimes E)$, and on the other hand, if $\beta'_X$ fulfills $\beta'_X q^R_{X,S}=\nu^R_{X}(\id_X\otimes E)$, then from $\beta'_X q^R_{X,S}=\nu^R_{X}(\id_X\otimes E)=\Upsilon_X\tilde{E}q^R_{X,S}$, we get $\beta'_X=\Upsilon_X\tilde{E}$ as $q^R_{X,S}$ is an epimorphism. We check that $\beta$ is a natural transformation. For every $f:X\to Y$ in $\cc_R$, we have
\[
\begin{split}
\beta_Y\circ\varphi_*\varphi^*f\circ q^R_{X,S}&=\beta_Y \hat{f}q^R_{X,S}=\beta_Y q^R_{Y,S}(f\otimes \id_S)=\nu^R_Y(\id_Y\otimes E)(f\otimes \id_S)\\&=\nu^R_Y(f\otimes\id_R)(\id_X\otimes E)=f \nu^R_X(\id_X\otimes E)=f\beta_X q^R_{X,S}.
\end{split}
\]
Since $q^R_{X,S}$ is an epimorphism, we obtain $\beta_Y\circ \varphi_*\varphi^*f=f\circ\beta_X$, so $\beta$ is a natural transformation and then $\alpha^{-1}$ is well defined.
Moreover, we have $\alpha^{-1}\alpha (t)=\alpha^{-1}(t_R\Upsilon'^{-1}_S)=(\beta_X=\Upsilon_X\widetilde{t_R \Upsilon'^{-1}_S})_{X\in\cc_R}$, where $\beta_X$ is characterized by $\beta_X q^R_{X,S}=\nu^R_{X}(\id_X\otimes \alpha(t))$. 

Consider an arbitrary morphism $\xi:1\to X$ in $\cc$ and define $g_\xi:=\nu^R_X(\xi\otimes \id_R):R\to X$. Note that $g_{\xi}$ is right $R$-linear as $\nu^R_X(\nu^R_X(\xi\otimes\id_R)\otimes\id_R)=\nu^R_X(\nu^R_X\otimes \id_R)(\xi\otimes\id_R\otimes\id_R)=\nu^R_X(\id_X\otimes m_R)(\xi\otimes\id_R\otimes\id_R)=\nu^R_X(\xi\otimes\id_R)m_R=g_\xi m_R$. By naturality of $t$ we have \[
g_\xi\alpha(t)=g_\xi t_R\Upsilon'^{-1}_S=t_X\hat{g}_{\xi}\Upsilon'^{-1}_S=t_X\hat{g}_{\xi}(q^R_{R,S}(u_R\otimes\id_S))=t_Xq^R_{X,S}(g_\xi u_R\otimes\id_S)=t_Xq^R_{X,S}(\xi\otimes \id_S).
\]
Since $g_\xi \alpha(t)=\nu^R_X(\xi\otimes \id_R)\alpha(t)=\nu^R_X(\id_X\otimes\alpha(t))(\xi\otimes\id_S)$ and $1$ is a left $\otimes$-generator in $\cc$, we obtain that $\nu^R_X(\id_X\otimes\alpha(t))=t_Xq^R_{X,S}$, for all $X\in\cc$, and so $\beta_Xq^R_{X,S}=t_Xq^R_{X,S}$. Since $q^R_{X,S}$ is an epimorphism in $\cc$, we get $\beta_X=t_X$, for all $X\in\cc$, thus $\alpha^{-1}\alpha(t)=t$.

For $E\in{}_R\Hom_R(S,R)$ we have \[\begin{split}
\alpha\alpha^{-1}(E)&=\alpha\left((\beta_X=\Upsilon_X\tilde{E}:X\otimes_RS\to X)_{X\in\cc_R} \right)=\beta_R\Upsilon'^{-1}_S\\&=\beta_R q^{R}_{R,S}(u_R\otimes \id_S)=\nu^R_R(\id_R\otimes E)(u_R\otimes \id_S)=E.\qedhere
\end{split}
\]
\end{proof}
\end{invisible}

\begin{invisible}
\begin{rmk}\label{rmk:cond-coflat}
 We observe that coflatness condition for all objects $X\in\cc$ can be weakened by requiring that $-\otimes X$ (resp., $ X\otimes-$) preserves a particular coequalizer. More precisely, let $A, R$ be algebras
 in $\cc$, and $X\in\cc_R$ and $Y\in{}_R\cc_A$.
Note that
 $-\otimes A$ (resp., $A\otimes -$ ) preserves the coequalizer 
 \begin{equation}\label{eq:coeq2}
\xymatrix@C=1.5cm{X\otimes R\otimes Y\ar@<-1ex>[r]_-{\id_X\otimes \mu_Y^R}\ar@<1ex>[r]^-{\nu^R_{X}\otimes\id_Y}& X\otimes Y\ar[r]^-{q^R_{X,Y}}& X\otimes _RY}
\end{equation}
if and only if $q^R_{X,Y}\otimes \id_A=q^R_{X,Y\otimes A}$ (resp.,  $\id_A\otimes q^R_{X,Y}=q^R_{A\otimes X,Y}$). Interestingly, if we replace coflatness condition by $q^R_{X,Y}\otimes \id_A=q^R_{X,Y\otimes A}$, Lemma \ref{lem:bij} still holds. In fact, $q^R_{X,Y}\otimes \id_A=q^R_{X,Y\otimes A}$ provides the right $A$-module structure $\nu^A_{X\otimes _RY}:X\otimes_R Y\otimes A\to X\otimes_RY$, which is uniquely determined by $\nu^A_{X\otimes_R Y}(q^R_{X,Y}\otimes\id_A)=q^R_{X,Y}(\id_X\otimes \nu^A_Y)$, where $\nu^A_Y$ is the right $A$-module structure of $Y$.
\end{rmk}
\end{invisible}

    
We now characterize the semiseparability for $\varphi^*$.

\begin{thm}\label{thm:phi*semisep}
    Let $(\cc, \otimes, 1)$ be a monoidal category. Let  $\varphi:R\to S$ be an algebra morphism in $\cc$. Then, the following assertions are equivalent:
\begin{itemize}    
    \item[$(i)$] $\varphi$ is regular as an $R$-bimodule morphism in $\cc$, i.e.\ there exists an $R$-bimodule morphism $E:S\to R$ in $\cc$ such that $\varphi\circ E\circ\varphi=\varphi$;
    \item[$(ii)$] there exists an $R$-bimodule morphism $E:S\to R$ in $\cc$ such that $\varphi E u_S=u_S$. 
\end{itemize}    
If  one of the previous equivalent conditions holds, then
\begin{itemize}
    \item[$(iii)$] the induction functor $\varphi^*:\cc_R\to\cc_S$ is semiseparable. 
\end{itemize}
   Moreover, if $1$ is a left $\otimes$-generator in $\cc$, then $(iii)$ is equivalent to both $(i)$ and $(ii)$.    
\end{thm}

\begin{proof}
We show that $(i)$ is equivalent to $(ii)$. On one hand, since $\varphi u_R=u_S$, from $\varphi E\varphi=\varphi$ we get that $\varphi E u_S=\varphi E\varphi u_R=\varphi u_R=u_S$. On the other hand, assume that  $\varphi E u_S=u_S$. Then, \[\begin{split}
\varphi&=m_S(\varphi\otimes u_S)=m_S(\varphi\otimes\varphi E u_S )=\varphi m_R(\id_R\otimes E)(\id_R\otimes u_S)\\&=\varphi E\mu^R_S(\id_R\otimes u_S)=\varphi E m_S(\varphi\otimes\id_S)(\id_R\otimes u_S)=\varphi E \varphi.
\end{split}
\]
 

$(ii)\Rightarrow (iii)$. By Theorem \ref{thm:rafael}, $\varphi^*$ is semiseparable if, and only if, there exists a natural transformation $\nu \in \mathrm{Nat}(\varphi_*\varphi^*,\id_{\cc_R})$ such that $\eta\circ\nu\circ\eta = \eta$. Given $E\in {}_{R}\Hom_{R}(S,R)$ such that $\varphi E u_S= u_S$,  define for every $M\in\cc_R$,
\begin{equation}\label{eq:nuM}
\nu_M:=\Upsilon_M\tilde{E}:M\otimes_RS\to M,
\end{equation}
as in the proof of Lemma \ref{lem:bij}, where $\tilde{E}$ is the morphism such that $\tilde{E}q^R_{M,S}=q^R_{M,R}(\id_M\otimes E)$, and $\Upsilon_M$ is the isomorphism such that $\Upsilon_M q^R_{M,R}=\nu^R_M$. Then, we get 
\[
\begin{split}
\eta_M\nu_M\eta_M &=\eta_M\Upsilon_M\tilde{E}\eta_M =\eta_M\Upsilon_M\tilde{E}q^R_{M,S}(\id_M\otimes u_S)=\eta_M\Upsilon_M q^R_{M,R}(\id_M\otimes Eu_S)
\\
&= q^R_{M,S}(\id_M\otimes u_S)\nu^R_M (\id_M\otimes Eu_S)=q^R_{M,S}(\id_M\otimes u_S)(\nu^R_M\otimes \id_1)(\id_M\otimes Eu_S\otimes \id_1)\\
&=q^R_{M,S}(\nu^R_M\otimes\id_S)(\id_{M}\otimes E u_S\otimes u_S)=q^R_{M,S}(\id_M\otimes \mu^R_S)(\id_M\otimes Eu_S\otimes  u_S)\\
&=q^R_{M,S}(\id_M\otimes m_S(\varphi\otimes\id_S))(\id_M\otimes E u_S\otimes u_S)\\
&=q^R_{M,S}(\id_M\otimes m_S)(\id_M\otimes \varphi E u_S\otimes u_S)\\
&=q^R_{M,S}(\id_M\otimes m_S)(\id_M\otimes u_S\otimes u_S)=q^R_{M,S}(\id_M\otimes u_S)=\eta_M.
\end{split}
\]
Hence, $\varphi^*$ is semiseparable.

$(iii)\Rightarrow (i)$. Assume that $1$ is a left $\otimes$-generator. 
Given $\nu \in \mathrm{Nat}(\varphi_*\varphi^*,\id_{\cc_R})$ such that $\eta\nu\eta = \eta$, we consider the corresponding $E\in {}_R\Hom_{R}(S,R)$ as in Lemma \ref{lem:bij}, which is defined by $E=\nu_R\Upsilon'^{-1}_S$. Recall that $\Upsilon'_S q^R_{R,S}=\mu^R_S$. We observe that $\varphi=\Upsilon'_S\eta_R$ as $\varphi =m_S(\varphi\otimes u_S)=m_S(\varphi\otimes \id_S)(\id_R\otimes u_S)=\mu^R_S(\id_R\otimes u_S)=\Upsilon'_Sq^R_{R,S}(\id_R\otimes u_S)=\Upsilon'_S\eta_R$. Then, 
\[
\begin{split}
\varphi E\varphi=\varphi\nu_R\Upsilon'^{-1}_S\varphi=\varphi\nu_R\Upsilon'^{-1}_S\Upsilon'_S\eta_R=\varphi\nu_R\eta_R=\Upsilon'_S\eta_R\nu_R\eta_R=\Upsilon'_S\eta_R=\varphi.\qedhere
\end{split}
\]
\end{proof}

\begin{rmk}
If $\varphi$ is regular as an $R$-bimodule morphism in $\cc$ through $E:S\to R$, then the idempotent natural transformation associated with the semiseparable functor $\varphi^*$ is given on components by 
\begin{equation}\label{eq:idemp-phi*}
e_M=\nu_M\eta_M=\Upsilon_M\tilde{E}q^R_{M,S}(\id_M\otimes u_S)=\Upsilon_M q^R_{M,R}(\id_M\otimes Eu_S)=\nu^R_M (\id_M\otimes Eu_S),
\end{equation}
for every $M\in\cc_R$. 
\end{rmk}

In the next result, we characterize the natural fullness of $\varphi^*$.

\begin{prop}\label{prop:phi*natfull}
	Let $\cc$ be a monoidal category. Let  $\varphi:R\to S$ be an algebra morphism in $\cc$. Then, the following assertions are equivalent:
\begin{itemize}    
    \item[$(i)$] $\varphi$ is a split-epi as an $R$-bimodule morphism in $\cc$, i.e. there exists an $R$-bimodule morphism $E:S\to R$ in $\cc$ such that $\varphi\circ E=\id$;
    \item[$(ii)$] $\varphi$ is an epimorphism and there exists an $R$-bimodule morphism $E:S\to R$ in $\cc$ such that $\varphi E u_S=u_S$. 
\end{itemize}    
If  one of the previous equivalent conditions holds, then
\begin{itemize}
    \item[$(iii)$] the induction functor $\varphi^*$ is naturally full. 
\end{itemize}
   Moreover, if $1$ is a left $\otimes$-generator in $\cc$, then $(iii)$ is equivalent to $(i)$, $(ii)$.    
\end{prop}
\begin{proof}
$(i)\Rightarrow (ii)$. It is obvious.

$(ii)\Rightarrow (i)$. If there exists an $R$-bimodule morphism $E:S\to R$ in $\cc$ such that $\varphi E u_S=u_S$, then $\varphi E \varphi=\varphi$ by Theorem \ref{thm:phi*semisep}. Since $\varphi$ is an epimorphism, we have $\varphi\circ E=\id_S$.

$(i)\Rightarrow (iii)$.	
	By \cite[Theorem 2.6]{ACMM06} $\varphi^*$ is naturally full if, and only if, there exists a natural transformation $\nu \in \mathrm{Nat}(\varphi_*\varphi^*,\id_{\cc_R})$ such that $\eta\circ\nu = \id$. 
     Given $E\in {}_{R}\Hom_{R}(S,R)$ such that $\varphi\circ E =\id$, define $\nu_M$ as in \eqref{eq:nuM}, for every $M\in\cc_R$.
	Then, we have 
	\[
	\begin{split}	\eta_M\nu_Mq^R_{M,S}&=q^R_{M,S}(\id_M\otimes u_S)\Upsilon_M\tilde{E}q^R_{M,S}= q^R_{M,S}(\id_M\otimes u_S)\Upsilon_Mq^R_{M,R}(\id_M\otimes E)
		\\&=q^R_{M,S}(\id_M\otimes u_S)\nu^R_M (\id_M\otimes E)=q^R_{M,S}(\id_M\otimes u_S)(\nu^R_M\otimes \id_1)(\id_M\otimes E\otimes \id_1)\\&=q^R_{M,S}(\nu^R_M\otimes\id_S)(\id_{M}\otimes \id_R\otimes u_S)(\id_M\otimes E\otimes \id_1)\\&=q^R_{M,S}(\id_M\otimes \mu^R_S)(\id_M\otimes E\otimes  u_S)=q^R_{M,S}(\id_M\otimes m_S(\varphi\otimes\id_S))(\id_M\otimes E\otimes u_S)\\&=q^R_{M,S}(\id_M\otimes m_S)(\id_M\otimes \varphi E \otimes u_S)=q^R_{M,S}(\id_M\otimes m_S)(\id_M\otimes\id_S\otimes u_S)=q^R_{M,S}.
	\end{split}
	\]
Since $q^R_{M,S}$ is an epimorphism, we get $\eta_M\nu_M=\id_{M\otimes_RS}$.

    $(iii)\Rightarrow (i)$.
    Given $\nu \in \mathrm{Nat}(\varphi_*\varphi^*,\id_{\cc_R})$ such that $\eta\nu = \id$, we consider the corresponding $E\in {}_R\Hom_{R}(S,R)$, given by $E=\nu_R\Upsilon'^{-1}_S$. From the proof of Theorem \ref{thm:phi*semisep} we know that $\varphi=\Upsilon'_S\eta_R$. Then, 
	\[
	\begin{split}
		\varphi E =\Upsilon'_S\eta_R\nu_R\Upsilon'^{-1}_S=\Upsilon'_S\Upsilon'^{-1}_S=\id.\qedhere
	\end{split}
	\]
\end{proof}

\begin{cor}
    Let $\varphi:R\to S$ be an algebra morphism in a monoidal category $\cc$ with the unit being a left $\otimes$-generator. Then, the induction functor $\varphi^*$ is naturally full if, and only if, $\varphi^*$ is semiseparable and $\varphi$ is an epimorphism. Similarly, $\varphi^*$ is separable if, and only if, $\varphi^*$ is semiseparable and $\varphi$ is a monomorphism.
\end{cor}
\begin{proof}
The naturally full case follows from Proposition \ref{prop:phi*natfull} and Theorem \ref{thm:phi*semisep}. 
    The separable case follows from Theorem \ref{thm:sepmon} and Theorem \ref{thm:phi*semisep}.   
\end{proof}
\begin{rmk}\label{rmk:left-rightsymm}
Given an algebra morphism $\varphi:R\to S$ in $\cc$, one can also consider the induction functor $\varphi^*=S\otimes_R-:{}_R\cc\to{}_S\cc$ between the categories of left modules, given by  $X\mapsto S\otimes_RX$, $f\mapsto S\otimes_R f$. One can show that $(i)$ in Theorem \ref{thm:phi*semisep} implies that $\varphi^*:{}_R\cc\to {}_S\cc$ is semiseparable. If in addition $1$ is a right $\otimes$-generator, then the converse holds. Therefore, if $1$ is a two-sided $\otimes$-generator, then $S\otimes_R-:{}_R\cc\to{}_S\cc$ is semiseparable  if, and only if, so is $-\otimes_RS:\cc_R\to\cc_S$. Moreover, similar statements hold for separable and naturally full cases.
\end{rmk}
\begin{rmk}\label{rmk:notens-gen}
Let $\cc$ be a monoidal category where the unit object is not necessarily a left $\otimes$-generator. From the proof of 
%
Lemma \ref{lem:bij} and Theorem \ref{thm:sepmon}, 
it follows that, if $\varphi^*$ is separable, 
then there exists $E\in \Hom_{R}(S,R)$ in $\cc$,  such that $ E\varphi=\id$, i.e.\ $Eu_S=u_R$. 
Similarly, when $\varphi^*$ is semiseparable (resp., naturally full), there is a morphism $E\in \Hom_{R}(S,R)$ in $\cc$, such that $\varphi E\varphi=\varphi$ (resp., $\varphi E=\id$). Besides, if there exists an $R$-bimodule morphism $E:S\to R$ in $\cc$ such that $ E\varphi=\id_R$, 
then $\varphi^*$ is separable. 
\end{rmk}
We observe that the left $\otimes$-generator condition is not always necessary for $(iii)$  to imply $(i)$ in Theorem \ref{thm:phi*semisep}, as the next example shows. 
\begin{es}\label{es:Hmodules} Let $H$ be a bialgebra over a field $\Bbbk$. Consider the category $(\m_H, \otimes_\Bbbk, \Bbbk)$ of right $H$-modules. In general, the unit $\Bbbk$ is not a generator in $\m_H$, 
hence, it is not a left $\otimes$-generator. 
Let $\varphi:R\to S$ be an algebra morphism in $\m_H$. By Remark \ref{rmk:notens-gen} we know that if $\varphi^*$ is semiseparable, then $E=\nu_R\Upsilon'^{-1}_S=\nu_R q^R_{R,S}(u_R\otimes\id_S):S\to R$ 
is in $\Hom_R(S,R)$, where $\nu$ is in $\Nat(\varphi_*\varphi^*,\id_{\cc_R})$ such that $\eta\circ\nu\circ\eta=\eta$. More explicitly, for any $s$ in $S$, we have $E(s)=\nu_R(1_R\otimes_R s)$. We show that $E$ is left $R$-linear. Consider $\xi:\Bbbk\to R$, $\xi(k)=k\xi(1)=kx$, where $x:=\xi(1)\in R$,
and $f_\xi:R\to R$, $r\mapsto \xi(1)r=xr$. Note that $f_\xi\in \m_H$ as $f_\xi(rh)=xrh=f_\xi(r)h$. Then, from $f_\xi\nu_R=\nu_R(f_\xi\otimes_RS)$, we get that 
$
E(xs)=\nu_R(1_R \otimes_Rxs)=\nu_R(x\otimes_Rs)=\nu_R(f_\xi\otimes_RS)(1\otimes_Rs)=f_\xi\nu_R(1_R \otimes_Rs)=xE(s).
$
\end{es}
Let $\cc$ be a right closed monoidal category and let $\varphi : R\to S$ be a morphism of algebras in $\cc$. By adapting \cite[Corollary 3.8]{PaI} to categories of right modules, since $S$ is in ${}_S\cc_R$, it induces a functor $\varphi^!:=[S,-]_R:\cc_R\to \cc_S$. Moreover, by \cite[Proposition 3.10]{PaI}, there is a natural isomorphism $\Hom_{\cc_S}(M, [S, N]_R)\cong \Hom_{\cc_R}( M \otimes_S S,N) \cong \Hom_{\cc_R}( \varphi_*(M), N)$, for every $M\in\cc_S$, $N\in \cc_R$, so $\varphi^!$ is the right adjoint of the restriction of scalars functor $\varphi_*:\cc_S\to\cc_R$. Thus, $\varphi^*\dashv \varphi_*\dashv \varphi^!$ is an adjoint triple. By \cite[Proposition 2.19]{AB22} we have that $\varphi^!$ is semiseparable (resp., separable, naturally full) if, and only if, so is $\varphi^*$. Therefore, we have the following corollary.

\begin{cor}
Let $(\cc, \otimes, 1)$ be a right closed monoidal category. Let  $\varphi:R\to S$ be an algebra morphism in $\cc$. If $(i)$ or $(ii)$ in Theorem \ref{thm:phi*semisep} holds true, then $\varphi^!$ is semiseparable. Moreover, if $1$ is a left $\otimes$-generator in $\cc$, then the converse holds as well.  
\end{cor}

Now, we turn our attention to abelian monoidal categories. Here, an abelian monoidal category is a
monoidal category which is abelian with additive tensor functors. Recall from \cite[Corollary 1.16]{AB22} that a functor is semiseparable if and only if it factors as $S \circ N$, where $S$ is a separable functor and $N$ is a naturally full functor. It is also well-known that there is an image factorization for any morphism in an abelian category, see e.g. \cite[Theorem 1.5.5]{BorII94}. These two facts give rise to a natural question: Does the image factorization lead to a factorization of a semiseparable induction functor as the composition of a naturally full induction functor followed by a separable induction functor? The answer is positive when the category is abelian monoidal with unit being a left $\otimes$-generator and with right exact tensor functors. 

First, we show the next lemma.
 
\begin{lem}\label{lem:abel}
Let $(\cc, \otimes, 1)$ be an abelian monoidal category with right exact tensor product, and let $f:R\to S$ be an algebra morphism in $\cc$. Consider the image factorization of $f=\varphi \circ \psi$ in $\cc$.
Then, $\im(f)$ is an algebra in $\cc$ and $\psi:R\to\im(f)$ and $\varphi:\im(f)\to S$ are algebra morphisms in $\cc$. 
\end{lem}
\begin{proof}
Because the tensor product is right exact, we know that $\psi\otimes \psi$ is an epimorphism. Consider the cokernel $\pi:S\to\mathrm{Coker}(\varphi)$ of $\varphi$. Since \[\pi\circ m_S \circ (\varphi\otimes \varphi) (\psi\otimes \psi)=\pi\circ m_S \circ (f\otimes f)=\pi\circ f\circ m_R=\pi\circ\varphi\circ\psi\circ m_R=0,\] we get $\pi m_S(\varphi\otimes\varphi)=0$. Note that $\mathrm{Ker}(\pi)\cong\mathrm{Im}(\varphi)=\mathrm{Im}(f)$. Hence by the universal property of the kernel, there exists a unique $m_{\im(f)}:\im(f)\otimes \im(f)\to \im(f)$ such that $\varphi\circ m_{\im(f)}=m_S\circ (\varphi\otimes\varphi)$. 

As a result, $m_{\im(f)} \circ (m_{\im(f)} \otimes \id_{\im(f)}) = m_{\im(f)} \circ ( \id_{\im(f)} \otimes m_{\im(f)})$ follows from the associativity of $m_S$ and since $\varphi$ is a monomorphism. Indeed, 
\[
\begin{split}
\varphi m_{\im(f)}(m_{\im(f)}\otimes\id_{\im(f)})&=m_S(\varphi\otimes\varphi)(m_{\im(f)}\otimes \id_{\im(f)})=m_S(\varphi m_{\im(f)}\otimes\varphi)\\&=m_S(m_S(\varphi\otimes\varphi)\otimes\varphi)=m_S(m_S\otimes\id_S)(\varphi\otimes\varphi\otimes\varphi)\\&=m_S(\id_S\otimes m_S)(\varphi\otimes\varphi\otimes\varphi)=m_S(\varphi\otimes m_S(\varphi\otimes\varphi))\\&=m_S(\varphi\otimes\varphi m_{\im(f)})=m_S(\varphi\otimes\varphi)(\id_{\im(f)}\otimes m_{\im(f)})\\&=\varphi m_{\im(f)}(\id_{\im(f)}\otimes m_{\im(f)}).
\end{split}
\]

Define $u_{\im(f)}:= \psi \circ u_R:1 \to \im(f)$. Note that $u_S = f \circ u_R = \varphi \circ \psi \circ u_R = \varphi \circ u_{\im(f)}$. We consider $\varphi \circ m_{\im(f)} \circ (u_{\im(f)} \otimes \id_{\im(f)}) = m_S \circ (\varphi \otimes \varphi) \circ ( u_{\im(f)} \otimes \id_{\im(f)}) = m_S \circ ((\varphi \circ u_{\im(f)}) \otimes \varphi) = m_S \circ (u_S\otimes \id_S)(\id_1 \otimes \varphi) = \varphi$. Due to the fact that $\varphi$ is a monomorphism, this means $ m_{\im(f)} \circ (u_{\im(f)} \otimes \id_{\im(f)}) = \id_{\im(f)}$. Similarly, one can show that $ m_{\im(f)} \circ (\id_{\im(f)} \otimes 
u_{\im(f)}) = \id_{\im(f)}$. Therefore, $\im(f)$ is an algebra in $\cc$. Since $u_S = \varphi \circ u_{\im(f)}$, we get that $\varphi$ is an algebra morphism in $\cc$. Because $\varphi \circ m_{\im(f)} \circ (\psi \otimes \psi) = m_S \circ (\varphi \otimes \varphi) \circ (\psi \otimes \psi) =  m_S \circ (f \otimes f) = f \circ m_R = \varphi \circ \psi \circ m_R$, and $\varphi$ is a monomorphism, then $m_{\im(f)} \circ (\psi \otimes \psi) = \psi \circ m_R$. Recall that $u_{\im(f)}= \psi \circ u_R$, so we get that $\psi$ is also an algebra morphism in $\cc$.
\end{proof}

As an application of Theorem \ref{thm:phi*semisep}, we have the following for an abelian monoidal category. 

\begin{cor}\label{semisimple factorization}
Let $(\cc, \otimes, 1)$ be an abelian monoidal category with right exact tensor product such that $1$ is a left $\otimes$-generator in $\cc$, and let $f:R\to S$ be an algebra morphism in $\cc$ such that $f^*=-\otimes_R S:\cc_R\to \cc_S$ is semiseparable. Consider the image factorization of $f=\varphi \circ \psi$ in $\cc$. Then, $\psi^*=-\otimes_R \im(f): \cc_{R}\to\cc_{\im(f)}$ is naturally full and $\varphi^*= -\otimes_{\im(f)} S: \cc_{\im(f)}\to\cc_S$ is separable.   
\end{cor}

\begin{proof}
By Theorem \ref{thm:phi*semisep}, there is an $R$-bimodule morphism $E:S\to R$ in $\cc$ such that $f\circ E\circ f=f$, i.e. $\varphi \circ \psi \circ E\circ \varphi \circ \psi = \varphi \circ \psi$. 
Therefore, since $\varphi$ is a monomorphism and $\psi$ is an epimorphism, we have $\psi  E \varphi = \id_{\im(f)}$. Moreover, $\psi$ and $\varphi$ are left $R$-module morphisms because, by Lemma \ref{lem:abel}, $m_{\im(f)} (\psi \otimes \id_{\im(f)})(\id_R\otimes \psi) = \psi m_R$ and $\varphi m_{\im(f)} (\psi \otimes \id_{\im(f)})=m_S(\varphi\otimes \varphi)(\psi\otimes \id_{\im(f)})=m_S (f\otimes \id_S) (\id_R \otimes\varphi)$, respectively. Similarly, $\psi$ and $\varphi$ are right $R$-module morphisms. It follows that $E\varphi$ is an $R$-bimodule morphism.  
Since $E$ is an $R$-bimodule morphism, we have $Em_S(f\otimes \id_S) = m_R(\id_R \otimes E)$. This implies that $\psi Em_S(\varphi \psi \otimes \id_S) = \psi Em_S(f\otimes \id_S) = \psi m_R(\id_R \otimes E) = m_{\im(f)} (\psi \otimes \psi E)$. Because $\psi \otimes \id_S$ is an epimorphism, we obtain that $\psi E m_S(\varphi\otimes \id_S) = m_{\im(f)} (\id_{\im(f)} \otimes \psi E)$ i.e. $\psi E$ is a left $\im(f)$-module morphism. Similarly, one can show that it is a right $\im(f)$-module morphism. By Proposition \ref{prop:phi*natfull} and Theorem \ref{thm:sepmon}, $\psi^*:-\otimes_R \im(f)$ is naturally full and $\varphi^*= -\otimes_{\im(f)} S$ is separable, respectively.
\end{proof}

\subsection{Examples}\label{subsect:examples}

We now apply Theorem \ref{thm:phi*semisep} and Proposition \ref{prop:phi*natfull} to some examples.

\begin{es}\label{es1}
    Let $K$ be a commutative unital ring and denote by ${}_K\!\m$ the category of left $K$-modules and left $K$-linear maps. It is known that ${}_K\!\m$ is monoidal with tensor product $\otimes =\otimes_K$, unit $1=K$. 
    Note that 
the tensor functors are right exact, see e.g. \cite[Proposition 19.13]{AF92}, so they preserve coequalizers, i.e.\ every object in ${}_K\m$ is coflat. As observed in \cite[Example 3.2 2)]{BT15}, $K$ is a left $\otimes$-generator in ${}_K\!\m$. In fact, any morphism $\varepsilon: K\to M$ in ${}_K\!\m$ is completely determined by $\varepsilon(1) = m\in M$.
\begin{invisible}
Indeed, suppose $f,g:M\otimes Z\to W$ are two morphisms in ${}_K\!\m$ such that $f(\varepsilon\otimes\id_Z)=g(\varepsilon\otimes\id_Z)$, for all $\varepsilon:K\to M$ in ${}_K\!\m$. We can write $\varepsilon(k) = km$ where $m$ is in $M$. For any $m$ in $M$ and $z$ in $Z$, we have $f(km\otimes z)=g(km\otimes z)$. By taking $k=1_K$, we get $f(m\otimes z)=g(m\otimes z)$, which means $f=g$. 
\end{invisible}
Thus, Theorem \ref{thm:phi*semisep} and Proposition \ref{prop:phi*natfull} apply.

In particular, for $K=\mathbb{Z}$ (the ring of integers), we get the monoidal category $(\mathsf{Ab}, \otimes_\mathbb{Z}, \mathbb{Z})$ of abelian groups. In this case, the algebras in $\mathsf{Ab}$ are  unital rings, so Theorem \ref{thm:phi*semisep} retrieves Proposition \ref{prop:inducfunc}, while Proposition \ref{prop:phi*natfull} recovers \cite[Proposition 3.1 2)]{ACMM06}. 
\end{es}


\begin{es}
The category $(\Set, \times, \{0\})$ of sets and functions is a cartesian monoidal category where the unit object is a one-element set $\{0\}$. The unit is a left $\otimes$-generator, see \cite[Example 3.2 1)]{BT15}. An algebra in $\Set$ is a monoid. Since $\Set$ is cartesian closed, see e.g. \cite[p. 98]{Mac98}, $-\times B$ is a left adjoint, hence it preserves coequalizers. 

\begin{invisible}In fact, given monoids $R$, $A$, for every $X\in \Set_R$, and $Y\in {}_R\Set_A$, we have $q^R_{X,Y}\otimes\id_A=q^R_{X,Y\otimes A}$ 
since $\nu^R_X\times\id_Y=\id_X\times\mu^R_Y$ if and only if $\nu^R_X\times\id_Y\times\id_A=\id_X\times\mu^R_Y\times\id_A$.
\end{invisible}

Consider the morphisms of monoids (with respect the usual products) $$\psi:\mathbb{N}\times\mathbb{Z}\to \mathbb{N},\quad (n,z)\mapsto n,\quad \text{and} \quad \varphi:\mathbb{N}\to \mathrm{M}_2(\mathbb{N}),\quad  n\mapsto n\mathrm{I}_2,$$  where $\mathrm{I}_2$ is the identity matrix of order 2.
We show that $\varphi^*\circ\psi^*\cong (\varphi\circ \psi)^*$ is semiseparable. One can check that $D:\mathbb{N}\to \mathbb{N}\times\mathbb{Z}$, $m\mapsto (m,0)$ is an $(\mathbb{N}\times\mathbb{Z})$-bimodule morphism in $\Set$. 
\begin{invisible}
Indeed, $D((n,z) m)=D(\psi(n,z)m)=D(nm)=(nm,0)=(n,z)(m,0)=(n,z)D(m)$ and $D(m(n,z))=D(m\psi(n,z))=D(mn)=(mn,0)=(m,0)(n,z)=D(m)(n,z)$. 
\end{invisible}
Moreover, $\psi D(n)=\psi(n,0)=n$, for every $n\in \mathbb{N}$, so $\psi^*$ is naturally full by Proposition \ref{prop:phi*natfull}. Consider $E:\mathrm{M}_2(\mathbb{N})\to \mathbb{N}$, 
$(a_{ij})\mapsto a_{11}$. 
Then, $E\varphi(n)=E(n\mathrm{I}_2)=n$ and $E$ is an $\mathbb{N}$-bimodule morphism in $\Set$. In fact, $E(n(a_{ij}))=E(\varphi(n)(a_{ij}))=E((na_{ij})) =na_{11}=n E((a_{ij}))$ and similarly for the other side. 
Hence, by Theorem \ref{thm:sepmon} $\varphi^*$ is separable. 

Then, by Lemma \ref{lem:comp-semisep}, $ (\varphi\circ \psi)^*\cong \varphi^*\circ\psi^*$ is semiseparable. We observe that $\varphi\psi$ is neither injective nor surjective, so $\varphi^*\circ\psi^*$ is neither separable nor naturally full. By Theorem \ref{thm:phi*semisep} there exists an $(\mathbb{N}\times\mathbb{Z})$-bimodule morphism $\xi:\mathrm{M}_2(\mathbb{N})\to \mathbb{N}\times\mathbb{Z}$ in $\cc$ such that $\varphi\psi\xi \varphi\psi=\varphi\psi$. In particular, $\xi$ can be chosen as $D\circ E$, $(a_{ij}) \mapsto (a_{11},0)$.

\begin{invisible}
We observe that, for every $X\in \Set_{\mathbb{N}\times\mathbb{Z}}$, $X\otimes_{\mathbb{N}\times\mathbb{Z}}\mathrm{M}_2(\mathbb{N})=\dfrac{X\times M_{2}(\mathbb{N})}{\tilde{R}}$, where $\tilde{R}$ is the least equivalence relation containing 
\[
\begin{split}
R=&\lbrace \left((\nu^{\mathbb{N}\times\mathbb{Z}}_X\times \id)(x, (n,z), (a_{ij})), (\id\times\mu^{\mathbb{N}\times \mathbb{Z}}_{\mathrm{M}_2(\mathbb{N})})(x, (n,z), (a_{ij})
)\right)\,\vert\\& \,  (x, (n,z),(a_{ij}))\in X\times (\mathbb{N}\times\mathbb{Z})\times\mathrm{M}_{2}(\mathbb{N})\rbrace
\\=&\left\lbrace \left((\nu^{\mathbb{N}\times\mathbb{Z}}_X(x,(n,z)), (a_{ij})), (x,(na_{ij}) )\right)\vert\, (x, (n,z), (a_{ij}))\in X\times (\mathbb{N}\times\mathbb{Z})\times\mathrm{M}_{2}(\mathbb{N})
\right\rbrace ,
\end{split}
\]
see \cite[page 65]{Mac98}. The elements of $X\otimes_{\mathbb{N}\times\mathbb{Z}}\mathrm{M}_2(\mathbb{N})$ are denoted by $\overline{(x,(a_{ij}) )}$.

The unit $\eta$ of the adjunction $(\varphi\psi)^*\dashv (\varphi\psi)_*$ is given for every $X\in \Set_{\mathbb{N}\times\mathbb{Z}}$ by $\eta_X:X\to X\otimes_{\mathbb{N}\times\mathbb{Z}}\mathrm{M}_2(\mathbb{N})$
\[
\eta_X=q^{\mathbb{N}\times\mathbb{Z}}_{X,\mathrm{M}_2(\mathbb{N})}(\id_{X}\times u_{\mathrm{M}_2(\mathbb{N})}), \quad (x,*)\mapsto \overline{(x, \mathrm{I}_2).}
\]
By Theorem \ref{thm:phi*semisep}, we have $\eta$ is regular through $\nu:(\varphi\psi)_*(\varphi\psi)^* \to \id$, given for every $X\in \Set_{\mathbb{N}\times\mathbb{Z}}$ by 
\[
\nu_X=\Upsilon_X\widetilde{(DE)}: X\otimes_{\mathbb{N}\times\mathbb{Z}} \mathrm{M}_2(\mathbb{N})\to X, 
\]
where $\Upsilon_X:X\otimes_{\mathbb{N}\times\mathbb{Z}} \mathbb{N}\times\mathbb{Z} \to X$ is the isomorphism uniquely determined by the property $\Upsilon_Xq^{\mathbb{N}\times\mathbb{Z}}_{X,\mathbb{N}\times\mathbb{Z}}=\nu^{\mathbb{N}\times\mathbb{Z}}_X$, and $\widetilde{DE}=\tilde{D}\tilde{E}$ is the unique morphism in $\Set$ satisfying the equation 
\[
\widetilde{(DE)}q^{\mathbb{N}\times\mathbb{Z}}_{X,\mathrm{M}_2(\mathbb{N})}=q^{\mathbb{N}\times\mathbb{Z}}_{X, \mathbb{N}\times\mathbb{Z}}(\id_X\times DE).
\]
We have $\widetilde{DE}: X\otimes_{\mathbb{N}\times \mathbb{Z}}\mathrm{M}_2(\mathbb{N})\to  X\otimes_{\mathbb{N}\times \mathbb{Z}} (\mathbb{N}\times \mathbb{Z})$, $\overline{(x,(a_{ij}))}\mapsto\overline{(x,(a_{11},0))}$. Then, $\nu_X(\overline{(x, (a_{ij})}))=\nu^{\mathbb{N}\times\mathbb{Z}}_X(x,(a_{11},0))$.

By the proof of \cite[Theorem 2.1]{AB22}, the idempotent natural transformation associated to the semiseparable left adjoint is given by $e=\nu\circ\eta$. Explicitly, for every $X\in \Set_{\mathbb{N}\times\mathbb{Z}}$, we have 
$e_X(x)=(\nu_X\circ\eta_X)(x)=\nu_X^{\mathbb{N}\times\mathbb{Z}}(x,(1,0))$. By \cite[Proposition 1.4]{AB22} the following universal property holds:
\[
f\otimes_{\mathbb{N}\times\mathbb{Z}}\mathrm{M}_2(\mathbb{N})=g\otimes_{\mathbb{N}\times\mathbb{Z}}\mathrm{M}_2(\mathbb{N})\Leftrightarrow 
\nu_Y^{\mathbb{N}\times\mathbb{Z}}(f(x), (1,0))=\nu^{\mathbb{N}\times \mathbb{Z}}_Y(g(x), (1,0)),
\]
for every $f,g:X\to Y$ in $\Set_{\mathbb{N}\times\mathbb{Z}}$.
\end{invisible}
\end{es}

Now we consider some monoidal categories where the unit object may not be a left $\otimes$-generator.

\begin{es}\label{es:bialgmod}
 Cf. Example \ref{es:Hmodules}. Let $H$ be a bialgebra over a field $\Bbbk$. We denote the tensor product $\otimes_\Bbbk$ over $\Bbbk$ by the unadorned $\otimes$.  The category $({}_H\!\m, \otimes, \Bbbk)$ of left modules over $H$ is monoidal, and the unit object $\Bbbk\in{}_H\!\m$ via the \emph{trivial action} $hk=\varepsilon (h)k$. For $X,Y\in{}_H\!\m$, $X\otimes Y\in{}_H\!\m$ via the \emph{diagonal action}
 \[
H\otimes (X\otimes Y)\to X\otimes Y, \,\, h\otimes x\otimes y\mapsto \sum h_{1}x\otimes h_2y. 
\]
Every object in ${}_H\m$ is coflat, because tensor functors are right exact. 
\begin{invisible}
We show that $\Bbbk$ is not in general a generator, hence not a left $\otimes$-generator. Let $H$ be a bialgebra. Note that a left $H$-module morphism $\xi:\Bbbk \to H$, $\xi(k)=k\xi(1)$, is completely determined by $t:=\xi(1)$, which is a left integral in $H$ as $ht=h\xi(1)=\xi(h1)=\xi(\varepsilon(h)1)=\varepsilon(h)\xi(1)=\varepsilon (h)t$. 

Consider now the Hopf algebra $\Bbbk G $, for a finite group $G$, and the left $\Bbbk G$-module morphisms $f_g: \Bbbk G\to \Bbbk G$, $x\mapsto xg$, and  $ f_{g^{-1}}: \Bbbk G\to \Bbbk G$, $x\mapsto xg^{-1}$, for $g\in G$ such that $g^2\neq 1$. Since a left $\Bbbk G$-module morphism $\xi:\Bbbk\to \Bbbk G$ is completely determined by $t:=\xi(1)=k\sum_{g\in G} g$, $k\in\Bbbk$, we get $f_g(t)=tg=t=tg^{-1}=f_{g^{-1}}(t)$, but $f_g\neq f_{g^{-1}}$.
\end{invisible}
\begin{invisible}
Since a morphism $H\to M$ in ${}_H\m$ is completely determined by an element $m\in M$, it follows that $H$ is a left $\otimes$-generator in ${}_H\!\m$. Indeed, consider two morphisms $f,g:M\otimes Z\to W$ in ${}_H\!\m$ such that $f(\varepsilon\otimes\id_Z)=g(\varepsilon\otimes\id_Z)$, for all $\varepsilon:H\to M$ in ${}_H\!\m$, so for all $h\in H$, we have $
hf(m\otimes z)=f(hm\otimes z)=g(hm\otimes z)=hg(m\otimes z),$
hence taking $h=1_H$, we get $f(m\otimes z)=g(m\otimes z)$, and then $f=g$. 
\end{invisible}
An algebra $A$ in ${}_H\!\m$ is a $\Bbbk$-algebra $A$ where the multiplication and the unit are left $H$-linear morphisms, i.e. a left $H$-module algebra. More precisely, a $\Bbbk$-algebra with a left $H$-module structure such that 
$$h\cdot (ab)=(h_1\cdot a)(h_2\cdot b),\quad\text{ and}\quad h\cdot 1=\varepsilon_H(h)1,$$ for all $h\in H$, $a,b\in A$.

Let $H$ be a Hopf algebra over a field $\Bbbk$ with antipode $S$, unit $u:\Bbbk \to H$ and counit $\varepsilon:H\to \Bbbk$. We recall that $H$ is a left $H$-module algebra with respect to the \emph{left adjoint action} $H\otimes H\to H$, $h\otimes a\mapsto h\triangleright_{\mathrm{Ad}} a:=h_1aS(h_2)$, see \cite[Example 1.6.3]{Majid-book}.  
Note that $\Bbbk$ is an $H$-module algebra via the trivial action. 
Moreover, $u$ is an algebra morphism in ${}_H\!\m$. 
Indeed, $u$ is left $H$-linear as 
$u(hk)=u(\varepsilon(h)k)=\varepsilon (h) k1_H= h_1k1_HS(h_2) = h_1 u(k)S(h_2)=h\triangleright_{\mathrm{Ad}} u(k)$ for every $h\in H$, $k\in \Bbbk$. Moreover, $\varepsilon$ 
is left $H$-linear as, for every $h,a\in H$, 
\begin{equation*}
\begin{split}
\varepsilon(h\triangleright_{\mathrm{Ad}} a)&=\varepsilon(h_1aS(h_2))=\varepsilon(h_1)\varepsilon(a)\varepsilon(S(h_2))
=\varepsilon(\varepsilon(h)1_H)\varepsilon(a)=\varepsilon(h)\varepsilon(a)=h\varepsilon(a).\end{split}\end{equation*}

Note that $\varepsilon$ is in ${}_\Bbbk\Hom_\Bbbk(H,\Bbbk)$ as $\varepsilon(k\cdot h)=\varepsilon(u(k)h)=\varepsilon(k1h)=\varepsilon(kh)=k\varepsilon(h)$ and $\varepsilon(h\cdot k)=\varepsilon(hu(k))=\varepsilon(hk1)=\varepsilon(h)k$.
Since for every $k\in \Bbbk$ $(\varepsilon\circ u)(k)=\varepsilon(k1_H)=k\varepsilon(1_H)=k1_\Bbbk=k$, $u$ is a split-mono. By Remark \ref{rmk:notens-gen}, we have that $u^*=-\otimes H:\vect_\Bbbk\to\m_H$ is separable.



Let $H'$ be an $H$-module algebra.
We now consider the projection $\psi:\Bbbk \oplus H'\to \Bbbk $, $(k,h)\mapsto k$. We check that $\Bbbk\oplus H'$ is an $H$-module algebra where the $H$-module structure is componentwise. In fact,
\[
\begin{split}
h\cdot ((k,a)(k',b))&=h\cdot (kk',ab)=(\varepsilon(h)kk',(h_1\cdot a)(h_2\cdot b ))\\&=
(\varepsilon(h_1)k, h_1\cdot a)(\varepsilon(h_2)k',h_2\cdot b)
=
(h_1\cdot (k,a))(h_2\cdot (k',b))
\end{split}\]
and 
$h\cdot (1, 1)=(\varepsilon(h)1_\Bbbk, \varepsilon(h)1_H)=\varepsilon(h)(1,1),$ for all $h\in H$, $k,k'\in \Bbbk$ $a,b\in H'$. Note that $\psi$ is an algebra morphism in ${}_H\m$. Consider $D:\Bbbk\to \Bbbk\oplus H'$, $k\mapsto (k,0)$, which is a morphism in ${}_H\m$. We have
\[
\begin{split}
    D((k,h)\cdot k')&=D(\psi((k,h)k')=D(kk')=(kk',0)=(k,h)(k',0)=(k,h)D(k'),
\\D(k'\cdot(k,h))&=D(k'\psi(k,h))=D(k'k)=(k'k,0)=(k',0)(k,h)=D(k')(k,h),
\end{split}
\]
so $D\in {}_{\Bbbk\oplus H'}\Hom_{\Bbbk\oplus H'}(\Bbbk, \Bbbk\oplus H')$ and $\psi D(k)=\psi((k,0))=k$. Thus, $\psi^*$ is naturally full by Proposition \ref{prop:phi*natfull}. Hence, the composite $u^*\circ \psi^*\cong (u\circ\psi)^*:({}_H\m)_{\Bbbk\oplus H'}\to ({}_H\m)_H$ is a semiseparable functor, which is neither separable nor naturally full.
\begin{invisible}
The unit $\eta$ of the adjunction $(u_H\psi)^*\dashv (u_H\psi)_*$ is given for every $X\in ({}_H\m)_{\Bbbk \oplus H'}$ by 
\[
\eta_X:X\to X\otimes_{\Bbbk \oplus H'} H, \quad \eta_X=q^{\Bbbk \oplus H'}_{X,H}(\id_X\otimes u), \quad x\mapsto x\otimes_{\Bbbk \oplus H'} 1_H.
\]
 By Theorem \ref{thm:rafael}, there exists a natural transformation $\nu : (u\psi)_*(u\psi)^*\to\id$ such that $\eta\nu\eta=\eta$. In particular, we have that $\nu_X=\Upsilon_X\tilde{E}:X\otimes_{\Bbbk \oplus H'}H\to X$, where $E=D\varepsilon_H$ and $\tilde{E}:X\otimes_{\Bbbk \oplus H'}H\to X\otimes_{\Bbbk \oplus H'}(\Bbbk \oplus H')$ is the unique morphism in ${}_{H}\m$ such that $\tilde{E} q^{\Bbbk \oplus H'}_{X,H}=q^{\Bbbk \oplus H'}_{X,\Bbbk \oplus H'}(\id_X\otimes E)$.
 \end{invisible}
 \begin{invisible}
Note that $E=D\varepsilon_H$ is a left $\Bbbk\oplus H'$-linear map. In fact, $D\varepsilon_H((k,h)\cdot h')=D\varepsilon(u\psi((k,h))h')=D(\varepsilon(u\psi((k,h)) \varepsilon(h')))=D(k\varepsilon(h'))=(k\varepsilon(h'),0)=(k,h)(\varepsilon(h'),0)=(k,h)D\varepsilon(h')$.
 \end{invisible}
 \begin{invisible}
 We have $\tilde{E}(x\otimes_{\Bbbk \oplus H'}h)=x\otimes_{\Bbbk \oplus H'}D\varepsilon(h)=x\otimes_{\Bbbk \oplus H'}(\varepsilon(h),0)$. Then, 
 \[
 \nu_X(x\otimes_{\Bbbk \oplus H'} h)=\nu_X^{\Bbbk \oplus H'}(x\otimes(\varepsilon(h),0)).
 \]
The idempotent natural transformation associated to the semiseparable left adjoint is given by $e=\nu\circ\eta$. Explicitly, for every $X\in ({}_H\m)_{\Bbbk \oplus H'}$ 
\[
e_X(x)=(\nu_X\circ\eta_X)(x)=\nu_X^{\Bbbk \oplus H'}(x\otimes ( 1,0)),
\]
for every $x\in X$.

We observe the following fact. Let $\varphi:R\to S$ be an algebra morphism in ${}_H\m$. By Remark \ref{rmk:notens-gen} we know that if $\varphi^*$ is semiseparable, then $E=\nu_R\Upsilon'^{-1}_S=\nu_R(1\otimes_R-):S\to R$
 is in $\Hom_R(S,R)$, where $\nu\in\Nat(\varphi_*\varphi^*,\id_{\cc_R})$ such that $\eta\circ\nu\circ\eta=\eta$. Even though in ${}_H\m$ the unit $\Bbbk$ is not a tensor generator in general, we show that $E$ is left $R$-linear as well. 
 Consider $\xi:\Bbbk\to R$, $\xi(k)=k\xi(1)=kx$, where $x:=\xi(1)\in R$,
and $f_\xi:R\to R$, $r\mapsto r\xi(1)=rx$. Note that $f_\xi\in {}_H\m$ as $f_\xi(hr)=hrx=hf_\xi(r)$. Then, from $f_\xi\nu_R=\nu_R(f_\xi\otimes_RS)$ we do not get $E(xs)=xE(s)$.
\end{invisible}
\end{es}

\begin{invisible}
\begin{es}
Let $K$ be a commutative unital ring and let $R$ be a $K$-algebra. We know that the category ${}_R\m_R$ of $R$-bimodules and $R$-bilinear morphisms is monoidal with $\otimes=\otimes_R$ and unit object $R$. Since the tensor functor $\otimes_{R}$ is right exact, 
the coflatness condition is fulfilled.

By \cite[Example 3.2 3)]{BT15}, $R$ is a left $\otimes$-generator if and only if $R$ is a generator in ${}_R\m_R$. In particular, if $R$ is a separable $K$-algebra, then $R$ is a generator in ${}_R\m_R$.  By Maschke Theorem for Hopf algebras \cite{LSw}, if $H$ is a semisimple Hopf algebra, then it is also separable, so $H$ is a left $\otimes$-generator of $({}_H\m_H, \otimes_H, H)$.  For instance, let $G$ be a finite group of cardinality $n\in \mathbb{N}, n \geq 1$, and consider the group algebra $\Bbbk G$ over a field $\Bbbk$ such that $\mathrm{char}(\Bbbk)\nmid n$. Then, $\Bbbk G$ is separable. Hence, the unit object $\Bbbk G$ is a left $\otimes$-generator in $({}_{\Bbbk G}\m_{\Bbbk G}, \otimes_{\Bbbk G}, \Bbbk G)$. We denote $\otimes:=\otimes_\Bbbk$.

Consider
$$\psi:\Bbbk G\oplus \Bbbk G\to\Bbbk G,\,\, (a,b)\mapsto a,\quad \varphi: \Bbbk G\to \Bbbk G\otimes \Bbbk G,\,\,\sum_{i=0}^{n-1} k_ig_i\mapsto \sum_{i=0}^{n-1}k_ig_i\otimes g_i.$$ Note that $\Bbbk G\oplus \Bbbk G$
is a monoid in ${}_{\Bbbk G}\m_{\Bbbk G}$, with unit
$u_{\Bbbk G\oplus \Bbbk G}:\Bbbk G\to \Bbbk G\oplus \Bbbk G$, $g\mapsto (g,g)$, and multiplication $$m_{\Bbbk G\oplus \Bbbk G}:(\Bbbk G\oplus \Bbbk G)\otimes_{\Bbbk G}(\Bbbk G\oplus \Bbbk G)\to \Bbbk G\oplus \Bbbk G,\,\,(a,b)\otimes_{\Bbbk G} (c,d)\mapsto (ac, bd),$$ where the left $\Bbbk G$-module structure of $\Bbbk G\oplus \Bbbk G$ is given by $g\cdot (a, b)=(ga,gb)$, for $a,b\in G$, while the $\Bbbk G$-module structure of $(\Bbbk G\oplus \Bbbk G)\otimes_{\Bbbk G}(\Bbbk G\oplus \Bbbk G)$ is given by 
\[
g\cdot ((a,b)\otimes_{\Bbbk G} (c,d))=(g\cdot (a,b) )\otimes_{\Bbbk G}(c,d)=(ga, gb)
\otimes_{\Bbbk G} (c,d).
\]
We check that $m_{\Bbbk G\oplus \Bbbk G}$ is left $\Bbbk G$-linear.
\[
\begin{split}
m_{\Bbbk G\oplus \Bbbk G}(g\cdot ((a,b)\otimes_{\Bbbk G} (c,d)))&=m_{\Bbbk G\oplus \Bbbk G}((g\cdot (a,b) )\otimes_{\Bbbk G}(c,d))=m_{\Bbbk G\oplus \Bbbk G}((ga, gb)
\otimes_{\Bbbk G} (c,d))\\&=(gac, gbd)=g\cdot (ac,bd)=g\cdot m_{\Bbbk G\oplus \Bbbk G}((a,b)\otimes_{\Bbbk G} (c,d)).
\end{split}
\]
Similarly, one can show that $m_{\Bbbk G\oplus \Bbbk G}$ is right $\Bbbk G$-linear. 
It is clear that $m_{\Bbbk G\oplus \Bbbk G}$ and $u_{\Bbbk G\oplus \Bbbk G}$ satisfy the associativity and unitality axioms. 

Observe that $\Bbbk G\otimes \Bbbk G$ is a monoid in ${}_{\Bbbk G}\m_{\Bbbk G}$ with unit $u_{\Bbbk G\otimes \Bbbk G}:\Bbbk G\to \Bbbk G\otimes\Bbbk G$, $g\mapsto g\otimes g$, and multiplication $$m_{\Bbbk G\otimes \Bbbk G}:(\Bbbk G\otimes \Bbbk G)\otimes_{\Bbbk G}(\Bbbk G\otimes\Bbbk G)\to \Bbbk G\otimes \Bbbk G,\,\, (a\otimes b)\otimes_{\Bbbk G} (c\otimes d)\mapsto ac\otimes bd,$$ where the $\Bbbk G$-module structure of $\Bbbk G\otimes \Bbbk G$ is given by $g\cdot (a\otimes b)=ga \otimes gb$, for $a,b\in G$, while the $\Bbbk G$-module structure of $(\Bbbk G\otimes \Bbbk G)\otimes_{\Bbbk G}(\Bbbk G\otimes \Bbbk G)$ is given by 
\[
g\cdot ((a\otimes b)\otimes_{\Bbbk G} (c\otimes d))=(g\cdot (a\otimes b) )\otimes_{\Bbbk G}(c\otimes d)=(ga\otimes gb)
\otimes_{\Bbbk G} (c\otimes d).
\]
We have 
\[
\begin{split}
m_{\Bbbk G\otimes \Bbbk G}(g\cdot ((a 
\otimes b)\otimes_{\Bbbk G} (c\otimes d)))&=m_{\Bbbk G\otimes \Bbbk G}((g\cdot (a\otimes b) )\otimes_{\Bbbk G}(c\otimes d))\\&=m_{\Bbbk G\otimes \Bbbk G}((ga\otimes gb)
\otimes_{\Bbbk G} (c \otimes d))=gac\otimes gbd\\&=g\cdot (ac\otimes bd)=g\cdot m_{\Bbbk G\otimes\Bbbk G}((a\otimes b)\otimes_{\Bbbk G} (c\otimes d)).
\end{split}
\]
Similarly, one can show that $m_{\Bbbk G\otimes \Bbbk G}$ is right $\Bbbk G$-linear.
One can check that $\varphi$, $\psi$ are algebra morphisms in ${}_{\Bbbk G}\m_{\Bbbk G}$. 
Indeed, for $t\in G$, $\varphi(t\sum_{i}k_ig_i)=\varphi(\sum_ik_itg_i)=\sum_ik_itg_i\otimes tg_i=t(\sum_ik_ig_i\otimes g_i)=t\varphi(\sum_ik_i g_i)$ and $\varphi((\sum_ik_ig_i) t)=\varphi(\sum_ik_ig_it)=\sum_{i}k_ig_it\otimes g_it=(\sum_ik_ig_i\otimes g_i)t=\varphi(\sum_ik_ig_i)t$, $\psi(g(a,b))=\psi((ga,gb))=ga=g\psi((a,b))$ and $\psi((a,b)g)=\psi((ag,bg))=ag=\psi((a,b))g$. Moreover, we have that $\varphi m_{\Bbbk G}(a\otimes c)=\varphi(ac)=ac\otimes ac=m_{\Bbbk G\otimes \Bbbk G}(\varphi\otimes \varphi)(a\otimes c)$, $\varphi u_{\Bbbk G}(g)=g\otimes g=u_{\Bbbk G\otimes \Bbbk G}(g)$, and 
$\psi m_{\Bbbk G\oplus \Bbbk G}((a,b)\otimes_{\Bbbk G}(c,d))=\psi((ac,bd))=ac=m_{\Bbbk G}(a\otimes_{\Bbbk G} c)=m_{\Bbbk G}(\psi\otimes_{\Bbbk G}\psi)((a,b)\otimes_{\Bbbk G}(c,d))$, $\psi u_{\Bbbk G\oplus \Bbbk G}(g)=\psi ((g,g))=g=u_{\Bbbk G}(g)$.  
Now define $E:=\varepsilon\otimes\id: \Bbbk G\otimes \Bbbk G\to \Bbbk G$, $E(g\otimes g') = g'$, where $g, g'\in G$. Because, for $t\in \Bbbk G$, $E(t(g\otimes g'))=E(tg\otimes tg')=tg'=tE(g\otimes g')$ and $E((g\otimes g')t)=E(gt\otimes  g't)=g't=E(g\otimes g')t$, we obtain that $E$ is in ${}_{\Bbbk G}\Hom_{\Bbbk G}(\Bbbk G\otimes \Bbbk G, \Bbbk G)$. Moreover,  $E\varphi(\sum_ik_ig_i)=E(\sum_ik_ig_i\otimes g_i)=\sum_ik_ig_i$, hence $\varphi^*$ is separable by Theorem \ref{thm:sepmon}. Note that $E$ can be also chosen to be $\id\otimes  \varepsilon$.

Next consider $D:\Bbbk G\to \Bbbk G\oplus \Bbbk G$, $g\mapsto (g,0)$, which is $\Bbbk G$-bilinear.
We have 
\[
\begin{split}
D((a,b)g)&=D(\psi((a,b))g)=D(ag)=(ag,0)=(a,b)(g,0)=(a,b)D(g),\\
D(g(a,b))&=D(g\psi((a,b)))=D(ga)=(ga, 0)=(g,0)(a,b)=D(g)(a,b),
\end{split}\]
so $D\in{}_{\Bbbk G\oplus\Bbbk G}\Hom_{\Bbbk G\oplus\Bbbk G}(\Bbbk G, \Bbbk G\oplus\Bbbk G)$. Moreover, $\psi D(g)=\psi((g,0))=g$, hence $\psi^*$ is naturally full by Proposition \ref{prop:phi*natfull}. Then, by Lemma \ref{lem:comp-semisep} $\varphi^*\circ\psi^*\cong (\varphi\circ \psi)^*$ is semiseparable. We observe that $\varphi\psi$ is neither injective nor surjective, so $\varphi^*\circ\psi^*$ is neither separable nor naturally full.

By Theorem \ref{thm:phi*semisep} there exists a $\Bbbk G\oplus\Bbbk G$-bimodule morphism $\xi:\Bbbk G\otimes\Bbbk G\to \Bbbk G\oplus\Bbbk G$ in ${}_{\Bbbk G}\m_{\Bbbk G}$ such that $\varphi\psi\xi \varphi\psi=\varphi\psi$. For instance, $\xi$ can be chosen as $D E: \Bbbk G\otimes\Bbbk G\to \Bbbk G\oplus\Bbbk G$, $g\otimes g'\mapsto (g',0)$.
The unit $\eta$ of the adjunction $(\varphi\psi)^*\dashv (\varphi\psi)_*$ is given for every $X\in ({}_{\Bbbk G}\m_{\Bbbk G})_{\Bbbk G\oplus\Bbbk G}$ by 
\[
\eta_X:X\to X\otimes_{\Bbbk G\oplus\Bbbk G}(\Bbbk G\otimes\Bbbk G),\quad\eta_X=q^{\Bbbk G\oplus\Bbbk G}_{X,\Bbbk G\otimes\Bbbk G}(\id_{X}\otimes_{\Bbbk G} u_{\Bbbk G\otimes\Bbbk G}), \quad x\mapsto x\otimes_{\Bbbk G\oplus\Bbbk G} 1\otimes 1.
\]
By Theorem \ref{thm:phi*semisep}, we have $\eta$ is regular through $\nu:(\varphi\psi)_*(\varphi\psi)^* \to \id$, given for every $X\in ({}_{\Bbbk G}\m_{\Bbbk G})_{\Bbbk G\oplus\Bbbk G}$ by 
\[
\nu_X=\Upsilon_X\widetilde{(DE)}: X\otimes_{\Bbbk G\oplus\Bbbk G} (\Bbbk G\otimes\Bbbk G)\to X, 
\]
where $\Upsilon_X:X\otimes_{\Bbbk G\oplus\Bbbk G} (\Bbbk G\oplus\Bbbk G) \to X$ is the isomorphism uniquely determined by the property $\Upsilon_Xq^{\Bbbk G\oplus\Bbbk G}_{X,\Bbbk G\oplus\Bbbk G}=\nu^{\Bbbk G\oplus\Bbbk G}_X$, and $\widetilde{DE}=\tilde{D}\tilde{E}$ is the unique morphism in ${}_{\Bbbk G}\m_{\Bbbk G}$ satisfying the equation 
\[
\widetilde{(DE)}q^{\Bbbk G\oplus\Bbbk G}_{X,\Bbbk G\otimes\Bbbk G}=q^{\Bbbk G\oplus\Bbbk G}_{X, \Bbbk G\oplus\Bbbk G}(\id_X\otimes_{\Bbbk G} DE).
\]
We have $\widetilde{DE}: X\otimes_{\Bbbk G\oplus\Bbbk G}(\Bbbk G\otimes \Bbbk G)\to  X\otimes_{\Bbbk G\oplus\Bbbk G} (\Bbbk G\oplus\Bbbk G)$, $x\otimes_{\Bbbk G\oplus\Bbbk G}g\otimes_{\Bbbk} g'\mapsto x\otimes_{\Bbbk G\oplus\Bbbk G}(g', 0)$. 

Then, $$\nu_X(x\otimes_{\Bbbk G\oplus\Bbbk G}(g\otimes g'))=\nu^{\Bbbk G\oplus\Bbbk G}_X(x\otimes_{\Bbbk G}(g',0)).$$

By the proof of \cite[Theorem 2.1]{AB22}, the idempotent natural transformation associated to the semiseparable left adjoint is given by $e=\nu\circ\eta$. Explicitly, for every $X\in ({}_{\Bbbk G}\m_{\Bbbk G})_{\Bbbk G\oplus\Bbbk G}$, we have 
$e_X(x)=(\nu_X\circ\eta_X)(x)=\nu_X^{\Bbbk G\oplus\Bbbk G}(x\otimes_{\Bbbk G}(1,0))$, for every $x\in X$.
\end{es}
\end{invisible}

\begin{es}\label{es:bimodules}
Let $K$ be a commutative unital ring and let $R$ be a $K$-algebra. We know that the category ${}_R\m_R$ of $R$-bimodules and $R$-bilinear morphisms is monoidal with $\otimes=\otimes_R$ and unit object $R$. Since the tensor product $\otimes_{R}$ is right exact, 
the coflatness condition is fulfilled. 
By \cite[Example 3.2 3)]{BT15}, $R$ is a left $\otimes$-generator if and only if $R$ is a generator in ${}_R\m_R$. In particular, if $R$ is a separable $K$-algebra, then $R$ is a generator in ${}_R\m_R$.  By Maschke Theorem for Hopf algebras \cite{LSw}, if $H$ is a semisimple Hopf algebra, then it is also separable, so $H$ is a left $\otimes$-generator of $({}_H\m_H, \otimes_H, H)$.  

For instance, let $G=\{g_1,\cdots , g_n\}$ be a finite group of cardinality $n\in \mathbb{N}, n \geq 1$, and consider the group algebra $\Bbbk G$ over a field $\Bbbk$ such that $\mathrm{char}(\Bbbk)\nmid n$. Then, $\Bbbk G$ is separable. Hence, the unit object $\Bbbk G$ is a left $\otimes$-generator in $({}_{\Bbbk G}\m_{\Bbbk G}, \otimes_{\Bbbk G}, \Bbbk G)$. We denote $\otimes:=\otimes_\Bbbk$.
Let $I$ be a finite index set of cardinality $\vert I \vert$. Denote $\Bbbk^{\vert I\vert}:= \oplus_{i\in I} \Bbbk$. Note that $\Bbbk^{\vert I\vert} \otimes \Bbbk G$ is an algebra in $({}_{\Bbbk G}\m_{\Bbbk G}, \otimes_{\Bbbk G}, \Bbbk G)$. More precisely, the $\Bbbk G$-module structure of $\Bbbk^{\vert I\vert}\otimes \Bbbk G$ is 
$$
\begin{aligned}
 \Bbbk G \otimes (\Bbbk^{\vert I\vert} \otimes \Bbbk G) &\to \Bbbk^{\vert I\vert} \otimes \Bbbk G, \quad   (\Bbbk^{\vert I\vert}\otimes \Bbbk G) \otimes\Bbbk G \to 
 \Bbbk^{\vert I\vert} \otimes \Bbbk G   \\
 h \otimes (a\otimes g) &\mapsto a \otimes hg \quad \quad\quad \quad \quad \quad \quad (a\otimes g)\otimes h  \mapsto a\otimes gh.
\end{aligned}
$$
Besides, the multiplication is 
$$
m:(\Bbbk^{\vert I\vert}\otimes \Bbbk G) \otimes_{\Bbbk G} (\Bbbk^{\vert I\vert}\otimes \Bbbk G) \to \Bbbk^{\vert I\vert}\otimes \Bbbk G,\,\, (a\otimes g)\otimes_{\Bbbk G} (b\otimes g') \to ab\otimes gg'.
$$ 
The left $\Bbbk G$-module structure of the source of $m$ is given by
$h\cdot((a\otimes g)\otimes_{\Bbbk G} (b\otimes g'))=(a\otimes hg)\otimes_{\Bbbk G} (b\otimes g')$. 
\begin{invisible} To be more specific, $m(h\cdot((a\otimes g)\otimes_{\Bbbk G} (b\otimes g'))) = m(h\cdot(a\otimes g)\otimes_{\Bbbk G} (b\otimes g')) = m((a\otimes hg)\otimes_{\Bbbk G} (b\otimes g')) = ab\otimes hgg'$, and $h\cdot m((a\otimes g)\otimes_{\Bbbk G} (b\otimes g')) = h\cdot (ab\otimes gg') =ab\otimes hgg'$. \end{invisible} Furthermore, the unit is $u : \Bbbk G \to \Bbbk^{\vert I\vert} \otimes \Bbbk G$, $g \mapsto 1\otimes g$. It is routine to show that the multiplication $m$ and unit $u$ are $\Bbbk G$-bilinear.
\begin{invisible} In fact, $u_{\Bbbk^{\vert I\vert} \otimes \Bbbk G}(h\cdot g)=1\otimes hg=h\cdot (1\otimes g)=h\cdot u_{\Bbbk^{\vert I\vert} \otimes \Bbbk G}(g)$.\end{invisible}

Now, suppose $\Bbbk^{\vert I\vert}$ has basis $\{a_j\}_{j\in I}$, where $a_j\in \Bbbk^{\vert I\vert}$ is the element with the $j$th component $1$ and the others are $0$. For any $l\in I$, define 
$$
\psi_l: \Bbbk^{\vert I\vert} \otimes \Bbbk G \to \Bbbk G,\,\, \sum_{i,j} k_{ij} a_i \otimes g_j \mapsto \sum_j k_{lj} g_j \text{ and } D_l: \Bbbk G  \to \Bbbk^{\vert I\vert} \otimes \Bbbk G,\,\, g \mapsto a_l\otimes g.
$$
It is clear that $\psi_l \circ D_l = \id_{\Bbbk G}$. We now show that $\psi_l$ is an algebra morphism in ${}_{\Bbbk G}\m_{\Bbbk G}$. Because $a_ia_s=a_i$ if $i=s$ and $a_ia_s=0$ if $i\not= s$, we have
\begin{align*}
&\psi_l (m((\sum k_{ij}a_i\otimes g_j)\otimes_{\Bbbk G}(\sum k'_{st}a_s\otimes g_t))) = \psi_l(\sum k_{ij} k'_{st}a_ia_s \otimes g_jg_t) = \psi_l(\sum k_{ij} k'_{it}a_i \otimes g_jg_t) \\= &\sum k_{lj} k'_{lt} g_jg_t
=(\sum k_{lj} g_j)(\sum k'_{lt} g_t) =m_{\Bbbk G}(\psi_l (\sum k_{ij}a_i\otimes g_j) \otimes_{\Bbbk G}\psi_l(\sum k'_{st}a_s\otimes g_t) )\\
= &m_{\Bbbk G}(\psi_l\otimes_{\Bbbk G}\psi_l) ((\sum k_{ij}a_i\otimes g_j) \otimes_{\Bbbk G}(\sum k'_{st}a_s\otimes g_t)),
\end{align*}
and $\psi_l u(g)=\psi_l(1\otimes g)=g$. Besides, $D_l$ is $\Bbbk G$-bilinear as $D_l(hg)=a_l\otimes hg=h(a_l\otimes g)=hD_l(g)$ and similarly for the right side.
It  is also $\Bbbk^{\vert I\vert} \otimes\Bbbk G$-bilinear, because
\[
\begin{split}
&D_l((a_k\otimes g)h)=D_l(\psi_l((a_k\otimes g))h)=\begin{cases}
    0\quad \text{if}\,\, k\neq l\\
    D_l(gh)=a_l\otimes gh\quad\text{if}\,\, k=l
\end{cases}\\
& (a_k\otimes g)D_l(h)=(a_k\otimes g)(a_l\otimes  h)=a_ka_l\otimes gh=\begin{cases}
    0\quad\text{if}\,\, k\neq l,\\
    a_l\otimes gh \quad\text{if}\,\, k= l
\end{cases}
\end{split}\]
and similarly for the othe side.
\begin{invisible}
\[
\begin{split}
&D_l(h(a_k\otimes g))=D_l(h\psi_l((a_k\otimes g)))=\begin{cases}
    0\quad \text{if}\,\, k\neq l\\
    D_l(hg)=a_l\otimes hg\quad\text{if}\,\, k=l
\end{cases}\\
& D_l(h)(a_k\otimes g)=(a_l\otimes  h)(a_k\otimes g)=a_la_k\otimes gh=\begin{cases}
    0\quad\text{if}\,\, k\neq l,\\
    a_l\otimes hg \quad\text{if}\,\, k= l
\end{cases}
\end{split}\]
\end{invisible}
Moreover, $\psi_l D_l(g)=\psi_l(a_l\otimes g)=g$, hence $\psi_l^*$ is naturally full by Proposition \ref{prop:phi*natfull}.

Now consider $\Bbbk G^{\otimes m }:=\Bbbk G\otimes\ldots\otimes \Bbbk G$ ($m$ times), where $m>1$ is a fixed natural number. The left $\Bbbk G$-module structure of $ \Bbbk G^{\otimes m }$ is defined on the basis, for $h\in G$, by
$ h \otimes (g_1\otimes\ldots \otimes g_m) \mapsto hg_1 \otimes\ldots\otimes hg_m$, 
and similarly for the right side. We observe that $ \Bbbk G^{\otimes m}$ is a monoid in ${}_{\Bbbk G}\m_{\Bbbk G}$ with multiplication $$m_{\Bbbk G^{\otimes m }}:\Bbbk G^{\otimes m }\otimes_{\Bbbk G}\Bbbk G^{\otimes m }\to \Bbbk G^{\otimes m },\,\,(a_1\otimes\ldots\otimes a_m)\otimes_{\Bbbk G} (b_1\otimes\ldots\otimes b_m)\mapsto a_1b_1\otimes \ldots\otimes a_mb_m,$$ where the $\Bbbk G$-module structure of $\Bbbk G^{\otimes m }\otimes_{\Bbbk G}\Bbbk G^{\otimes m }$ is 
given by
\[
\begin{split}
g\cdot ((a_1\otimes \ldots\otimes a_m)\otimes_{\Bbbk G} (b_1\otimes \ldots\otimes b_m))&=(g\cdot (a_1\otimes \ldots\otimes a_m))\otimes_{\Bbbk G}(b_1\otimes \ldots\otimes b_m)\\&=(ga_1\otimes\ldots\otimes ga_m)
\otimes_{\Bbbk G} (b_1\otimes \ldots\otimes b_m).
\end{split}
\]
\begin{invisible}
We have 
\[
\begin{split}
&m_{\Bbbk G^{\otimes m }}(g\cdot ((a_1 
\otimes \ldots\otimes a_m)\otimes_{\Bbbk G} (b_1\otimes\ldots\otimes b_m)))\\=&m_{\Bbbk G^{\otimes m }}((ga_1\otimes\ldots\otimes ga_m)
\otimes_{\Bbbk G} (b_1 \otimes \ldots\otimes b_m))=ga_1b_1\otimes \ldots\otimes ga_mb_m\\
=&g(a_1b_1\otimes \ldots\otimes a_mb_m)=g\cdot m_{\Bbbk G^{\otimes m }}((a_1\otimes\ldots\otimes  a_m)\otimes_{\Bbbk G} (b_1\otimes \ldots\otimes b_m)).
\end{split}
\]
Similarly, 
\end{invisible}
One can check that $m_{\Bbbk G^{\otimes m }}$ is left and right $\Bbbk G$-linear. The unit is given by $u_{\Bbbk G^{\otimes m }}:\Bbbk G\to \Bbbk G^{\otimes m }$, $g\mapsto g\otimes \ldots\otimes g$.
Consider the algebra morphism $$\varphi: \Bbbk G\to \Bbbk G^{\otimes m },\quad \sum_{i=1}^{n} k_ig_i\mapsto \sum_{i=1}^{n}k_ig_i\otimes\ldots\otimes g_i$$
in ${}_{\Bbbk G}\m_{\Bbbk G}$. \begin{invisible}
Note that $\varphi$ is an algebra morphism in ${}_{\Bbbk G}\m_{\Bbbk G}$ as $\varphi m_{\Bbbk G}(a\otimes c)=\varphi(ac)=ac\otimes  \ldots\otimes ac=m_{\Bbbk G^{\otimes m }}(\varphi\otimes \varphi)(a\otimes c)$, for $a,c\in G$, and $\varphi u_{\Bbbk G}(g)=g\otimes \ldots\otimes g=u_{\Bbbk G^{\otimes m }}(g)$.
\end{invisible} Define $E:=\varepsilon^{\otimes (m-1)}\otimes\id: \Bbbk G^{\otimes m }\to \Bbbk G$, $E(h_1\otimes \ldots\otimes  h_m) = h_m$, where $h_1,\ldots, h_m\in G$. Then, $E$ is in ${}_{\Bbbk G}\Hom_{\Bbbk G}(\Bbbk G^{\otimes m }, \Bbbk G)$, since $E(t(h_1\otimes \ldots\otimes h_m))=E(th_1\otimes \ldots\otimes th_m)=th_m=tE(h_1\otimes\ldots\otimes h_m)$ and similarly for the right side.
\begin{invisible}$E((h_1\otimes \ldots\otimes h_m)t)=E(h_1t\otimes\ldots\otimes h_mt)=h_mt=E(h_1\otimes\ldots\otimes h_m)t$.\end{invisible} Moreover,  $E\varphi(\sum_ik_ig_i)=E(\sum_ik_ig_i\otimes\ldots\otimes g_i)=\sum_ik_ig_i$, hence $\varphi^*$ is separable by Theorem \ref{thm:sepmon}. Note that $E$ can be also chosen to be $\varepsilon\otimes\ldots\otimes\id\otimes\varepsilon$, $\cdots$, $\id\otimes \varepsilon\otimes\ldots\otimes\varepsilon$ ($m-1$ ways). Then, by Lemma \ref{lem:comp-semisep} $\varphi^*\circ\psi^*\cong (\varphi\circ \psi)^*$ is semiseparable. We observe that $\varphi\psi$ is neither injective nor surjective, so $\varphi^*\circ\psi^*$ is neither separable nor naturally full. By Theorem \ref{thm:phi*semisep} there exists an $\Bbbk^{\vert I\vert} \otimes \Bbbk G $-bimodule morphism $\xi: \Bbbk G^{\otimes m }\to \Bbbk^{\vert I\vert} \otimes \Bbbk G $ in ${}_{\Bbbk G}\m_{\Bbbk G}$ such that $\varphi\psi\xi \varphi\psi=\varphi\psi$. For instance, $\xi$ can be chosen as $D_l E: \Bbbk G^{\otimes m }\to \Bbbk^{\vert I\vert} \otimes \Bbbk G $, $g_1\otimes\ldots\otimes g_m\mapsto a_l\otimes g_m $.
\begin{invisible}
The unit $\eta$ of the adjunction $(\varphi\psi)^*\dashv (\varphi\psi)_*$ is given for every $X\in ({}_{\Bbbk G}\m_{\Bbbk G})_{\Bbbk^{\vert I\vert} \otimes \Bbbk G }$ by 
\[
\eta_X:X\to X\otimes_{\Bbbk^{\vert I\vert} \otimes \Bbbk G }\Bbbk G^{\otimes m },\quad\eta_X=q^{\Bbbk^{\vert I\vert} \otimes \Bbbk G }_{X,\Bbbk G^{\otimes m }}(\id_{X}\otimes_{\Bbbk G} u_{\Bbbk G^{\otimes m }}), \quad x\mapsto x\otimes_{\Bbbk^{\vert I\vert} \otimes \Bbbk G} (1\otimes\ldots\otimes 1).
\]
By Theorem \ref{thm:phi*semisep}, we have $\eta$ is regular through $\nu:(\varphi\psi)_*(\varphi\psi)^* \to \id$, given for every $X\in ({}_{\Bbbk G}\m_{\Bbbk G})_{\Bbbk^{\vert I\vert} \otimes \Bbbk G}$ by 
\[
\nu_X=\Upsilon_X\widetilde{(D_lE)}: X\otimes_{\Bbbk^{\vert I\vert} \otimes \Bbbk G} \Bbbk G^{\otimes m }\to X, 
\]
where $\Upsilon_X:X\otimes_{\Bbbk^{\vert I\vert} \otimes \Bbbk G} (\Bbbk^{\vert I\vert} \otimes \Bbbk G) \to X$ is the isomorphism uniquely determined by the property $\Upsilon_Xq^{\Bbbk^{\vert I\vert} \otimes \Bbbk G}_{X,\Bbbk^{\vert I\vert} \otimes \Bbbk G}=\nu^{(\Bbbk^{\vert I\vert} \otimes \Bbbk G}_X$, and $\widetilde{D_lE}=\tilde{D_l}\tilde{E}$ is the unique morphism in ${}_{\Bbbk G}\m_{\Bbbk G}$ satisfying the equation 
\[
\widetilde{(D_lE)}q^{\Bbbk^{\vert I\vert} \otimes \Bbbk G}_{X,\Bbbk G^{\otimes m }}=q^{\Bbbk^{\vert I\vert} \otimes \Bbbk G}_{X, \Bbbk^{\vert I\vert} \otimes \Bbbk G}(\id_X\otimes_{\Bbbk G} DE).
\]
We have $\widetilde{D_lE}: X\otimes_{\Bbbk^{\vert I\vert} \otimes \Bbbk G}\Bbbk G^{\otimes m }\to  X\otimes_{\Bbbk^{\vert I\vert} \otimes \Bbbk G} (\Bbbk^{\vert I\vert} \otimes \Bbbk G)$, $x\otimes_{\Bbbk^{\vert I\vert} \otimes \Bbbk G}(g_1\otimes_{\Bbbk}\ldots\otimes g_m)\mapsto x\otimes_{\Bbbk^{\vert I\vert} \otimes \Bbbk G}(a_l\otimes g_m)$. 

Then, $$\nu_X(x\otimes_{\Bbbk^{\vert I\vert} \otimes \Bbbk G}(g_1\otimes\ldots\otimes  g_m))=\nu^{\Bbbk^{\vert I\vert} \otimes \Bbbk G}_X(x\otimes_{\Bbbk G}(a_l\otimes g_m)).$$

By the proof of \cite[Theorem 2.1]{AB22}, the idempotent natural transformation associated to the semiseparable left adjoint is given by $e=\nu\circ\eta$. Explicitly, for every $X\in ({}_{\Bbbk G}\m_{\Bbbk G})_{\Bbbk^{\vert I\vert} \otimes \Bbbk G}$, we have 
$e_X(x)=(\nu_X\circ\eta_X)(x)=\nu_X^{\Bbbk^{\vert I\vert} \otimes \Bbbk G}(x\otimes_{\Bbbk G}(a_l\otimes 1))$, for every $x\in X$.
\end{invisible}
\end{es}




\begin{es}
 Let $\Bbbk$ be 
 a field and $\mathcal{Z}_\Bbbk$ be the category of Zunino over $\Bbbk$, see \cite[2.2]{CDL06}. Objects of $\mathcal{Z}_\Bbbk$ are pairs $\underline{M}=(X, (M_x)_{x\in X})$ consisting of a set $X$ and a family $(M_x)_{x\in X}$ of $\Bbbk$-vector spaces indexed by $X$. A morphism $\varphi:\underline{M}=(X, (M_x)_{x\in X})\to \underline{N}=(Y, (N_y)_{y\in Y})$ in $\mathcal{Z}_\Bbbk$ is a pair $(f, (\varphi_x)_{x\in X})$, where $f:X\to Y$ is a morphism in $\Set$ and $(\varphi_x:M_x\to N_{f(x)})_{x\in X}$ is a family of $\Bbbk$-linear maps. If $\underline{\chi}=(g, (\chi_y)_{y\in Y}):\underline{N}\to\underline{P}$ is another morphism in $\mathcal{Z}_\Bbbk$, then the composition is defined by $\underline{\chi}\circ \underline{\varphi}=(g\circ f, (\chi_{f(x)}\circ\varphi_x)_{x\in X} )$. The monoidal structure of $\mathcal{Z}_\Bbbk$ is given by
\[
\underline{M}\otimes\underline{N}=(X\times Y, (M_x\otimes N_y)_{(x,y)\in X\times Y} )
\]
for objects $\underline{M}, \underline{N}$ in $\mathcal{Z}_\Bbbk$, and
\[
\underline{\varphi}\otimes
\underline{\varphi}'=(f\times f', (\varphi_x\otimes\varphi'_{x'})_{(x,x')\in X\times X'})
\]
for morphisms $\underline{\varphi}:\underline{M}\to \underline{N}$, $\underline{\varphi}':\underline{M}'\to \underline{N}'$ in $\mathcal{Z}_\Bbbk$.
The unit is $(\{*\}, \Bbbk)$ and in \cite[Example 3.2 5)]{BT15} it is shown that it is a left $\otimes$-generator.

By \cite[Proposition 2.11]{CDL06}, the category $\mathcal{Z}_\Bbbk$ has coequalizers. Moreover, by \cite[Proposition 3.4]{AGM23} the category $\mathcal{Z}_\Bbbk=\mathsf{Fam}(\vect_\Bbbk)$ is closed monoidal. Hence, any object in $\mathcal{Z}_\Bbbk$ is coflat. Indeed, one can also check the coflatness of objects in $\mathcal{Z}_\Bbbk$ directly by using the coflatness of objects in $\Set$ and in $\vect_\Bbbk$. 
\begin{invisible}
We also provide a sketch of a direct proof for the coflatness of every object in $\mathcal{Z}_\Bbbk$. Let $\underline{\varphi}=(f, (\varphi_x:M_x\to N_{f(x)})_{x\in X})$, $\underline{\psi}=(g, (\psi_{x}:M_x\to N_{g(x)})_{x\in X}):(X, {(M_x)}_{x\in X})\to (Y, {(N_y)}_{y\in Y})$ be parallel morphisms in $\mathcal{Z}_\Bbbk$. By \cite[Proposition 2.11]{CDL06}, their coequalizer in $\mathcal{Z}_\Bbbk$ is given by 
\[
\underline{q}=(q, (p_y)_{y\in Y}):(Y, {(N_y)}_{y\in Y})\to (\overline{Y}, (\overline{N}_{\overline{y}})_{\overline{y}\in \overline{Y}}),
\]
where $\overline{Y}=\mathrm{Coeq}(f,g)$,
$\overline{N}_{\overline{y}}=\coprod_{y\in\overline{y}} N_y/\sim$,  $\sim$ is the vector space generated by
\[
\lbrace (i_{f(x)}\circ\varphi_x)(m)- (i_{g(x)}\circ\psi_x)(m)\,\vert\, x\in f^{-1}(\overline{y})=g^{-1}(\overline{y}), m\in M_x\rbrace,
\]
and $p_y$ is the composition of the inclusion $i_y:N_y\to \coprod_{y\in\overline{y}} N_y$ with the projection $\pi_{\overline{y}}:\coprod_{y\in\overline{y}} N_y\to\overline{N}_{\overline{y}}$. 
We want to show that \begin{equation*}
\xymatrix@C=0.5cm{(X,(M_x)_{x\in X})\otimes (W, (L_w)_{w\in W}) \ar@<-.8ex>[r]_-{\underline{\psi}\otimes\underline{\id} }\ar@<.8ex>[r]^-{\underline{\varphi}\otimes\underline{\id}}&(Y,(N_y)_{y\in Y})\otimes (W, (L_w)_{w\in W})\ar[r]^-{\underline{q}\otimes\underline{\id}}& (\overline{Y}, (\overline{N}_{\overline{y}})_{\overline{y}\in \overline{Y}})\otimes (W, (L_w)_{w\in W}) }
\end{equation*}
is a coequalizer in $\mathcal{Z}_\Bbbk$. First, it is clear that   $(\underline{q}\otimes\underline{\id})(\underline{\varphi}\otimes\underline{\id})=(\underline{q}\otimes\underline{\id})(\underline{\psi}\otimes\underline{\id})$. We now show the universal property. Suppose there is a morphism 
$$
\underline{\theta}=(h,\theta_{(y,w)}):(Y\times W, (N_y\otimes L_w)_{(y,w)\in Y\times W})\to (Z, (H_z)_{z\in Z})
$$
in $\mathcal{Z}_\Bbbk$ such that $\underline{\theta}(\underline{\varphi}\otimes\underline{\id}) = \underline{\theta}(\underline{\psi}\otimes\underline{\id})$. This means $h(f\times \id_W)= h(g\times \id_W)$ and $\theta_{(f(x),w)}(\varphi_x\otimes \id_{L_w})= \theta_{(g(x),w)}(\psi_x\otimes \id_{L_w})$ for any $x \in X, w\in W$. Because any object in $\Set$ is coflat, $q\times \id_W$ is the coequalizer of $f\times \id_W$ and $g \times \id_W$. By the universal property of this coequalizer, there is a unique morphism $t:\overline{Y} \times W\to Z$ in $\Set$ such that $h=t(q\times \id_W)$. For the second component, we show first that $$\coprod_{(y,w)\in \overline{(y,w)}} (N_y\otimes L_w)/\sim' \cong (\coprod_{y\in\overline{y}} N_y/\sim)\otimes L_w,$$ where $\overline{(y,w)}$ is in $\mathrm{Coeq}(f \times \id_W, g \times \id_W)$, $\overline{y}$ is in $\mathrm{Coeq}(f, g)$ and $\sim'$ is the vector space generated by
$$
\begin{aligned}
\lbrace (i_{(f(x),w)}\circ &(\varphi_x \otimes \id_{L_w}))(m\otimes l)- (i_{(g(x),w)}\circ (\psi_x \otimes \id_{L_w}))(m \otimes l)\,\vert\,\\
&(x,w)\in (f\times \id_W)^{-1}(\overline{(y,w)})=(g\times \id_W)^{-1}(\overline{(y,w)}), m\otimes l \in M_x \otimes L_w\rbrace.
\end{aligned}
$$ 
Since $\mathrm{Coeq}(f \times \id_W, g \times \id_W) = \mathrm{Coeq}(f, g)\times W$, we can rewrite $\overline{(y,w)} = (\overline{y},w)$. Therefore, $\coprod_{(y,w)\in \overline{(y,w)}} (N_y\otimes L_w)=  \coprod_{y\in \overline{y}} (N_y\otimes L_w) \cong (\coprod_{y\in \overline{y}} N_y)\otimes L_w$. Besides, since $i_{(f(x),w)}=\alpha( i_{f(x)}\otimes \id_{L_w})$ and $i_{(g(x),w)}= \alpha(i_{g(x)}\otimes \id_{L_w})$, where $\alpha:(\coprod_{y\in \overline{y}} N_y)\otimes L_w \to \coprod_{y\in \overline{y}} (N_y\otimes L_w)$ is the isomorphism, it follows that $\sim'$ is the vector space generated by 
$$
\begin{aligned}
\lbrace \alpha(i_{f(x)}\otimes \id_{L_w}) &(\varphi_x \otimes \id_{L_w})(m\otimes l)- \alpha(i_{g(x)}\otimes \id_{L_w}) (\psi_x \otimes \id_{L_w})(m \otimes l)\,\vert\,\\
&(x,w)\in (f\times \id_W)^{-1}(\overline{y},w)=(g\times \id_W)^{-1}(\overline{y},w), m\otimes l \in M_x \otimes L_w\rbrace,
\end{aligned}
$$ 
which is isomorphic to the vector space generated by
$$
\lbrace i_{f(x)} \varphi_x(m)\otimes l- i_{g(x)} \psi_x(m) \otimes l\,\vert\,
x\in f^{-1}(\overline{y}) = g^{-1}(\overline{y}), m\otimes l \in M_x \otimes L_w\rbrace.
$$
Hence, $\sim'$ is isomorphic to $\sim \otimes L_w$.
Thus, from coflatness of $\vect_\Bbbk$, we obtain an isomorphism $\Phi:(\coprod_{y\in\overline{y}} N_y/\sim)\otimes L_w \to \coprod_{(y,w)\in \overline{(y,w)}} (N_y\otimes L_w)/\sim'$, which induces the canonical projection $$\pi_{(\overline{y},w)} = \Phi(\pi_{\overline{y}}\otimes \id_{L_w})\alpha^{-1}:\coprod_{y\in\overline{y}}(N_y\otimes L_w) \to \coprod_{y\in \overline{y}} (N_y\otimes L_w)/\sim'.$$ By the universal property of the coqualizer in $\vect_\Bbbk$, there exists a unique $\gamma_{(\overline{y},w)}: \overline{N}_{\overline{y}}\otimes L_w \to H_{t(\overline{y},w)}$ such that the following diagram commutes
$$
\xymatrix{
\sim' \ar[r] & \coprod_{(y,w)\in (\overline{y},w)} (N_y\otimes L_w) \ar[r]^{\pi_{(\overline{y},w)}} \ar[d]_{\sum_{y\in \overline{y}} \theta_{(y,w)} p_{y,w}} & \coprod_{(y,w)\in (\overline{y},w)} (N_y\otimes L_w)/\sim' \ar@{.>}[ld]^{\gamma_{(\overline{y},w)}} \\
& H_{t(\overline{y},w)} & 
}
$$
where $p_{y,w}: \coprod_{(y,w)\in (\overline{y},w)} (N_y\otimes L_w) \to N_y\otimes L_w$ is the canonical projection. Therefore, 
$$
\begin{aligned}
\gamma_{(\overline{y},w)} \Phi (p_y\otimes \id_{L_w}) &= \gamma_{(\overline{y},w)} \Phi (\pi_{\overline{y}} i_y \otimes \id_{L_w}) = \gamma_{(\overline{y},w)} \Phi(\pi_{\overline{y}} \otimes \id_{L_w}) \alpha^{-1} \alpha (i_y \otimes \id_{L_w})\\ 
&= (\sum_{y' \in \overline{y}} \theta_{(y',w)} p_{y',w})\alpha (i_y \otimes \id_{L_w}) = (\sum_{y'\in \overline{y}} \theta_{(y',w)} p_{y',w}) i_{(y,w)} = \theta_{(y,w)}.
\end{aligned}
$$ 
It follows that $(t,\gamma_{(\overline{y},w)} \Phi) (\underline{q}\otimes\underline{\id}) = \underline{\theta}$, and the proof of the universal property is completed.
\end{invisible}
Therefore, Theorem \ref{thm:phi*semisep} and Proposition \ref{prop:phi*natfull} apply. Furthermore, by \cite[Proposition 2.7]{CDL06}  $(G,A)$ is an algebra in $\mathcal{Z}_\Bbbk$ if, and only if, $G$ is a monoid and $A$ is a $G$-algebra in the sense of \cite[Definition 1.6]{CDL06}, which is equivalent to say that algebras in $\mathcal{Z}_\Bbbk$ are in one to one correspondence to algebras graded by a monoid, see \cite[Example 4.2 4)]{BT15}.
\end{es}

\subsection{Tensor functors of algebras in a monoidal category}\label{subsect:alg-mon}

Let $(A, m_A, u_A)$ be an algebra in a monoidal category $\cc$. Consider the forgetful functor $F:\cc_A\to\cc$ and its left adjoint $G=-\otimes A:\cc\to \cc_A$, $M\mapsto (M\otimes A,M\otimes m_A )$. The unit $\eta:\id\to FG$ of the adjunction is given by $\eta_{M}= M\otimes u_A:M\to M\otimes A$, for every object $M$ in $\cc$, while the counit $\epsilon : GF\to \id$ is given by $\epsilon_{N}= \mu^A_N: F(N)\otimes A= N\otimes A\to N$, for every object $(N, \mu^A_N)$ in $\cc_A$.

As observed in \cite[page 730]{BT15}, when the algebra extension is given by the unit morphism $u_A:1\to A$, the induction functor $u_A^*=-\otimes _1 A: \cc_1\to \cc_A$ is just the tensor functor $-\otimes A:\cc\to\cc_A$. In fact, every object in $\cc$ has automatically a right $1$-module structure, hence $\cc_1$ coincides with $\cc$. 
Thus, Theorem \ref{thm:phi*semisep} applies. More precisely, if $u_A:1\to A$ is regular, by Theorem \ref{thm:phi*semisep}, $u_A^*$ is semiseparable. Hence, so is $-\otimes A:\cc\to\cc_A$. Note that Theorem \ref{thm:phi*semisep} requires that the monoidal category $\cc$ has coequalizers and every object is coflat. However, we can obtain the result without these assumptions. 
     
\begin{prop}\label{prop:algebra-induct-functor} 
    Let $A$ be an algebra in a monoidal category $\cc$. Then, the following assertions are equivalent:
    \begin{itemize}
        \item[$(i)$] The unit $u_A:1\to A$ is regular (resp., split-mono, split-epi) as a morphism in $\cc$.

        \item[$(ii)$] The tensor functor $-\otimes A:\cc\to\cc_A$ is semiseparable (resp., separable, naturally full).

        \item[$(iii)$] The tensor functor $A\otimes -:\cc\to {}_A\cc$ is semiseparable (resp., separable, naturally full).
    \end{itemize} 
\end{prop}
\begin{proof} We consider the semiseparable case. The separable and naturally full cases can be shown similarly.

    $(i)\Rightarrow (ii)$. Assume that $u_A:1\to A$ is regular, i.e.\ there exists $\chi:A\to 1$ in $\cc$ such that $u_A\chi  u_A=u_A$. For every $M\in \cc_A$, define $\nu_M= \chi\otimes M: M\otimes A\to M$, which is clearly a natural transformation. 
    Moreover, $\eta_M\nu_M\eta_M=(M\otimes u_A)(M\otimes \chi)(M\otimes u_A)=M\otimes u_A=\eta_M$.
    Thus, by Theorem \ref{thm:rafael} the functor $-\otimes A:\cc\to \cc_A$ is semiseparable.
    


 $(ii)\Rightarrow (i)$. If $-\otimes A$ is semiseparable, then by Theorem \ref{thm:rafael} there exists a natural transformation $\nu : F(-\otimes A)\to\id $ such that $\eta\nu\eta=\eta$. Define $\chi=\nu_1: A\to 1$. Since $u_A\chi u_A=(1\otimes u_A)\nu_1(1\otimes u_A)=\eta_1\nu_1 \eta_1=\eta_1=1\otimes u_A=u_A$, we get $(i)$.


    $(i)\Leftrightarrow (iii)$. It follows similarly. 
\end{proof}

\begin{rmk}\label{rmk:algmonoid-induct}
    We observe that $(i)$ in Proposition \ref{prop:algebra-induct-functor} is equivalent to the existence of a morphism $\chi:A\to 1$ in $\cc$ such that $(\chi\otimes A)(u_A\otimes A)=\id_A$. On one hand, $u_A\chi u_A =u_A$ implies  $u_A\otimes A=u_A\chi u_A\otimes A=(u_A\otimes A)(\chi u_A \otimes A)$. Since $u_A\otimes A$ is a monomorphism, $\id_A=\chi u_A\otimes A=(\chi\otimes A)(u_A\otimes A)$.   On the other hand, $(\chi\otimes A)(u_A\otimes A)=\id_A$ implies $u_A=1\otimes u_A=(\chi\otimes A)(u_A\otimes A)(1\otimes u_A)=(1\otimes u_A)(\chi\otimes 1)(u_A\otimes 1)=u_A\chi u_A$. 
    \end{rmk}

\begin{es}
    Let $A$ be an algebra with augmentation $\varepsilon : A\to 1$ in a monoidal category $\cc$. Since $\varepsilon$ is an algebra morphism in $\cc$, we have $\varepsilon  u_A=\id_1$, so $u_A$ is a split-mono. By Proposition \ref{prop:algebra-induct-functor}, $-\otimes A:\mathcal{C}\to \mathcal{C}_A$ is separable.  Example \ref{es:bialgmod} constitutes an instance of this situation, where  $u:\Bbbk \to H$ is the unit of a Hopf algebra over a field $\Bbbk$. 
\end{es}

\subsection{Induction functors varied by a monoidal functor}\label{subsect:monoidal-functor}
Given monoidal categories $(\cc,\otimes,1)$ and $(\dd,\otimes,1')$, a functor $F:\cc\to\dd$ is said to be \emph{lax monoidal}, see e.g. \cite[Definition 3.1]{AM10}, if there is a natural transformation $\phi^2$, given on components by $F(X)\otimes F(Y)\to F(X\otimes Y)$, and a morphism $\phi^0:1'\to F(1)$ such that
\begin{align*}
   \phi^2_{X, Y\otimes Z}(\id\otimes \phi^2_{Y,Z}) &=\phi^2_{X\otimes Y, Z}(\phi^2_{X,Y}\otimes \id),\\
l_{F(X)}=F(l_X)\phi^2_{1,X}(\phi^0\otimes \id),&\quad r_{F(X)}= F(r_X)\phi^2_{X,1}(\id\otimes \phi^0),
\end{align*}
where $l, r$ are the left and right unit constraints of the understood categories. If $\phi^2$, $\phi^0$ are invertible, then the functor $F$ is said to be \emph{strong monoidal}. Moreover, a lax monoidal functor $(F:\cc\to\dd,\phi^2,\phi^0)$ maps an algebra $(R, m_R, u_R)$ in $\cc$ into an algebra $(F(R), F(m_R) \phi^2_{R,R}, F(u_R) \phi^0)$ in $\dd$, see e.g.\ \cite[Proposition 3.29]{AM10}. Moreover, if $\varphi:R\to S$ is an algebra morphism in $\cc$, then $F(\varphi):F(R)\to F(S)$ is an algebra morphism in $\dd$. 

\begin{invisible}
Recall from \cite[Definition 6.1]{Sz01} the notion of  separable Frobenius monoidal functor. Let $(\cc, \otimes, 1)$ and $(\dd, \otimes, I)$ be  monoidal categories. A functor $F:\cc\to\dd$ is called \emph{Frobenius monoidal} if it is equipped with a monoidal structure $(\phi^2, \phi^0)$ and an opmonoidal structure $(\psi^2,\psi^0)$ such that, for every $X,Y,Z\in\cc$, the diagrams
\begin{equation}\label{eq:semisepFrob1}
\begin{split}	
\xymatrixcolsep{1cm}\xymatrixrowsep{0.5cm}\xymatrix{F(X\otimes Y)\otimes F(Z)\ar[r]^-{\psi^2_{X,Y}\otimes\id}\ar[d]_-{\phi^2_{X\otimes Y, Z}}&F(X)\otimes F(Y)\otimes F(Z)\ar[d]^-{\id\otimes\phi^2_{Y,Z}}\\F(X\otimes Y\otimes Z)\ar[r]_-{\psi^2_{X,Y\otimes Z}}&F(X)\otimes F(Y\otimes Z)\\F(X)\otimes F(Y\otimes Z)\ar[r]^-{\id\otimes\psi^2_{Y,Z}}\ar[d]_-{\phi^2_{X,Y\otimes Z}}&F(X)\otimes F(Y)\otimes F(Z)\ar[d]^-{\phi^2_{X,Y}\otimes\id}\\F(X\otimes Y\otimes Z)\ar[r]_-{\psi^2_{X\otimes Y, Z}}&F(X\otimes Y)\otimes F(Z)}
\end{split}
\end{equation}
are commutative. If, furthermore, $\phi^2_{X,Y}\psi^2_{X,Y}=\id_{F(X\otimes Y)}$, for every $X,Y\in\cc$, then $F$ is said to be \emph{separable Frobenius monoidal}. 
\end{invisible}

\begin{lem}\label{lem:module-monoidal}
Let $(\cc, \otimes, 1)$, $(\dd, \otimes, 1')$ be  monoidal categories, $(F:\cc\to\dd,\phi^2,\phi^0)$ be a lax monoidal functor, $R$ be an algebra in $\cc$, and let $X$ be a right $R$-module via $\nu_X^R$ in $\cc$. Then, $F(X)$ is a right $F(R)$-module with module structure $F(\nu_X^R)\circ \phi^2_{X,R}$. Similarly, if $X$ is a left $R$-module via $\mu_X^R$, then $F(X)$ is a left $F(R)$-module with module structure $F(\mu_X^R)\circ \phi^2_{R,X}$.

As a consequence, $F$ induces functors $\mathcal{C}_R\to \dd_{F(R)}$, ${}_R\cc\to {}_{F(R)}\dd$ and ${}_R\cc_R\to {}_{F(R)}\dd_{F(R)}$. 
\end{lem}
\begin{proof}
    It is straightforward.
\end{proof}

\begin{invisible}
\begin{proof}
The module structure is verified by the following calculation:
$$
\begin{aligned}
&F(\nu_X^R)\circ \phi^2_{X,R}\circ (\id_{FX}\otimes (F(m_R)\circ \phi^2_{R,R}))
= F(\nu_X^R)\circ \phi^2_{X,R}\circ (\id_{FX}\otimes F(m_R))  \circ (\id_{FX}\otimes \phi^2_{R,R})\\
= &F(\nu_X^R)\circ F(\id_{X}\otimes m_R) \circ \phi^2_{X,R\otimes R} \circ (\id_{FX}\otimes \phi^2_{R,R})
= F(\nu_X^R)\circ F(\nu_{X}^R\otimes \id_{R}) \circ \phi^2_{X\otimes R, R} \circ (\phi^2_{X,R} \otimes \id_{FR})\\
= &F(\nu_X^R)\circ \phi^2_{X,R}\circ ((F(\nu_X^R)\circ \phi^2_{X,R}) \otimes \id_{FR}),
\end{aligned}
$$
and
\[
\begin{aligned}
&F(\nu_X^R)\circ \phi^2_{X,R} \circ (\id_{FX}\otimes (F(u)\circ \phi^0))= F(\nu_X^R)\circ \phi^2_{X,R} \circ (\id_{FX}\otimes F(u)) \circ (\id_{FX}\otimes \phi^0))\\
= &F(\nu_X^R)\circ F(\id_X\otimes u) \circ \phi^2_{X,1} \circ (\id_{FX}\otimes \phi^0) = \id_{FX}.
\end{aligned}
\]
Similarly,
    \[
\begin{split}
    &F(\mu_X^R)\circ \phi^2_{R,X}\circ ( (F(m_R)\circ \phi^2_{R,R})\otimes\id_{FX})\\
= &F(\mu_X^R)\circ \phi^2_{R,X}\circ ( F(m_R)\otimes\id_{FX}) \circ (\phi^2_{R\otimes R, X})^{-1} \circ \phi^2_{R\otimes R, X} \circ (\phi^2_{R,R}\otimes\id_{FX})\\
= &F(\mu_X^R)\circ F( m_R\otimes\id_X) \circ \phi^2_{R\otimes R, X} \circ (\phi^2_{R,R}\otimes\id_{FX})\\
= &F(\mu_X^R)\circ F(\id_R\otimes \mu_{X}^R) \circ \phi^2_{R,R\otimes X} \circ (\id_{FR}\otimes\phi^2_{R,X})\\
= &F(\mu_X^R)\circ \phi^2_{R,X}\circ ( \id_{FR}\otimes(F(\mu_X^R)\circ \phi^2_{R,X}) ),
\end{split}
    \]
    and 
   \[
\begin{split}
    &F(\mu_X^R)\circ \phi^2_{R,X} \circ ( (F(u)\circ \phi^0)\otimes\id_{FX})= F(\mu_X^R)\circ \phi^2_{R,X} \circ ( F(u)\otimes\id_{FX}) \circ (\phi^0\otimes\id_{FX}))\\
= &F(\mu_X^R)\circ F(u\otimes\id_X) \circ \phi^2_{1,X} \circ (\phi^0\otimes\id_{FX}) = \id_{FX}.
\end{split}
   \] 
\end{proof}    
\end{invisible}
The next result shows that, under the tensor generator assumption, the semiseparability (resp., separability, natural fullness) of induction functors is preserved by a lax monoidal functor.

\begin{prop}\label{prop:algebra-monoidalfunctor}
 Let $(\cc, \otimes, 1)$, $(\dd, \otimes, 1')$ be monoidal categories. 
 Assume that $1$ is a left $\otimes$-generator in $\cc$. Let $(F:\cc\to\dd,\phi^2,\phi^0)$ be a lax monoidal functor and let $\varphi:R\to S$ be an algebra morphism in $\cc$ such that $\varphi^*=-\otimes_RS:\cc_R\to\cc_S$ is semiseparable (resp., separable, naturally full). Then, $F(\varphi)^*=-\otimes_{F(R)}F(S):\dd_{F(R)}\to\dd_{F(S)}$ is semiseparable (resp., separable, naturally full).

 Conversely, suppose $1'$ is a left $\otimes$-generator in $\dd$ and $(F:\cc\to\dd,\phi^2,\phi^0)$ is a fully faithful lax monoidal functor such that $\phi^2$ is epimorphism on components. If $F(\varphi)^*=-\otimes_{F(R)}F(S):\dd_{F(R)}\to\dd_{F(S)}$ is semiseparable (resp., separable, naturally full), then $\varphi^*=-\otimes_RS:\cc_R\to\cc_S$ is semiseparable (resp., separable, naturally full).
\end{prop}
\begin{proof}
If $\varphi^*$ is semiseparable, then by Theorem \ref{thm:phi*semisep} there exists an $R$-bimodule morphism $E:S\to R$ in $\cc$ such that $\varphi E\varphi=\varphi$. It follows that $F(\varphi)F(E) F(\varphi)=F(\varphi E\varphi)=F(\varphi)$. By Lemma \ref{lem:module-monoidal}, $F(E)$ is an $F(R)$-bimodule morphism in $\dd$. 
\begin{invisible}
Indeed, 
\[
\begin{split}
    F(E)&\circ m_{F(S)}\circ (\id_{F(S)}\otimes F(\varphi))=F(E\circ m_S)\circ\phi^2_{S.S}\circ(\id_{F(S)}\otimes F(\varphi))\\&=F(E\circ m_S\circ (\id_S\otimes\varphi))\circ\phi^2_{S,R}=F(m_R\circ (E\otimes\id_R))\circ\phi^2_{S,R}\\&=F(m_R)\circ\phi^2_{R,R}\circ (F(E)\otimes \id_{F(R)})=m_{F(R)}\circ (F(E)\otimes\id_{F(R)}),
\end{split} 
\]
and similarly $F(E)\circ m_{F(S)}\circ (F(\varphi)\otimes \id_{F(S)})= m_{F(R)}\circ (\id_{F(R)}\otimes F(E))$.
\end{invisible}
Thus, by Theorem \ref{thm:phi*semisep} $F(\varphi)^*$ is semiseparable. 

Now, suppose $1'$ is a left $\otimes$-generator in $\dd$, $F$ is fully faithful, and $F(\varphi)^*=-\otimes_{F(R)}F(S):\dd_{F(R)}\to\dd_{F(S)}$ is semiseparable. We show that $\varphi^*=-\otimes_RS:\cc_R\to\cc_S$ is semiseparable. By Theorem \ref{thm:phi*semisep}, there exists an $F(R)$-bimodule morphism $E':F(S)\to F(R)$ in $\dd$ such that $F(\varphi) E'F(\varphi)=F(\varphi)$. Since $F$ is full, $E'=F(E)$ for some $E:S\to R$ in $\cc$. Hence, $F(\varphi  E \varphi)=F(\varphi)$, which implies $\varphi E\varphi=\varphi$ because $F$ is faithful. It remains to show that $E$ is an $R$-bimodule morphism in $\cc$. Because $E'=F(E)$ is a right $F(R)$-module morphism, which means $F(E) m_{F(S)} (\id_{F(S)}\otimes F(\varphi))= m_{F(R)} (F(E)\otimes\id_{F(R)})$, we have $F(E m_S (\id_S\otimes\varphi))\phi^2_{S,R}=F(m_R (E\otimes\id_R))\phi^2_{S,R}$. Since $\phi^2$ is epimorphism on components and $F$ is faithful, we get $E m_S (\id_S\otimes\varphi)=m_R (E\otimes\id_R)$. This means that $E$ is a right $R$-module morphism. Similarly, one can show that  $E$ is also  left $R$-linear. Separable and naturally full cases follow in a similar way.
\end{proof}

Next, we prove results that also hold true without the assumption that $1$ is a left $\otimes$-generator in $\cc$. The first one is the special case of induction functors for algebras in a monoidal category, which is an immediate consequence of Proposition \ref{prop:algebra-induct-functor}.

\begin{prop}\label{lax monoidal variation unit}
Let $(\cc, \otimes, 1)$, $(\dd, \otimes, 1')$ be monoidal categories, $A$ be an algebra in $\cc$ with unit $u_A:1\to A$. Let $(F:\cc\to\dd,\phi^2,\phi^0)$ be a lax monoidal functor. 
\begin{itemize}
\item[$i)$] If  $\phi^0$ is a split-epi and $-\otimes A:\cc\to\cc_A$ is semiseparable (resp., naturally full), then $-\otimes F(A):\dd \to \dd_{F(A)}$ is semiseparable (resp.,  naturally full).

\item[$ii)$] If $\phi^0$ is a split-mono and $-\otimes A:\cc\to\cc_A$ is separable, then $-\otimes F(A):\dd \to \dd_{F(A)}$ is separable. 
\end{itemize}
\end{prop}

\begin{proof}
$i)$. If $-\otimes A:\cc\to\cc_A$ is semiseparable, then by Proposition \ref{prop:algebra-induct-functor}, there is a morphism $E:A \to 1$ in $\cc$ such that $u_A E u_A=u_A$. It follows that $F(u_A) F(E) F(u_A)=F(u_A)$. Since there is a morphism $\psi^0:F(1) \to 1'$ in $\dd$ satisfying $\phi^0 \psi^0 =\id_{F(1)}$, we get $F(u_A) \phi^0 \psi^0 F(E) F(u_A) \phi^0=F(u_A) \phi^0$, which means $u_{F(A)} \psi^0 F(E) u_{F(A)} = u_{F(A)}$. By Proposition \ref{prop:algebra-induct-functor} again, we obtain that $-\otimes F(A):\dd \to \dd_{F(A)}$ is semiseparable. The naturally full case follows similarly.
 
\begin{invisible}
Now, suppose $-\otimes A:\cc\to\cc_A$ is naturally full. By Proposition \ref{prop:algebra-induct-functor}, there is a morphism $E:A \to 1$ in $\cc$ such that $u_A E =\id_A$. Thus, $F(u_A) F(E) = \id_{F(A)}$. Hence, $F(u_A) \phi^0 \psi^0 F(E)= \id_{F(A)}$, which means $u_{F(A)} \psi^0 F(E) = \id_{F(A)}$. By Proposition \ref{prop:algebra-induct-functor}, we obtain that $-\otimes F(A):\dd \to \dd_{F(A)}$ is naturally full. 
\end{invisible}

$ii)$.   If $-\otimes A:\cc\to\cc_A$ is separable, then by Proposition \ref{prop:algebra-induct-functor} there is a morphism $E:A \to 1$ in $\cc$ such that $E u_A=\id_1$. It follows that $ F(E) F(u_A)=\id_{F(1)}$. Since there is a morphism $\psi^0:F(1) \to 1'$ in $\dd$ satisfying $\psi^0 \phi^0 =\id_{1'}$, we have $\psi^0 F(E) u_{F(A)}= \psi^0 F(E)F(u_A)\phi^0=\psi^0\phi^0 = \id_{1'}$. By Proposition \ref{prop:algebra-induct-functor} $-\otimes F(A)$ is separable.
\end{proof}

We now provide a converse of the previous result.
\begin{prop}
 \label{lax-monoidal-separable}
Let $(\cc, \otimes, 1)$, $(\dd, \otimes, 1')$ be monoidal categories, $A$ be an algebra in $\cc$ with unit $u_A:1\to A$. Let $(F:\cc\to\dd,\phi^2,\phi^0)$ be a fully faithful lax monoidal functor. If $\phi^0$ is an epimorphism and $-\otimes F(A):\dd \to \dd_{F(A)}$ is semiseparable (resp., separable, naturally full), then $-\otimes A:\cc\to\cc_A$ is semiseparable (resp., separable, naturally full).
\end{prop}

\begin{proof} If $-\otimes F(A)$ is semiseparable, by Proposition \ref{prop:algebra-induct-functor}, we know that $u_{F(A)} E' u_{F(A)} = u_{F(A)}$ for some morphism $E':F(A) \to 1'$ in $\dd$. Since $F$ is full, there is a morphism $E:A \to 1$ in $\cc$ such that $F(E) = \phi^0 E'$. This implies $F(u_A) F(E) F(u_A) \phi^0=F(u_A) \phi^0$. Because $\phi^0$ is an epimorphism and $F$ is faithful, we obtain $u_A E u_A=u_A$. As a result, by Proposition \ref{prop:algebra-induct-functor}, $-\otimes A:\cc\to\cc_A$ is semiseparable. The naturally full case follows similarly.

If $-\otimes F(A)$ is separable, by Proposition \ref{prop:algebra-induct-functor}, we know that $E'F(u_A)\phi^0 = E' u_{F(A)} = \id_{1'}$ for some morphism $E':F(A) \to 1'$ in $\dd$. Then, $\phi^0 E'F(u_A)\phi^0 =\phi^0$. Since $\phi^0$ is an epimorphism, we have $\phi^0 E'F(u_A)=\id_{F(1)}$. Since $F$ is full, there is a morphism $E:A \to 1$ in $\cc$ such that $F(E) = \phi^0 E'$. Thus, $F(E) F(u_A)=\id_{F(1)}$. Because $F$ is faithful, we obtain that $ E u_A=\id_1$. As a result, by Proposition \ref{prop:algebra-induct-functor}, $-\otimes A:\cc\to\cc_A$ is separable.
\begin{invisible}
For the naturally full case we have: suppose $u_{F(A)} E' = F(u_{A}) \phi^0 E'= \id_{F(A)}$ for some morphism $E':F(A) \to 1'$ in $\dd$.  Since $F$ is full, there is a morphism $E:A \to 1$ such that $F(E) = \phi^0 E'$. Thus, $F(u_{A}) F(E)= F(u_A)\phi^0 E'=\id_{F(A)}$.    
\end{invisible}    
\end{proof}

Our next aim is to extend Proposition \ref{lax monoidal variation unit} and Proposition \ref{lax-monoidal-separable} to an arbitrary algebra morphism in $\cc$, which is based on a full, essentially surjective, strong monoidal functor that preserves coequalizers,  see Proposition \ref{prop:monoidalfunct} and Proposition \ref{prop:refl-semisep-ind}. First, we need the following lemma.

\begin{lem}\label{lem:PhiRS}
Let $(\cc, \otimes, 1)$, $(\dd, \otimes , 1')$ be  monoidal categories, $(F:\cc\to\dd,\phi^2,\phi^0)$ be a strong monoidal functor, $\varphi:R\to S$ be an algebra morphism in $\cc$. Assume that $F$ preserves coequalizers. Then, there is an isomorphism $\Phi_{X,S}^R:FX\otimes_{FR}FS \to F(X\otimes_R S)$, for every $X\in\cc_R$, such that $\Phi_{X,S}^R\circ q^{FR}_{FX,FS}=F(q^R_{X,S})\circ \phi_{X,S}^2$. Moreover, $\Phi_{X,S}^R$ is natural on $X$.
\end{lem}

\begin{proof}
Since $F$ preserves coequalizers, we have that
\begin{equation}\label{eq:coeq1}
\xymatrix@C=1.5cm{
F(X\otimes R\otimes S) \ar@<-1ex>[r]_-{F(\id_X\otimes \mu_S^R)}\ar@<1ex>[r]^-{F(\nu^R_{X}\otimes\id_Y)}
& F(X\otimes S)\ar[r]^-{F(q^R_{X,S})}
& F(X\otimes _RS)}
\end{equation}
is a coequalizer. From the naturality of $\phi^2$, we know that $F(\nu_X^R \otimes \id_S) = \phi^2_{X,S}\circ (F(\nu_X^R)\otimes \id_{FS}) \circ (\phi^2_{X\otimes R, S})^{-1}$ and $F(\id_X\otimes \mu_S^R) = \phi^2_{X,S}\circ (\id_{FX} \otimes F(\mu_S^R)) \circ (\phi^2_{X, R\otimes S})^{-1}$. Besides, by Lemma \ref{lem:module-monoidal}, $F(X)$ is a right $F(R)$-module with the module structure $F(\nu_X^R)\circ \phi^2_{X,R}$, and $F(S)$ is a left $F(R)$-module with the module structure $F(\mu_S^R)\circ \phi^2_{R,S}$. Therefore, we consider the following coequalizer:
\begin{equation}\label{eq:coeqXRS}
\xymatrix@C=1.5cm{
FX\otimes FR\otimes FS\ar@<-1ex>[r]_-{\id_{FX}\otimes F(\mu_S^R)\circ \phi^2_{R,S}}\ar@<1ex>[r]^-{F(\nu_X^R)\circ \phi^2_{X,R} \otimes \id_{FS}}
& FX\otimes FS\ar[r]^-{q^{FR}_{FX,FS}}
& FX\otimes_{FR} FS.}
\end{equation}
We have
$$
\begin{aligned}
&F(q^R_{X,S})\circ \phi_{X,S}^2\circ (F(\nu_X^R)\circ \phi^2_{X,R} \otimes \id_{FS})\\
= &F(q^R_{X,S})\circ \phi_{X,S}^2\circ (F(\nu_X^R)\otimes \id_{FS})\circ (\phi^2_{X\otimes R,S})^{-1} \circ \phi^2_{X\otimes R,S} \circ (\phi^2_{X,R} \otimes \id_{FS})\\
= &F(q^R_{X,S})\circ F(\nu_X^R \otimes \id_S) \circ \phi^2_{X,R\otimes S} \circ (\id_{FX} \otimes \phi^2_{R,S})\\
= &F(q^R_{X,S})\circ F(\id_{X}\otimes \mu_S^R) \circ \phi^2_{X,R\otimes S} \circ (\id_{FX} \otimes \phi^2_{R,S})\\
= &F(q^R_{X,S})\circ \phi^2_{X,S}\circ (\id_{FX} \otimes F(\mu_S^R)) \circ (\phi^2_{X, R\otimes S})^{-1} \circ \phi^2_{X,R\otimes S} \circ (\id_{FX} \otimes \phi^2_{R,S})\\
= &F(q^R_{X,S})\circ \phi_{X,S}^2\circ (\id_{FX} \otimes F(\mu_S^R)\phi^2_{R,S}).\\
\end{aligned}
$$
Then, by the universal property of  coequalizer \eqref{eq:coeqXRS} there is a unique morphism $\Phi_{X,S}^R:FX\otimes_{FR}FS \to F(X\otimes_R S)$ such that $\Phi_{X,S}^R\circ q^{FR}_{FX,FS}=F(q^R_{X,S})\circ \phi_{X,S}^2$. Similarly, by the universal property of coequalizer \eqref{eq:coeq1}, there is a unique  $\Psi_{X,S}^R: F(X\otimes_RS)\to FX\otimes_{FR} FS$ such that $\Psi_{X,S}^R \circ F(q^R_{X,S}) = q^{FR}_{FX,FS} \circ (\phi_{X,S}^2)^{-1}$.  It is routine to check that $\Psi^R_{X,S}$ and $\Phi^R_{X,S}$ are inverses of each other. 
\begin{invisible}
Note that 
$$
\Psi_{X,S}^R\circ \Phi_{X,S}^R\circ q^{FR}_{FX,FS} = \Psi_{X,S}^R\circ F(q^R_{X,S})\circ \phi_{X,S}^2 = q^{FR}_{FX,FS} \circ (\phi_{X,S}^2)^{-1} \circ \phi_{X,S}^2 = q^{FR}_{FX,FS}
$$
and
$$
\Phi_{X,S}^R\circ \Psi_{X,S}^R \circ F(q^R_{X,S}) = \Phi_{X,S}^R\circ q^{FR}_{FX,FS} \circ (\phi_{X,S}^2)^{-1} = F(q^R_{X,S})\circ \phi_{X,S}^2 \circ (\phi_{X,S}^2)^{-1} = F(q^R_{X,S}).
$$
Since $q^{FR}_{FX,FS}$ and $F(q^R_{X,S})$ are epimorphisms, we obtain that $\Phi_{X,S}^R$ is an isomorphism with inverse $\Psi_{X,S}^R$. 
\end{invisible} 
We show the naturality of $\Phi^R_{X,S}$ on $X$. For any morphism $f:X\to Y$ in $\cc_R$, we have 
\[
\begin{split}
&F(f\otimes_RS)\circ\Phi^R_{X,S}\circ q^{FR}_{FX,FS}=F(f\otimes_RS)\circ F(q^R_{X,S})\circ\phi^2_{X,S}=F((f\otimes_RS) q^R_{X,S})\circ\phi^2_{X,S}\\=&F(q^R_{Y,S}(f\otimes S))\circ\phi^2_{X.S}=F(q^R_{Y,S})\circ F(f\otimes S)\circ\phi^2_{X,S}=F(q^R_{Y,S})\circ\phi^2_{Y,S}\circ (Ff\otimes FS)\\=&\Phi^R_{Y,S}\circ q^{FR}_{FY,FS}\circ (Ff\otimes FS)=\Phi^R_{Y,S}\circ (Ff\otimes_{FR}FS)\circ q^{FR}_{FX,FS}.
\end{split}
\]
Hence, $F(f\otimes_RS)\circ\Phi^R_{X,S}=\Phi^R_{Y,S}\circ (Ff\otimes_{FR}FS)$, as $q^{FR}_{FX,FS}$ is an epimorphism.\qedhere
\end{proof}

\begin{prop}\label{prop:monoidalfunct}
Let $(\cc, \otimes, 1)$, $(\dd, \otimes, 1')$ be  monoidal categories, and $\varphi:R\to S$ be an algebra morphism in $\cc$ such that $\varphi^*=-\otimes_RS:\cc_R\to\cc_S$ is semiseparable (resp., separable, naturally full). Let $(F:\cc\to\dd,\phi^2,\phi^0)$ be a full, essentially surjective, strong monoidal functor, and assume that it preserves coequalizers. Then $F(\varphi)^*=-\otimes_{F(R)}F(S):\dd_{F(R)}\to\dd_{F(S)}$ is also semiseparable (resp., separable, naturally full).   
\end{prop}

\begin{proof}
Recall that the unit of the adjunction $(\varphi^*, \varphi_*)$ is given by $\eta_X:=q^R_{X,S}(\id_X\otimes u_S):X\to X\otimes_RS$. Considering the algebra morphism $F(\varphi)$ in $\cc$, the unit of the adjunction $(F(\varphi)^*, F(\varphi)_*)$ is given by 
$\eta'_{D}:=q^{FR}_{D,FS}(\id_{D}\otimes u_{FS}):D\to D\otimes_{FR} FS$, for every $D\in\dd_{FR}$. By the Rafael-type Theorem \ref{thm:rafael}, there exists a natural transformation $\nu:\varphi_*\varphi^*\to\id$ such that $\eta\nu\eta=\eta$. 


Essential surjectivity of $F$ means that, for every $D\in\dd$, there exist $X\in\cc$ and an isomorphism $\alpha_D:D\to FX$ in $\dd$. For every $D\in\dd_{FR}$ with  module structure $\nu^{FR}_D: D \otimes FR \to D$, it is easy to check that  $\alpha_D\nu^{FR}_D (\alpha^{-1}_D \otimes FR):FX\otimes FR \to FX$ is a right $FR$-module structure of $FX$ and $\alpha_D:D\to FX$ is in $\dd_{FR}$.  We define
\[
\nu'_D:D\otimes_{FR}FS\to D, \,\,\text{ by }\,\,\nu'_D=\alpha^{-1}_D\circ F(\nu_X)\circ\Phi^R_{X,S}\circ (\alpha_D\otimes_{FR} FS).
\]
We observe that the definition of $\nu'_D$ does not depend on the choice of $\alpha_D$. In fact, if there is another isomorphism $\alpha'_D:D\to FX'$ in $\dd$, then by the fullness of $F$ there is a morphism $h:X'\to X$ in $\cc$ such that $F(h)=\alpha_D (\alpha'_{D})^{-1}$. Thus, 
\begin{align*}
\nu'_D&=\alpha'^{-1}_D F(\nu_{X'})\Phi^R_{X',S} (\alpha'_D\otimes_{FR} FS)=\alpha^{-1}_D F(h) F(\nu_{X'})\Phi^R_{X',S}(F(h)^{-1}\alpha_D\otimes_{FR} FS)\\&=\alpha^{-1}_D F(\nu_{X}) F(h\otimes_{FR}FS )\Phi^R_{X',S} (F(h)^{-1}\alpha_D\otimes_{FR} FS)= \alpha^{-1}_D F(\nu_{X}) \Phi^R_{X,S} (\alpha_D\otimes_{FR} FS).
\end{align*}

The naturality of $\nu'$ follows from the next commutative diagram, for every $f:D\to D'$ in $\dd_{FR}$:
\[
\xymatrixcolsep{3.5pc}\xymatrix{D\otimes_{FR}FS\ar[r]^-{\alpha_D\otimes_{FR}FS}\ar[d]_-{f\otimes_{FR}FS}&FX\otimes_{FR}FS\ar[d]^-{(\alpha_{D'}f\alpha_D^{-1})\otimes_{FR}FS}\ar[r]^-{\Phi^R_{X,S}} &F(X\otimes_RS)\ar[d]^-{F(h\otimes_RS)}\ar[r]^-{F(\nu_X)}&FX\ar[d]^-{\alpha_{D'}f\alpha_{D}^{-1}}\ar[r]^-{\alpha_D^{-1}}&D\ar[d]^-{f}\\ D'\otimes_{FR} FS\ar[r]_-{\alpha_{D'}\otimes_{FR}FS}&FX'\otimes_{FR}FS\ar[r]_-{\Phi^R_{X',S}}&F(X'\otimes_RS)\ar[r]_-{F(\nu_{X'})}&FX'\ar[r]_-{\alpha^{-1}_{D'}}&D',}
\]
where $h:X\to X'$ is the morphism in $\cc$ such that $Fh=\alpha_{D'}f\alpha_D^{-1}$, the second square is commutative by naturality of $\Phi^R_{-,S}$ (Lemma \ref{lem:PhiRS}) and the third one is commutative by naturality of $\nu$. 
By Rafael-type Theorem \ref{thm:rafael}, it remains to show that $\eta'\nu'\eta'=\eta'$. Note that $\phi^2_{X,1}(\id_{FX}\otimes\phi^0)=\id_{FX}$ and $\phi^2_{1,X}(\phi^0\otimes \id_{FX})=\id_{FX}$. For every $D\cong FX\in\dd_{FR}$, we have
\begin{align*}
 &\eta'_D\nu'_D\eta'_D= q^{FR}_{D,FS}(\id_{D}\otimes u_{FS})\alpha^{-1}_D F(\nu_X)\Phi^R_{X,S} (\alpha_D\otimes_{FR} FS)q^{FR}_{D,FS}(\id_{D}\otimes F(u_S)\phi^0)\\=& q^{FR}_{D,FS}(\id_{D}\otimes u_{FS})\alpha^{-1}_D F(\nu_X)\Phi^R_{X,S} q^{FR}_{FX,FS}(\alpha_D\otimes \id_{FS})(\id_D\otimes F(u_S)\phi^0 )
 \\=& q^{FR}_{D,FS}(\id_{D}\otimes u_{FS})\alpha^{-1}_D F(\nu_X)F(q^{R}_{X,S})\phi^2_{X,S}(\id_{FX}\otimes F(u_S)\phi^0)(\alpha_D\otimes \id_{1'} )\\=&q^{FR}_{D,FS}(\id_{D}\otimes u_{FS})\alpha^{-1}_D F(\nu_X)F(q^{R}_{X,S}) F(\id_X\otimes u_S)\phi^2_{X,1}(\id_{FX}\otimes \phi^0)(\alpha_D\otimes \id_{1'} )\\=&q^{FR}_{D,FS}(\id_{D}\otimes u_{FS})\alpha^{-1}_D F(\nu_X)F(\eta_X)(\alpha_D\otimes \id_{1'})\\=&q^{FR}_{D,FS}(\alpha^{-1}_{D}\alpha_D\otimes F(u_S)\phi^0)\alpha^{-1}_D F(\nu_X)F(\eta_X)(\alpha_D\otimes \id_{1'} )\\=&q^{FR}_{D,FS}(\alpha^{-1}_{D}\otimes \id_{FS})(\id_{FX}\otimes F(u_S))(\alpha_D \otimes\phi^0)\alpha^{-1}_DF(\nu_X)F(\eta_X)(\alpha_D\otimes \id_{1'} )\\=&q^{FR}_{D,FS}(\alpha^{-1}_{D}\otimes \id_{FS})(\phi^2_{X,S})^{-1}F(\id_X\otimes u_S)\phi^2_{X,1}(\id_{FX}\otimes\phi^0)F(\nu_X)F(\eta_X)(\alpha_D\otimes \id_{1'} )\\=&q^{FR}_{D,FS}(\alpha^{-1}_{D}\otimes \id_{FS})(\phi^2_{X,S})^{-1}F(\id_X\otimes u_S)F(\nu_X)F(\eta_X)(\alpha_D\otimes \id_{1'} )\\=&
(\alpha^{-1}_D\otimes_{FR}FS)q^{FR}_{FX,FS}(\phi^2_{X,S})^{-1}F(\id_X\otimes u_S)F(\nu_X)F(\eta_X)(\alpha_D\otimes \id_{1'} )\\=&
(\alpha^{-1}_D\otimes_{FR}FS)
(\Phi^R_{X,S})^{-1}F(q^R_{X,S})F(\id_X\otimes u_S)F(\nu_X)F(\eta_X)(\alpha_D\otimes \id_{1'} )\\=&
(\alpha^{-1}_D\otimes_{FR}FS)
(\Phi^R_{X,S})^{-1}F(q^R_{X,S})F(\id_X\otimes u_S)(\alpha_D\otimes \id_{1'} )\\=&
(\alpha^{-1}_D\otimes_{FR}FS)
q^{FR}_{FX,FS}(\phi^2_{X,S})^{-1}F(\id_X\otimes u_S)(\alpha_D\otimes \id_{1'} )\\=&(\alpha^{-1}_D\otimes_{FR}FS)
q^{FR}_{FX,FS}(\id_{FX}\otimes Fu_S)(\phi^2_{X,1})^{-1}(\alpha_D\otimes \id_{1'} )\\=&(\alpha^{-1}_D\otimes_{FR}FS)
q^{FR}_{FX,FS}(\id_{FX}\otimes Fu_S)(\id_{FX}\otimes\phi^0)(\alpha_D\otimes \id_{1'} )\\=&q^{FR}_{D,FS}(\alpha^{-1}_{D}\otimes \id_{FS})(\id_{FX}\otimes u_{FS})(\alpha_D\otimes \id_{1'} )=q^{FR}_{D,FS}(\id_{D}\otimes u_{FS})=\eta'_D.
\end{align*}
For the separable case, we have
\begin{align*}
\nu'_D\eta'_D= &
\alpha^{-1}_D F(\nu_X)F(q^{R}_{X,S}) F(\id_X\otimes u_S)\phi^2_{X,1}(\id_{FX}\otimes \phi^0)(\alpha_D\otimes \id_{1'} )\\=&\alpha^{-1}_D F(\nu_X)F(\eta_X)(\alpha_D\otimes \id_{1'})=\alpha^{-1}_D F(\nu_X\eta_X)(\alpha_D\otimes \id_{1'})=\alpha_D^{-1}\alpha_D=\id_D.
\end{align*}
By \cite[Theorem 1.2]{Raf90}, $F(\varphi)^*$ is separable. For the naturally full case, we have
\begin{align*}
&\eta'_D\nu'_D q^{FR}_{D,FS}= q^{FR}_{D,FS}(\id_{D}\otimes F(u_S)\phi^0)\alpha^{-1}_D F(\nu_X)\Phi^R_{X,S} (\alpha_D\otimes_{FR} FS)q^{FR}_{D,FS}\\=& q^{FR}_{D,FS}(\id_{D}\otimes F(u_S)\phi^0)\alpha^{-1}_D F(\nu_X)\Phi^R_{X,S} q^{FR}_{FX,FS}(\alpha_D\otimes \id_{FS})
 \\=&q^{FR}_{D,FS}(\id_{D}\otimes F(u_S)\phi^0)\alpha^{-1}_DF(\nu_X)F(q^{R}_{X,S})\phi^2_{X,S}(\alpha_D\otimes \id_{FS} )\\=&q^{FR}_{D,FS}(\alpha^{-1}_{D}\otimes F(u_S)\phi^0)F(\nu_X)F(q^{R}_{X,S})\phi^2_{X,S} (\alpha_D\otimes \id_{FS} )
 \\=&q^{FR}_{D,FS}(\alpha^{-1}_{D}\otimes \id_{FS})(\phi^2_{X,S})^{-1}F(\id_X\otimes u_S)\phi^2_{X,1}(\id_{FX}\otimes\phi^0)F(\nu_X)F(q^{R}_{X,S}) \phi^2_{X,S}(\alpha_D\otimes \id_{FS} )\\=&q^{FR}_{D,FS}(\alpha^{-1}_{D}\otimes \id_{FS})(\phi^2_{X,S})^{-1}F(\id_X\otimes u_S)F(\nu_X)F(q^{R}_{X,S}) \phi^2_{X,S}(\alpha_D\otimes \id_{FS} )\\=&
(\alpha^{-1}_D\otimes_{FR}FS)q^{FR}_{FX,FS}(\phi^2_{X,S})^{-1}F(\id_X\otimes u_S)F(\nu_X)F(q^{R}_{X,S}) \phi^2_{X,S}(\alpha_D\otimes \id_{FS} )\\=&
(\alpha^{-1}_D\otimes_{FR}FS)
(\Phi^R_{X,S})^{-1}F(q^R_{X,S})F(\id_X\otimes u_S)F(\nu_X)F(q^{R}_{X,S}) \phi^2_{X,S}(\alpha_D\otimes \id_{FS} )\\=&
(\alpha^{-1}_D\otimes_{FR}FS)
(\Phi^R_{X,S})^{-1}F(\eta_X\nu_X)F(q^{R}_{X,S}) \phi^2_{X,S}(\alpha_D\otimes \id_{FS} )\\=&
(\alpha^{-1}_D\otimes_{FR}FS)
(\Phi^R_{X,S})^{-1}F(q^{R}_{X,S}) \phi^2_{X,S}(\alpha_D\otimes \id_{FS} )\\=&(\alpha^{-1}_D\otimes_{FR}FS)
q^{FR}_{FX,FS} (\alpha_D\otimes \id_{FS} )=q^{FR}_{D,FS}(\alpha_D^{-1}\otimes\id_{FS})(\alpha_D\otimes\id_{FS})=q^{FR}_{D,FS}.
\end{align*}
Since $q^{FR}_{D,FS}$ is an epimorphism, then $\eta'_D\circ\nu'_D=\id_{D\otimes_{FR}FS}$. We conclude by \cite[Theorem 2.6]{ACMM06}.
\end{proof}
Replacing the essential surjectivity of $F$ by faithfulness, we can prove the converse of Proposition \ref{prop:monoidalfunct}.
\begin{prop}\label{prop:refl-semisep-ind}
Let $(\cc, \otimes, 1)$, $(\dd, \otimes, 1')$ be  monoidal categories  
and $\varphi:R\to S$ be an algebra morphism in $\cc$. Let $(F:\cc\to\dd,\phi^2,\phi^0)$ be a fully faithful, strong monoidal functor, and assume that it preserves coequalizers.
If $F(\varphi)^*=-\otimes_{F(R)}F(S):\dd_{F(R)}\to\dd_{F(S)}$ is semiseparable (resp., separable, naturally full), then  $\varphi^*=-\otimes_RS:\cc_R\to\cc_S$ is semiseparable (resp., separable, naturally full).  
\end{prop}

\begin{proof}
Suppose $F(\varphi)^*=-\otimes_{F(R)}F(S):\dd_{F(R)}\to\dd_{F(S)}$ is semiseparable (resp., separable, naturally full). There exists a natural transformation $\nu' = (\nu'_D :D \otimes_{FR} FS \to D)_{D\in \dd}$ such that $\eta' \nu' \eta' = \eta'$ (resp., $\nu' \eta' = \id$, $\eta' \nu' = \id$). Because $F$ is full, there exists a morphism $\nu_X: X\otimes_R S \to X$ such that $F(\nu_X) = \nu'_{FX} (\Phi^R_{X,S})^{-1}$, where $\Phi_{X,S}^R:FX\otimes_{FR}FS \to F(X\otimes_R S)$ is the isomorphism as in Lemma \ref{lem:PhiRS}. Since, for any  morphism $f:X \to Y$ in $\cc_R$,  
$$
\begin{aligned}
F(\nu_Y)F(f\otimes_R S) = &\nu'_{FY} (\Phi^R_{Y,S})^{-1} F(f\otimes_R S) = \nu'_{FY} (Ff\otimes_{FR} FS) (\Phi^R_{X,S})^{-1}\\ = &F(f) \nu'_{FX}(\Phi^R_{X,S})^{-1} = F(f) F(\nu_X),
\end{aligned}
$$
we obtain the naturality of $\nu$ as $F$ is faithful. Since $u_{FS}= F(u_S) \phi^0$, we have
$$
\begin{aligned}
F(\eta_X) = F(q^R_{X,S})F(\id_X \otimes u_S) \phi^2_{X,1} (\phi^2_{X,1})^{-1} = F(q^R_{X,S}) \phi^2_{X,S} (F(\id_X) \otimes F(u_S))(\phi^2_{X,1})^{-1}\\
= \Phi^R_{X,S} q^{FR}_{FX,FS} (F(\id_X) \otimes F(u_S))(\id_{FX}\otimes \phi^0) = \Phi^R_{X,S} q^{FR}_{FX,FS} (F(\id_X) \otimes u_{FS}).
\end{aligned}
$$
Therefore,
$$
\begin{aligned}
F(\eta_X)F(\nu_X)F(\eta_X)= 
&\Phi^R_{X,S} q^{FR}_{FX,FS} (F(\id_X) \otimes u_{FS})\nu'_{FX} (\Phi^R_{X,S})^{-1} \Phi^R_{X,S} q^{FR}_{FX,FS} (F(\id_X) \otimes u_{FS})\\
= &\Phi^R_{X,S} q^{FR}_{FX,FS} (F(\id_X) \otimes u_{FS})\nu'_{FX} q^{FR}_{FX,FS} (F(\id_X) \otimes u_{FS})\\
=& \Phi^R_{X,S} \eta'_{FX}\nu'_{FX} \eta'_{FX}=\Phi^R_{X,S} \eta'_{FX}=\Phi^R_{X,S} q^{FR}_{FX,FS} (F(\id_X) \otimes u_{FS})= F(\eta_X).
\end{aligned}
$$
As $F$ is faithful, we obtain 
$\eta_X \nu_X \eta_X = \eta_X$. Therefore, $\varphi^*=-\otimes_RS:\cc_R\to\cc_S$ is semiseparable. Similarly, one can prove the separable and naturally full cases.
\end{proof}


Now, we provide some functors which satisfy the conditions in Proposition \ref{prop:monoidalfunct}.

\begin{es}\label{es:serre}
  Let $(\mathcal{A}, \otimes, 1)$ be a semisimple abelian monoidal category (see \cite[Definition 1.5.1]{Eti15}) with biexact tensor product. 
Recall that a full additive subcategory $\cc\subset \mathcal{A}$ is a \emph{Serre subcategory} provided that $\cc$ is closed under taking subobjects, quotients and extensions. A Serre subcategory $\cc$ of $\mathcal{A}$ is called a \emph{two-sided Serre tensor-ideal} of $\mathcal{A}$ if for any $X\in\mathcal{A}$, $Y\in\cc$, we have $X\otimes Y\in\cc$ and $Y\otimes X\in\cc$.

Consider a two-sided Serre ideal $\cc$ of $\mathcal{A}$ and the quotient category $\mathcal{A}/\mathcal{C}$, defined as in \cite[Chapter III, Section 1]{G62}. It is well-known that the canonical quotient functor $T:\mathcal{A}\to \mathcal{A}/\cc$ is exact, see \cite[Chapter III, Section 1, Proposition 1]{G62}. Hence $T$ preserves coequalizers. By \cite[Proposition 4.8]{ZL24}, the quotient category $\mathcal{A}/\cc$ is monoidal and $T$ is a  monoidal functor. By definition of the quotient functor, $T(M) = M$ for any object $M$ in $\mathcal{A}$. It follows that $T$ is essentially surjective. 

Besides, recall that for any morphism $\bar{f}:M\to N$ in $\mathcal{A}/\cc$, $\bar{f}$ can be written as $\bar{f}=T(f)$ for some $f:M'\to N/N'$, where $M', N'$ are subobjects of $M,N$ respectively. In fact, $T(f)= T(q)\bar{f}T(i)$ for a monomorphism $i:M'\to M$ and an epimorphism $q:N\to N/N'$, and $T(i), T(q)$ are isomorphisms, see \cite[Chapter III, Section 1, Proposition 1]{G62}. Since $\mathcal{A}$ is semisimple, $M'$ and $N/N'$ are direct summands of $M$ and $N$, respectively. Consequently, there are an epimorphism $p:M\to M'$ and a monomorphism $j:N/N' \to N$ such that $pi=\id_{M'}$ and $qj=\id_{N/N'}$. Because $T(i)$ and $T(q)$ are isomorphisms in $\mathcal{A}/\cc$, we have $T(i)^{-1}=T(p)$ and $T(q)^{-1}=T(j)$. This means $\bar{f}=T(j)T(f)T(p) = T(jfp)$. For this reason, $T$ is full. 
\end{es}

\begin{es}\label{es:coidentifier}
   Let $(\cc, \otimes, 1)$ be a monoidal category with coequalizers and assume that every object is coflat. Let $e:\id_\cc\to \id_\cc$ be an idempotent natural transformation such that $e_{X\otimes Y}=e_X\otimes e_Y$, for every $X,Y\in\cc$ and let $\cc_e$ be the coidentifier category defined in \cite[Example 17]{FOPTST99} by $\mathrm{Ob}\left( \cc_{e}\right) =\mathrm{Ob}\left( \cc\right) $ and $\Hom_{\cc_{e}}\left( X,Y\right) =\Hom_{\cc}\left( X,Y\right) /\hspace{-2pt}\sim $, where the congruence relation $\sim$ on the hom-sets is given, for all $f,g:X\to Y$ in $\cc$, by setting $f\sim g$ if and only if $e_{Y}\circ f=e_{Y}\circ g$.
   
For any morphism $f\in \Hom_{\cc%
}\left( X,Y\right)$, we denote by $\overline{f}$ the class of $f$ in $\Hom_{\cc_{e}}\left( X,Y\right) $.
The quotient functor $H:\cc\rightarrow \cc_{e}$
acts as the identity on objects and as the canonical projection on
morphisms. By \cite[1.5]{AB22} we know that $H$ is essentially surjective and naturally full with respect to $\mathcal{P}_{X,Y}:\Hom_{%
\cc_{e}}\left( X,Y\right)\to \Hom_{\cc}\left(
X,Y\right)$ defined by  $\mathcal{P}_{X,Y}(\overline{f})=e_Y\circ f$. 

We observe that $(\cc_e, \otimes, 1)$ is a monoidal category. More precisely, the tensor product of objects in $(\cc_e, \otimes, 1)$ is defined by the same object as in $\cc$, and on morphisms by $\overline{f}\otimes\overline{g}:=H(f\otimes g)$. It is well-defined as, for $f,f':X\to X'$, $g,g':Y\to Y'$ in $\cc$ such that $\overline{f}=\overline{f'}$ and $\overline{g}=\overline{g'}$, we have $e_{X'\otimes Y'}(f\otimes g)=e_{X'}f\otimes e_{Y'}g=e_{X'}f'\otimes e_{Y'}g'=e_{X'\otimes Y'}(f'\otimes g')$. Therefore, $H(f\otimes g)=H(f'\otimes g')$. By definition, it is clear that $H$ is a strict monoidal functor.
   
We show that $H$ preserves coequalizers. Let $(Q,q)$ be the coequalizer of parallel morphisms $f,g:U\to V$ in $\cc$. We have $Hq\circ Hf=H(qf)=H(qg)=Hq\circ Hg$.
For any $\xi:HV\to HZ$ in $\cc_e$ such that $\xi \circ Hf=\xi\circ Hg$, we have $H(h)=\xi$, for some $h:V\to Z$ in $\cc$, so $H(h\circ f)=H(h\circ g)$. Hence, $e_Z\circ h\circ f=e_Z\circ h\circ g$. By the universal property of the coequalizer $(Q,q)$, there exists a unique $l:Q\to Z$ in $\cc$ such that $l\circ q=e_Z\circ h$, and then $Hl\circ Hq=Hh=\xi$. Suppose there is a morphism $t:Q \to Z$ in $\cc_e$ such that $Hl\circ Hq = t\circ Hq$. Since $H$ is full, there is a morphism $t':Q \to Z$ in $\cc$ such that $t=H(t')$. Thus, $H(lq)= H(t'q)$, which implies $e_Zlq = e_Zt'q$. Since $q$ is an epimorphism, we have $e_Zl = e_Zt'$. Consequently, $Hl= Ht' =t$. This means $Hl$ is the unique morphism in $\cc_e$ such that $Hl\circ Hq=\xi$.
\end{es}

\begin{es}
Let $(\mathcal{A},\otimes,1)$ be an abelian monoidal category with biexact tensor product. Recall \cite[Example 2.4]{XZ25}, the bounded derived category $(\textbf{D}^{b}(\mathcal{A}), \widetilde{\otimes}, 
1^{\bullet})$ is a monoidal triangulated category whose monoidal structure $\widetilde{\otimes}$ is inherited from the bounded homotopy category $\textbf{K}^{b}(\mathcal{A})$, which means the localization functor $Q:\textbf{K}^{b}(\mathcal{A}) \to \textbf{D}^{b}(\mathcal{A})$ is a strong monoidal functor. Besides, the canonical functor $\mathcal{A} \to \textbf{D}^{b}(\mathcal{A})$ is fully faithful, see e.g. \cite[1.5]{Kra7}, which implies the localization functor $Q$ is full. Therefore, $Q$ is a full, essentially surjective, strong monoidal functor. Besides,  $Q$ preserves finite colimits, see e.g.  \cite[\href{https://stacks.math.columbia.edu/tag/05Q2}{Lemma 05Q2}]{stacks-project}. Consequently, it preserves coequalizers.
\end{es}




\section{Coinduction   functors}\label{sect:coinduct}

Let $(C, \Delta,\varepsilon)$ be a coalgebra in a monoidal category $(\cc, \otimes, 1)$ with equalizers. We recall from e.g. \cite{A08} (see also \cite[Definition 2.2.1]{A97}) that, given a right $C$-comodule $(V, \rho^C_V)$  and a left $C$-comodule $(W,\lambda^C_W)$ in $\cc$, their cotensor product over $C$ in $\cc$ is defined to be the equalizer $(V\square_C W, e^C_{V,W})$ of the pair of morphisms $(\rho^C_V\otimes W, V\otimes \lambda^C_W)$: 
\begin{equation}\label{eq:cotens}
\xymatrix@C=1.5cm{V\square_CW\ar[r]^-{e^C_{V,W}}&V\otimes W \ar@<-1ex>[r]_-{\id_V\otimes \lambda^C_W}\ar@<1ex>[r]^-{\rho^C_{V}\otimes\id_W}& V\otimes C\otimes W.}
\end{equation}

For any left $C$-colinear morphism $f:X\to Y$ in $\cc$, let $V\square_Cf:V\square_C X\to V\square_C Y$ be the unique morphism in $\cc$ such that
\[
e^C_{V,Y}(V\square_C f)=(\id_V\otimes f) e^C_{V,X}.
\]
Let $g:M\to N$ in $\cc^C$ and $Y$ be in ${}^C\cc$. Then, $g\square_CY: M\square_CY\to N\square_CY$ is the unique morphism in $\cc$ such that
\[
e^C_{N,Y}(g\square_CY)=(g\otimes\id_Y)e^C_{M,Y}.
\]
We have the canonical (natural) isomorphisms $\Lambda_M$, $\Lambda_{M,X}$, $\Lambda'_{Y}$:
\begin{itemize}
    \item $\Lambda _M:M\to M\square_CC$, uniquely determined by the property $\rho^C_M= e^C_{M,C} \Lambda_M$;
    \item $\Lambda_{M,X}: M\otimes X\to M\square_C(C\otimes X)$, uniquely determined by the property $\rho^C_M\otimes\id_X=e^C_{M, C\otimes X}\Lambda_{M,X}$;
    \item $\Lambda'_Y: Y\to C\square_CY$, uniquely determined by the property $\lambda^C_Y=e^C_{C,Y}\Lambda'_Y$.
\end{itemize}
One can easily check that
\[
\Lambda_M^{-1}=(\id_M\otimes\varepsilon)e^C_{M,C}, \quad (\Lambda'_Y)^{-1}=(\varepsilon\otimes\id_Y)e^C_{C,Y}.
\]
\begin{invisible}
We show that $\Lambda_M^{-1}=(\id_M\otimes\varepsilon)e^C_{M,C}$. On one hand, $\Lambda_M^{-1} \circ \Lambda_M=(\id_M\otimes\varepsilon)e^C_{M,C} \Lambda_M =(\id_M\otimes\varepsilon)\rho^C_M = \id_M$. On the other hand, since $e^C_{M,C} \Lambda_M \Lambda_M^{-1} = \rho^C_M(\id_M\otimes\varepsilon)e^C_{M,C}= (\id_{M\otimes C}\otimes\varepsilon)(\rho^C_M\otimes \id_C)e^C_{M,C}= (\id_{M\otimes C}\otimes\varepsilon)(\id_M\otimes \Delta_C)e^C_{M,C} = e^C_{M,C}$ and $e^C_{M,C}$ is a monomorphism, we have $\Lambda_M \Lambda_M^{-1} = \id_{M\square_CC}$.  

Similarly we check that $(\Lambda'_Y)^{-1}=(\varepsilon\otimes\id_Y)e^C_{C,Y}$. We have $
e^C_{C,Y}\Lambda'_Y(\Lambda'_Y)^{-1}=e^C_{C,Y}\Lambda'_Y(\varepsilon_C\otimes\id_Y)e^C_{C,Y}=\lambda^C_Y(\varepsilon_C\otimes\id_Y)e^C_{C,Y}=(\varepsilon_C\otimes\id_C\otimes\id_Y)(\id_C\otimes\lambda^C_Y)e^C_{C,Y}=(\varepsilon_C\otimes\id_C\otimes\id_Y)(\Delta_C\otimes\id_Y)e^C_{C,Y}=e^C_{C,Y}.$

 Recall that a $(D,C)$-bicomodule $(V, \lambda^D_V, \rho^C_V)$ is an object $V$ in $\cc$ with left $D$-coaction $\lambda^D_V:V\to D\otimes V$ and right $C$-coaction $\rho^D_V:V\to V\otimes D$ such that 
\[
(\lambda^D_V\otimes \id_C)\rho^C_V=(\id_C\otimes \rho^C_V)\lambda^D_V.
\]
\end{invisible}
Let $C,D$ be coalgebras in $\cc$. We denote by ${}^D\cc^C$ the category of $(D,C)$-bicomodules. An object $X$ in $\cc$ is said to be \emph{right flat} (resp., \emph{left flat}) if $-\otimes X$ (resp., $X\otimes -$) preserves equalizers. 
We have:
\begin{itemize}
    \item[$i)$] if $C$ is right flat, then for every $V\in {}^D\cc^C$, $M\in\cc^D$ the morphism $\rho^C_{M\square_DV}$ is uniquely determined by $(e^D_{M,V}\otimes C)\rho^C_{M\square_DV}=(\id_M\otimes\rho^C_V)e^D_{M,V}$;
\item[$ii)$] if $C$ is left flat, then for every $V\in {}^D\!\cc$, $M\in {}^C\cc^D$, the morphism $\lambda^C_{M\square_DV}:M\square_DV\to C\otimes (M\square_DV)$ is uniquely determined by 
$(\id_C\otimes e^D_{M,V})\lambda^C_{M\square_DV}=(\lambda^C_M\otimes\id_V)e^D_{M,V}$.
\end{itemize}

Let $\psi:C\to D$ be a coalgebra morphism in $\cc$. Consider the \emph{corestriction of coscalars functor} $\psi_*:\cc^C\to\cc^D$ and the \emph{coinduction functor} 
\[
\begin{split}
\psi^* :=-\square_DC: \cc^D\to \cc^C, \quad (M, \rho^D_M)\mapsto (M\square_DC, \rho^C_{M\square_DC}), \quad f\mapsto  f\square_DC,
\end{split}
\]
where $\rho^C_{M\square_DC}:M\square_DC\to (M\square_DC)\otimes C$ is uniquely determined by 
\[
(e^D_{M,C}\otimes\id_C)\rho^C_{M\square_DC}=(\id_M\otimes \Delta_C)e^D_{M,C}.
\]
%
\begin{invisible}
We show that the functor $-\square_DC$ is well defined. For every $f:M\to N \in \cc^D$ we have 
\[
\begin{split}
(e^D_{N,C}\otimes \id_C)&\rho^C_{N\square_DC}(f\square_DC)=(\rho^C_N\otimes\id_C)e^D_{N,C}(f\square_DC)=(\rho^C_N\otimes\id_C)(f\otimes\id_C)e^D_{M,C}\\&=(f\otimes\id_C\otimes\id_C)(\rho^C_M\otimes\id_C)e^D_{M,C}=(f\otimes\id_C\otimes\id_C)(e^C_{M,Y}\otimes\id_C)\rho^C_{M\square_DC}\\&=(e^D_{N,C}\otimes\id_C)((f\square_DC)\otimes\id_C)\rho^C_{M\square_DC},
\end{split}
\]
hence $\rho^C_{N\square_DC}(f\square_DC)=((f\square_DC)\otimes\id_C)\rho^C_{M\square_DC}$, as $e^{D}_{N,C}$ is a monomorphism. Thus, $f\square_DC \in \cc^C$.

We check that $\rho^C_{M\square_DC}$ is a coaction. In fact, we have
\[
\begin{split}
(e^D_{M,C}\otimes \mathrm{Id}_C\otimes\id_C)(\id_{M\square_DC}\otimes\Delta_C)\rho^C_{M\square_D C}&=(\id_{M}\otimes\id_{C}\otimes \Delta_C)(e^D_{M,C}\otimes\id_C)\rho^C_{M\square_DC}\\=(\id_M\otimes \id_C\otimes\Delta_C)(\id_M\otimes\Delta_C)e^D_{M,C}&=(\id_M\otimes \Delta_C\otimes\id_C)(\id_M\otimes\Delta_C)e^D_{M,C}\\=
(\id_M\otimes \Delta_C\otimes\id_C)(e^D_{M,C}\otimes\id_C)\rho^C_{M\square_DC}&=(e^D_{M,C}\otimes\mathrm{Id}_C\otimes\id_C)(\rho^C_{M\square_DC}\otimes\id_C)\rho^C_{M\square_DC},
\end{split}
\]
hence $(\id_{M\square_DC}\otimes\Delta_C)\rho^C_{M\square_D C}=(\rho^C_{M\square_DC}\otimes\id_C)\rho^C_{M\square_DC}$ since $e^D_{M,C}$ is a monomorphism, so $\rho^C_{M\square_DC}$. Moroever, $\rho^C_{M\square_DC}$ is counital as
\[
\begin{split}
e^D_{M,C}(\id_{M\square_DC}\otimes\varepsilon_C)\rho^C_{M\square_DC}&=(\id_M\otimes\id_C\otimes\varepsilon_C)(e^D_{M,C}\otimes\id_C)\rho^C_{M\square_DC}\\&=(\id_M\otimes\id_C\otimes\varepsilon_C)(\id_M\otimes\Delta_C)e^D_{M,C}=e^D_{M,C}
\end{split}
\]
and $e^D_{M,C}$ is a monomorphism.\end{invisible}
We have the adjunction $\psi_*\dashv \psi^*:\cc^D\to\cc^C$ with unit and counit given by
\[\eta_M:=\overline{\rho^C_M}:M\to M\square_DC,\quad \epsilon_N:=(\id_N\otimes\varepsilon_C)e^D_{N,C}:N\square_DC\to N,
\]
for every $(M,\rho^C_M)\in\cc^C$, $(N, \rho^D_N)\in\cc^D$,
where $\overline{\rho^C_M}$ is uniquely determined by $\rho^C_M=e^D_{M,C}\overline{\rho^C_M}$. \begin{invisible} as $(\id_M\otimes \lambda^D_C)\rho^C_M=(\id_M\otimes (\psi\otimes \id_C)\Delta_C)=(\id_M\otimes\psi\otimes\id_C)(\id_M\otimes\Delta_C)\rho^C_M=(\id_M\otimes\psi\otimes\id_C)(\rho^C_M\otimes\id_C)\rho^C_M=(\rho^D_M\otimes\id_C)\rho^C_M$. 

We show that $\eta$ and $\epsilon$ are natural transformations.
For every $f:M\to M$ in $\cc^C$, one has
\[
\begin{split}
e^D_{M', C}(f\square_DC)\eta_M&=(f\otimes\id_C)e^D_{M,C}\overline{\rho^C_M}=(f\otimes\id_C)\rho^C_M=\rho^C_{M'}f=e^D_{M',C}\overline{\rho^C_{M'}}f=e^D_{M',C}\eta_{M'}f,
\end{split}
\]
hence $(f\square_DC)\eta_M=\eta_{M'}f$ as $e^D_{M',C}$ is a monomorphism. For every $g:N\to N'$ in $\cc^D$, one has
\[
\begin{split}
\epsilon_{N'}(g\square_DC)=(\id_{N'}\otimes\varepsilon_C)e^D_{N',C}(g\square_DC)=(\id_{N'}\otimes\varepsilon_C)(g\otimes\id_C)e^D_{N,C}=g(\id_N\otimes\varepsilon_C)e^D_{N,C}=g\varepsilon_N.
\end{split}\]
Finally, we check that $\eta$ and $\epsilon$ fulfill the triangle identities. In fact, 
\[
\begin{split}
e^D_{N,C}(\epsilon_N\square_DC)\overline{\rho^C_{N\square_DC}}&=(\epsilon_N\otimes\id_C)e^D_{N\square_DC, C}\overline{\rho^C_{N\square_DC}}=(\epsilon_N\otimes\id_C)\rho^C_{N\square_DC}\\&=((\id_N\otimes\varepsilon_C)e^D_{N,C}\otimes\id_C)\rho^C_{N\square_DC}\\&=(\id_N\otimes\varepsilon_C\otimes\id_C)(\id_N\otimes\Delta_C)e^D_{N,C}=e^D_{N,C}
\end{split}
\]
so $(\epsilon_N\square_DC)\overline{\rho^C_{N\square_DC}}=\id_{N\square_DC}$ since $e^D_{N,C}$ is a monomorphism; for every $(M,\rho^C_M)\in \cc^C$ one has
\[
\begin{split}
\epsilon_{\psi_*(M,\rho^C_M)}\circ\psi_*\eta_{(M,\rho^C_M)}&=\epsilon_{(M,\rho^D_M)}\circ \psi_*(\overline{\rho^C_{(M,\rho^C_M)}})\\&=(\id_{(M,\rho^D_M)}\otimes\varepsilon_C)e^D_{(M,\rho^D_M),C}\circ\overline{\rho^C_{(M,\rho^D_M)}}\\&=(\id_{(M,\rho^D_M)}\otimes\varepsilon_C)\rho^C_{(M,\rho^D_M)}=\id_{(M,\rho^C_M)}.
\end{split}
\]
\end{invisible}

\begin{rmk}\label{rmk:comp-coinduct}
  By the dual of Remark \ref{rmk:comp-phi-psi},  if $\psi:C\to D$ and $\varphi:D\to E$ are coalgebra morphisms in $\cc$, then $\varphi_*\circ\psi_*\cong(\varphi\circ\psi)_*$.
   \begin{invisible}
   Indeed, $\varphi_*\circ\psi_*=(\varphi\circ \psi)_*$  is the left adjoint of $\psi^*\circ\varphi^*$ 
   by \cite[Proposition 3.2.1]{Bor94}.
Meanwhile, we have the adjunction $(\varphi\circ\psi)_*\dashv(\varphi\circ\psi)^*$, hence $\psi^*\circ\varphi^*\cong(\varphi\circ\psi)^*$. 
\end{invisible}
\end{rmk}

In this section, unless stated otherwise, we always assume that monoidal categories have  equalizers and their objects are flat.

\subsection{Semiseparability of coinduction functors}
Let $\psi:C\to D$ be a morphism of coalgebras over a field. 
Recall from \cite[Proposition 3.8]{AB22} that the coinduction functor $\psi^*=-\square_DC:\m^D\to\m^C$ is semiseparable if and only if $\psi$ is a regular morphism of $D$-bicomodules if and only if there is a $D$-bicomodule morphism $\chi:D\to C$ such that $\varepsilon_C\circ\chi\circ\psi=\varepsilon_C$. We extend this result to the monoidal setting. To this aim, we need the dual notion of left $\otimes$-generator in $\cc$. 
\begin{defn}\label{cogenerator}
    Let $\cc$ be a monoidal category. An object $Q$ of $\cc$ is a \emph{left $\otimes$-cogenerator} of $\cc$ if for any two morphisms $f,g:Y\to Z\otimes W$ in $\cc$ such that $(\xi\otimes\id_W)f=(\xi\otimes\id_W)g$, for every $\xi:Z\to Q$ in $\cc$, then $f=g$.  Besides, an object is said to be a \emph{two-sided $\otimes$-cogenerator} if it is both a left and right $\otimes$-cogenerator.
\end{defn}

The next lemma is the dual version of Lemma \ref{lem:bij}. 
\begin{lem}\label{lem:Dbicomd}
Let $\psi:C\to D$ be a morphism of coalgebras in a monoidal category $\cc$. Assume that the unit object $1$ of $\cc$ is a left $\otimes$-cogenerator. Then, there is a bijection
\[
\mathrm{Nat} (\id_{\cc^D},\psi_* \psi^*)\cong {}^D\mathrm{Hom}^D(D,C).
\]
\end{lem}

\begin{proof}
We only give a sketch of the proof. 
Define $\alpha:\mathrm{Nat} (\id_{\cc^D},\psi_* \psi^*)\to {}^D\mathrm{Hom}^D(D,C)$ by $\alpha(t)=(\Lambda'_C)^{-1}\circ t_D$. 
One can check that $\alpha(t)$ is a $D$-bicomodule morphism in $\cc$. To show that $\alpha(t)$ is in ${}^D\cc$, one uses the fact that $1$ is a left $\otimes$-cogenerator. Moreover, $\alpha$ is a bijection with inverse $\alpha'$ defined by
\[
\alpha'(\chi)=(\beta_Y:=(Y\square_D\chi)\Lambda_Y:Y\to Y\square_DC)_{Y\in\cc^D},
\]
for every $\chi$ in ${}^D\Hom^D(D,C)$.
\begin{invisible} Completed proof: 
Define $\alpha:\mathrm{Nat} (\id_{\cc^D},\psi_* \psi^*)\to {}^D\mathrm{Hom}^D(D,C)$ by $\alpha(t)=(\Lambda'_C)^{-1}\circ t_D$. First, we show that $\alpha(t)$ is a right $D$-comodule map in $\cc$. Since $t_D$ is in $\cc^D$, it suffices to show $(\Lambda'_C)^{-1}$ is a right $D$-comodule map. This can be obtained from the following
$$
\begin{aligned}
((\Lambda'_C)^{-1}\otimes \id_C) \rho^D_{D\square_DC} = (\varepsilon_D\otimes \id_C \otimes \id_C)(e^D_{D,C}\otimes \id_C)\rho^D_{D\square_DC}= \\
(\varepsilon_D\otimes \id_C \otimes \id_C)(\id_D\otimes \rho^D_C)e^D_{D,C} = \rho^D_C (\varepsilon_D\otimes \id_C)e^D_{D,C}= \rho^D_C (\Lambda'_C)^{-1}.
\end{aligned}
$$
Then, $\alpha(t)$ is in $\cc^D$.

Consider $\xi:D\to 1$ in $\cc$ and define $f_\xi=(\xi\otimes\id_D)\Delta_D$. We have
$
(f_\xi\otimes\id_D)\Delta_D=(\xi\otimes\id_D\otimes\id_D)(\Delta_D\otimes\id_D)\Delta_D=(\xi\otimes\id_D\otimes\id_D)(\id_D\otimes\Delta_D)\Delta_D
$, so $f_\xi$ is in $\cc^D$. By naturality of $t$ we have $t_Df_\xi=(f_\xi\square_DC)t_D$, which is also equivalent to $(\Lambda'_C)^{-1}t_Df_\xi=(\Lambda'_C)^{-1}(f_\xi\square_DC)t_D$. Since $(\Lambda'_C)^{-1}(f_\xi\square_DC)t_D=(\varepsilon_D\otimes\id_C)e^D_{D,C}(f_\xi\square_DC)t_D=(\varepsilon_D\otimes\id_C)(f_\xi\otimes\id_C)e^D_{D,C}t_D=(\varepsilon_D\otimes\id_C)(\xi\otimes\id_D\otimes\id_C)(\Delta_D\otimes\id_C)e^D_{D,C}t_D=(\xi\otimes\id_C)(\id_D\otimes\varepsilon_D\otimes\id_C)(\Delta_D\otimes\id_C)e^D_{D,C}t_D=(\xi\otimes\id_C)e^D_{D,C}t_D$, we have that
\[
(\xi\otimes\id_C)(\id_D\otimes\alpha(t))\Delta_D=(\xi\otimes\id_C)e^D_{D,C}t_D
\]
as $(\xi\otimes\id_C)(\id_D\otimes\alpha(t))\Delta_D=(\xi\otimes (\Lambda'_C)^{-1})\Delta_D=(\Lambda'_C)^{-1}t_D(\xi\otimes\id_D)\Delta_D$. Moreover, we have
\[
\begin{split}
(\psi\otimes\id_C)\Delta_C\alpha(t)&=(\psi\otimes\id_C)\Delta_C(\Lambda'_C)^{-1}t_D=(\psi\otimes\id_C)\Delta_C(\varepsilon_D\otimes\id_C)e^D_{D,C}t_D\\&=(\varepsilon_D\otimes\id_D\otimes\id_C)(\id_D\otimes(\psi\otimes\id_C)\Delta_C)e^D_{D,C}t_D\\&=(\varepsilon_D\otimes\id_D\otimes\id_C)(\Delta_D\otimes\id_C)e^D_{D,C}t_D=e^D_{D,C}t_D,
\end{split}\]
so we get
\[
(\xi\otimes\id_C)(\id_D\otimes\alpha(t))\Delta_D=(\xi\otimes\id_C)(\psi\otimes\id_C)\Delta_C\alpha(t)
\]
for all $\xi:D\to 1$. Since $1$ is a left $\otimes$-cogenerator in $\cc$, then $\alpha(t)$ is a left $D$-colinear morphism in $\cc$, hence a $D$-bicolinear morphism in $\cc$.
We show that $\alpha$ is a bijection with inverse $\alpha'$ which is defined by
\[
\alpha'(\chi)=(\beta_Y:=(Y\square_D\chi)\Lambda_Y:Y\to Y\square_DC)_{Y\in\cc^D},
\]
for every $\chi$ in ${}^D\Hom^D(D,C)$. We have that $\beta:=(\beta_Y)_{Y\in\cc^D}$ is completely determined by the property $e^D_{Y,C}\beta_Y=(\id_Y\otimes \chi)\rho^D_Y$ as $e^D_{Y,C}\beta_Y=e^D_{Y,C}(Y\square_D\chi)\Lambda_Y=(\id_Y\otimes\chi)e^D_{Y,D}\Lambda_Y=(\id_Y\otimes\chi)\rho^D_Y$, and on the other hand if $\beta'_Y$ satisfies $e^D_{Y,C}\beta'_Y=(\id_Y\otimes \chi)\rho^D_Y$ then from $e^D_{Y,C}\beta'_Y=(\id_Y\otimes \chi)\rho^D_Y=e^D_{Y,C}(Y\square_D\chi)\Lambda_Y$, we get $\beta'_Y=(Y\square_D\chi)\Lambda_Y$ as $e^D_{Y,C}$ is a monomorphism. We check that $\beta$ is a natural transformation. In fact, for every $f:X\to Y$ in $\cc^D$ we have
\[
\begin{split}
    e^D_{Y,C}\circ(f\square_DC)\circ\beta_X&=(f\otimes\id_C)e^D_{X,C}\beta_X=(f\otimes\id_C)(\id_X\otimes\chi)\rho^D_X=(\id_Y\otimes\chi)(f\otimes\id_C)\rho^D_X\\&=(\id_Y\otimes\chi)\rho^D_Y f=e^D_{Y,C}\circ\beta_Y\circ f,
\end{split}
\]
so that $(f\square_DC)\circ\beta_X=\beta_Y\circ f$ as $e^D_{Y,C}$ is a monomorphism. Thus, $\alpha'$ is well defined. Moreover, we have $\alpha'\alpha(t)=\alpha'((\Lambda'_C)^{-1}t_D)=(\beta_Y=(Y\square_D(\Lambda'_C)^{-1}t_D)\Lambda_Y)_{Y\in\cc^D}$, where $\beta_Y$ is characterized by $e^D_{Y,C}\beta_Y=(\id_Y\otimes\alpha(t))\rho^D_Y$. Consider an arbitrary morphism $\xi:Y\to 1$ in $\cc$, and define $g_\xi:=(\xi\otimes\id_D)\rho^D_Y: Y\to D$. Note that $g_\xi$ is right $D$-colinear as $(g_\xi\otimes\id_D)\rho^D_Y=(\xi\otimes\id_D\otimes\id_D)(\rho^D_Y\otimes\id_D)\rho^D_Y=(\xi\otimes\id_D\otimes\id_D)(\id_Y\otimes\Delta_D)\rho^D_Y=\Delta_D(\xi\otimes\id_D)\rho^D_Y=\Delta_Dg_\xi$.

By naturality of $t$ we have 
\[
\alpha(t)g_\xi=(\Lambda'_C)^{-1}(g_\xi\square_DC)t_Y=(\varepsilon_D\otimes\id_C)e^D_{D,C}(g_\xi\square_DC)t_Y=(\varepsilon_Dg_\xi\otimes\id_C)e^D_{Y,C}t_Y=(\xi\otimes\id_C)e^D_{Y,C}t_Y.
\]
Since $\alpha(t)g_\xi=\alpha(t)(\xi\otimes\id_D)\rho^D_Y=(\xi\otimes\id_C)(\id_Y\otimes\alpha(t))\rho^D_Y$ and $1$ is a left $\otimes$-cogenerator in $\cc$, we obtain that $e^D_{Y,C}\beta_Y=e^D_{Y,C}t_Y$, so $\beta_Y=t_Y$ for all $Y\in \cc^D$ as $e^D_{Y,C}$ is a monomorphism. Hence $\alpha'\circ \alpha=\id$. For $\chi\in {}^D\Hom^{D}(D,C)$, we have 
\[
\begin{split}
    \alpha\alpha'(\chi)&=\alpha((\beta_X=(Y\square_D\chi)\Lambda_Y )_{Y\in\cc^D})=(\Lambda'_C)^{-1}\beta_D\\&=(\varepsilon_D\otimes\id_C)e^D_{D,C}\beta_D=(\varepsilon_D\otimes\id_C)(\id_Y\otimes \chi)\rho^D_Y=\chi,
\end{split}
\]
so $\alpha\circ\alpha'=\id$.
\end{invisible}
\end{proof}

Several results in this section follow by duality from the counterparts in Section \ref{sub:induct}. We include them here without proof for future reference. Dually to Theorem \ref{thm:phi*semisep}, the semiseparability of coinduction functors can be characterized as follows.

\begin{thm}\label{prop:semisep-coind-nocogen}
   Let $\psi:C\to D$ be a morphism of coalgebras in a  monoidal category $\cc$. Then, the following assertions are equivalent:
\begin{itemize}   
\item[$(i)$] $\psi$ is a regular morphism of $D$-bicomodules in $\cc$ (i.e. there exists a $D$-bicomodule morphism $\chi:D\to C$ in $\cc$ such that $\psi\circ\chi\circ\psi=\psi$);
\item[$(ii)$] there is a $D$-bicomodule morphism $\chi:D\to C$ in $\cc$ such that $\varepsilon_C\circ\chi\circ\psi=\varepsilon_C$. 
\end{itemize}
If one of the previous equivalent conditions holds, then 
\begin{itemize}
\item[$(iii)$] the coinduction functor $\psi^*=-\square_DC:\cc^D\to\cc^C$ is semiseparable.
\end{itemize}
Moreover, if the unit object $1$ of $\cc$ is a left $\otimes$-cogenerator, then $(iii)$ is equivalent to $(i)$, $(ii)$. 
\end{thm}
\begin{invisible}
\begin{proof}
$(i)\Rightarrow (ii)$. Suppose that $\psi$ is a regular morphism of $D$-bicomodules in $\cc$, i.e. there exists a $D$-bicomodule morphism $\chi:D\to C$ in $\cc$ such that $\psi\circ\chi\circ\psi=\psi$. Then, $\varepsilon_C\circ\chi\circ\psi=\varepsilon_D\circ\psi\circ\chi\circ\psi=\varepsilon_D\circ\psi=\varepsilon_C$.

$(ii)\Rightarrow (i)$. Assume that $\chi$ is a $D$-bicomodule morphism in $\cc$ such that $\varepsilon_C\circ\chi\circ\psi=\varepsilon_C$. Then, \[
\begin{split}
    \psi&=(\psi\otimes\varepsilon_C)\Delta_C=(\psi\otimes \varepsilon_C\chi\psi)\Delta_C=(\id_D\otimes\varepsilon_C\chi)(\psi\otimes\psi)\Delta_C\\&=(\id_D\otimes\varepsilon_C)(\id_D\otimes \chi)\Delta_D\psi=(\id_D\otimes\varepsilon_C)\rho^D_C\chi\psi=(\id_D\otimes\varepsilon_C)(\psi\otimes\id_C)\Delta_C\chi\psi\\&=(\psi\otimes\id_1)(\id_C\otimes \varepsilon_C)\Delta_C\chi\psi=\psi\chi\psi.
\end{split}
\]


$(ii)\Rightarrow (iii)$. Suppose that there is a $D$-bicomodule morphism $\chi:D\to C$ in $\cc$ such that $\varepsilon_C\circ\chi\circ\psi=\varepsilon_C$. As in the proof of Lemma \ref{lem:Dbicomd}, we know that $\gamma_N: N\to N\square_DC$ defined by
\[
\gamma_N=(N\square_D\chi)\Lambda_N : N\to N\square_DC, 
\]
which is completely determined by $e^D_{N,C}\gamma_N=(\id_N\otimes\chi)\rho^D_N$, determines a natural transformation $\gamma:=(\gamma_N)_{N\in\cc^D}$ (here we do not need the assumption that $1$ is a left $\otimes$-cogenerator). For every $N\in\cc^D$, we have 
\[
\begin{split}
\epsilon_N\circ\gamma_N\circ\epsilon_N&=(\id_N\otimes\varepsilon_C)e^D_{N,C}\gamma_N(\id_N\otimes\varepsilon_C)e^D_{N,C}=(\id_N\otimes\varepsilon_C\chi)\rho^D_{N}(\id_N\otimes\varepsilon_C)e^D_{N,C}\\&=(\id_N\otimes\varepsilon_C\chi)(\id_{N\otimes D}\otimes\varepsilon_C)(\rho^D_N\otimes\id_C)e^D_{N,C}\\&=(\id_N\otimes\varepsilon_C\chi)(\id_{N\otimes D}\otimes\varepsilon_C)(\id_N\otimes\lambda^D_C)e^D_{N,C}\\&=(\id_N\otimes\varepsilon_C\chi)(\id_{N}\otimes \id_{D}\otimes\varepsilon_C)(\id_N\otimes(\psi\otimes\id_C)\Delta_C)e^D_{N,C}\\&=(\id_N\otimes\varepsilon_C\chi)(\id_{N}\otimes\psi)e^D_{N,C}=(\id_N\otimes\varepsilon_C)e^D_{N,C}=\epsilon_N.
\end{split}
\]

Assuming that the unit object $1$ of $\cc$ is a left $\otimes$-cogenerator, we show $(iii)\Rightarrow (i)$. 
If $\psi^*$ is a semiseparable functor, then by Theorem \ref{thm:rafael} there exists a natural transformation $\gamma:\id_{\cc^D}\to \psi_*\psi^*$ such that $\epsilon_N\circ\gamma_N\circ\epsilon_N=\epsilon_N$, for every $N\in \cc^D$. 
By Lemma \ref{lem:Dbicomd} we know that $\chi:D\to C$, $\chi:=(\Lambda'_C)^{-1}\circ\gamma_D=(\varepsilon_D\otimes\id_C)e^D_{D,C}\circ\gamma_D$, is a morphism of $D$-bicomodules. Note that $\psi=\epsilon_D\Lambda'_C$ as $\epsilon_D\Lambda'_C=(\id_D\otimes \varepsilon_C)e^D_{D,C}\Lambda'_C=(\id_D\otimes\varepsilon_C)\lambda^D_C=(\id_D\otimes\varepsilon_C)(\psi\otimes\id_C)\Delta_C=\psi(\id_C\otimes\varepsilon_C)\Delta_C=\psi$. Then, we have $\psi\circ\chi\circ\psi=\epsilon_D\Lambda'_C(\Lambda'_C)^{-1}\gamma_D\epsilon_D\Lambda'_C=\epsilon_D\Lambda'_C=\psi$. Note that $\psi$ is a $D$-bicomodule morphism in $\cc$. Indeed, $(\psi\otimes\id_D)(\id_C\otimes\psi)\Delta_C=\Delta_D\psi=(\id_D\otimes\psi)(\psi\otimes\id_C)\Delta_C$. 
\qedhere  
\end{proof}
\end{invisible}


The next is the dual of Theorem \ref{thm:sepmon} (see also Remark \ref{rmk:notens-gen}).

\begin{prop}\label{prop:sep-coind-nocogen}
   Let $\psi:C\to D$ be a morphism of coalgebras in a  monoidal category $\cc$. Then, the following assertions are equivalent:
\begin{itemize}   
\item[$(i)$] $\psi$ is a split-epi as a $D$-bicomodule morphism in $\cc$ (i.e. there exists a $D$-bicomodule morphism $\chi:D\to C$ in $\cc$ such that $\psi\circ\chi =\id_D$);
\item[$(ii)$] there is a $D$-bicomodule morphism $\chi:D\to C$ in $\cc$ such that $\varepsilon_C\circ\chi=\varepsilon_D$. 
\end{itemize}
If one of the previous equivalent conditions holds, then 
\begin{itemize}
\item[$(iii)$] the coinduction functor $\psi^*=-\square_DC:\cc^D\to\cc^C$ is separable.
\end{itemize}
Moreover, if the unit object $1$ of $\cc$ is a left $\otimes$-cogenerator, then $(iii)$ is equivalent to $(i)$, $(ii)$. 
\end{prop}
\begin{invisible}
\begin{proof}
$(i)\Rightarrow (ii)$. Suppose that $\psi$ is a split-epi as a morphism of $D$-bicomodules in $\cc$, i.e. there exists a $D$-bicomodule morphism $\chi:D\to C$ in $\cc$ such that $\psi\circ\chi=\id_D$. Then, $\varepsilon_C\circ\chi=\varepsilon_D\circ\psi\circ\chi=\varepsilon_D$. 

$(ii)\Rightarrow (i)$. Assume that $\chi$ is a $D$-bicomodule morphism in $\cc$ such that $\varepsilon_C\circ\chi=\varepsilon_D$. Then, \begin{displaymath}
\begin{split}
   \psi\chi&=(\id_D\otimes\varepsilon_D)\Delta_D\psi\chi=(\id_D\otimes\varepsilon_D)(\psi\otimes\psi)\Delta_C\chi=(\id_D\otimes\varepsilon_D\psi)(\psi\otimes\id_C)\Delta_C\chi\\&=(\id_D\otimes\varepsilon_C)(\id_D\otimes\chi)\Delta_D=(\id_D\otimes\varepsilon_C\chi)\Delta_D=(\id_D\otimes\varepsilon_D)\Delta_D=\id_D.
\end{split}
\end{displaymath}
\par 
$(ii)\Rightarrow (iii)$. Suppose that there is a $D$-bicomodule morphism $\chi:D\to C$ in $\cc$ such that $\varepsilon_C\circ\chi=\varepsilon_D$. As in the proof of Lemma \ref{lem:Dbicomd}, we know that $\gamma_N: N\to N\square_DC$ defined by
\[
\gamma_N=(N\square_D\chi)\Lambda_N : N\to N\square_DC, 
\]
which is completely determined by $e^D_{N,C}\gamma_N=(\id_N\otimes\chi)\rho^D_N$, determines a natural transformation $\gamma:=(\gamma_N)_{N\in\cc^D}$ (here we do not need the assumption that $1$ is a left $\otimes$-cogenerator). For every $N\in\cc^D$, we have 
\[
\begin{split}
\epsilon_N\circ\gamma_N =(\id_N\otimes\varepsilon_C)e^D_{N,C}\gamma_N=(\id_N\otimes\varepsilon_C\chi)\rho^D_{N}=(\id_N\otimes\varepsilon_D)\rho^D_N=\id_N.
\end{split}
\]
Assuming that the unit object $1$ of $\cc$ is a left $\otimes$-cogenerator, we show $(iii)\Rightarrow (i)$. if $\psi^*$ is a separable functor, then by \cite[Theorem 1.2]{Raf90} there exists a natural transformation $\gamma:\id_{\cc^D}\to \psi_*\psi^*$ such that $\epsilon_N\circ\gamma_N=\id_N$, for every $N\in \cc^D$. 
By Lemma \ref{lem:Dbicomd} we have a morphism of $D$-bicomodules $\chi:D\to C$ in $\cc$ given by $\chi:=(\Lambda'_C)^{-1}\circ\gamma_D=(\varepsilon_D\otimes\id_C)e^D_{D,C}\circ\gamma_D$. 
Since $\psi=\epsilon_D\Lambda'_C$, we have $\psi\circ\chi=\epsilon_D\Lambda'_C(\Lambda'_C)^{-1}\gamma_D=\epsilon_D\gamma_D=\id_D$. 
Moreover, $\psi$ is a $D$-bicomodule morphism in $\cc$.
\qedhere   
\end{proof}
\end{invisible}

The next result describes natural fullness for coinduction functors, dually to Proposition \ref{prop:phi*natfull}.

\begin{prop}\label{prop:natfull-coind}
 Let $\psi:C\to D$ be a morphism of coalgebras in a monoidal category $\cc$. 
Then, the following assertions are equivalent:
\begin{itemize}   
\item[$(i)$] $\psi$ is a split-mono as a $D$-bicomodule morphism in $\cc$ (i.e. there exists a $D$-bicomodule morphism $\chi:D\to C$ in $\cc$ such that $\chi\circ\psi =\id_C$);
\item[$(ii)$] $\psi$ is a monomorphism and there is a $D$-bicomodule morphism $\chi:D\to C$ in $\cc$ such that $\varepsilon_C\circ\chi\circ\psi=\varepsilon_C$. 
\end{itemize}
If one of the previous equivalent conditions holds, then 
\begin{itemize}
\item[$(iii)$] the coinduction functor $\psi^*=-\square_DC:\cc^D\to\cc^C$ is naturally full.
\end{itemize}
Moreover, if the unit object $1$ of $\cc$ is a left $\otimes$-cogenerator, then $(iii)$ is equivalent to $(i)$, $(ii)$.
\end{prop}
\begin{invisible}
\begin{proof}
$(i)\Rightarrow (ii)$. It is obvious.

$(ii)\Rightarrow (i)$. By Theorem \ref{prop:semisep-coind-nocogen} $\varepsilon_C\circ\chi\circ\psi=\varepsilon_C$, is equivalent to $\psi\circ\chi\circ\psi=\psi$. Since $\psi$ is a monomorphism, one has $\chi\circ\psi=\id_C$.

$(i)\Rightarrow (iii)$. Assume that there exists a morphism of $D$-comodules $\chi:D\to C$ in $\cc$ such that $\chi\circ\psi =\id_C$.

   As in the proof of Lemma \ref{lem:Dbicomd}, consider the natural transformation $\gamma$ given by 
\[
\gamma_N=(N\square_D\chi)\Lambda_N : N\to N\square_DC, 
\]
which is completely determined by $e^D_{N,C}\gamma_N=(\id_N\otimes\chi)\rho^D_N$. 
For every $N\in\cc^D$, we have 
\[
\begin{split}
e^D_{N,C}\circ\gamma_N\circ\epsilon_N &=e^D_{N,C}\gamma_N (\id_N\otimes\varepsilon_C)e^D_{N,C} =
(\id_N\otimes\chi)\rho^D_N (\id_N\otimes\varepsilon_C) e^D_{N,C} \\&=(\id_N\otimes\chi\otimes \id_1)(\id_N\otimes\id_D\otimes\varepsilon_C)(\rho^D_N\otimes\id_C)e^D_{N,C}\\&=(\id_N\otimes\chi\otimes\id_1)(\id_N\otimes\id_D\otimes\varepsilon_C)(\id_N\otimes\lambda^D_C)e^D_{N,C}\\&=(\id_N\otimes\chi\psi\otimes\varepsilon_C)(\id_N\otimes\Delta_C)e^D_{N,C}\\&=(\id_N\otimes\id_C\otimes\varepsilon_C)(\id_N\otimes\Delta_C)e^D_{N,C}=e^D_{N,C},
\end{split}
\]
hence $\gamma_N\circ\epsilon_N=\id_{N\square_DC}$. We conclude by Rafael-type Theorem \cite[Theorem 2.6]{ACMM06}. 

Assuming that the unit object $1$ of $\cc$ is a left $\otimes$-cogenerator, we show $(iii)\Rightarrow (i)$.

If $\psi^*$ is naturally full, then there exists a natural transformation $\gamma:\id_{\cc^D}\to \psi_*\psi^*$ such that $\gamma_N\circ\epsilon_N=\id_{N\square_DC}$, for every $N\in \cc^D$. 
By Lemma \ref{lem:Dbicomd} we have a morphism of $D$-bicomodules $\chi:D\to C$ in $\cc$ given by $\chi:=(\Lambda'_C)^{-1}\circ\gamma_D=(\varepsilon_D\otimes\id_C)e^D_{D,C}\circ\gamma_D$. 
Since $\psi=\epsilon_D\Lambda'_C$, we have $\chi\circ\psi=(\Lambda'_C)^{-1}\gamma_D\epsilon_D\Lambda'_C=(\Lambda'_C)^{-1}\Lambda'_C=\id_C$. 
\end{proof}
\end{invisible}
\begin{rmk}\label{rmk:coind-nocotens}
    One can observe that if the coinduction functor $\psi^*:\cc^D\to\cc^C$ is semiseparable (resp., separable, naturally full), then there exists $\chi\in\Hom^D(D,C)$ such that $\psi\chi\psi=\psi$ (resp., $\psi\circ\chi=\id_D$, $\chi\circ\psi=\id_C$).
\end{rmk}
\begin{cor}
    Let $\psi:C\to D$ be a coalgebra morphism in a monoidal category $\cc$ with unit being a left $\otimes$-cogenerator. Then, the coinduction functor $\psi^*:\cc^D\to\cc^C$ is naturally full if, and only if, $\psi^*$ is semiseparable and $\psi$ is a monomorphism. Similarly, $\psi^*$ is separable if, and only if, $\psi^*$ is semiseparable and $\psi$ is an epimorphism.
\end{cor}

\begin{rmk}\label{rmk:coind-left-rightsymm}
Given a coalgebra morphism $\psi:C\to D$ in $\cc$, we can also consider the coinduction functor $\psi^*=C\square_D-:{}^D\cc\to {}^C\cc$. Indeed, $(i)$ in Proposition \ref{prop:semisep-coind-nocogen} implies that $\psi^*:{}^C\cc\to {}^D\cc$ is semiseparable. If $1$ is a right $\otimes$-cogenerator, then also the converse holds. Therefore, if $1$ is a two-sided $\otimes$-cogenerator, then $C\square_D- $ is semiseparable if, and only if, so is $-\square_DC$. Similar statements hold for separable and naturally full cases.
\end{rmk}

\subsection{Examples}
We now apply our results to some examples.

\begin{es}\label{Vec cogenerator}
Let $\Bbbk$ be a field. Consider the monoidal category $(\vect_\Bbbk, \otimes_\Bbbk, \Bbbk)$ of $\Bbbk$-vector spaces. 
For every  $X\in \vect_\Bbbk$, it is well-known that the functor $-\otimes_\Bbbk X$ (resp., $X\otimes_\Bbbk -$) preserves equalizers. Now, we show that $\Bbbk$ is a left
 $\otimes$-cogenerator. Let $f,g: Y\to Z\otimes_\Bbbk W$ be $\Bbbk$-linear maps. Suppose that $(\xi\otimes_\Bbbk \id_W)(f(y))=(\xi\otimes_\Bbbk\id_W)(g(y))$, for all $\xi:Z\to \Bbbk$ in $\vect_\Bbbk$. Write $f(y)$ and $g(y)$ on a basis, we obtain $\sum_{i,j}k^f_{ij}\xi(z_i)\otimes w_j=\sum_{i,j}k^g_{ij}\xi(z_i)\otimes w_j$. Then, $\sum_{i}k^f_{ij}\xi(z_i)=\sum_i k^g_{ij}\xi(z_i)$, for all $\xi:Z\to \Bbbk$. In particular, for $\xi=z^*_l:Z\to \Bbbk$, we get $k^f_{lj}=k^g_{lj}$. Therefore $f=g$.
Then, Theorem \ref{prop:semisep-coind-nocogen} recovers \cite[Proposition 3.8]{AB22}, Proposition \ref{prop:sep-coind-nocogen} recovers \cite[Theorem 2.7]{CGN97}, and Proposition \ref{prop:natfull-coind} recovers \cite[Examples 3.23 (1)]{ACMM06}. 
\end{es}

\begin{es}
Let $(\cc, \otimes, 1)$ and $(\cc', \otimes, 1')$ be two monoidal categories. By \cite[page 4]{AM10}, their product category $\cc_1\times \cc_2$ is a monoidal category with tensor product $(X, X')\otimes (Y, Y'):=(X\otimes Y, X'\otimes Y')$
and unit object $(1, 1')$. If $1$ and $1'$ are left $\otimes$-cogenerators of $\cc$ and $\cc'$ respectively, then $(1,1')$ is a left $\otimes$-cogenerator in $\cc\times \cc'$. In fact, let $(f,f')$, $(g,g'): (Y,Y') \to (Z,Z') \otimes (W,W')$ be morphisms in $\cc\times \cc'$ 
such that $((\xi, \xi')\otimes \id_{(W,W')})(f,f') = ((\xi, \xi')\otimes \id_{(W,W')})(g,g')$ for all $(\xi, \xi'): (Z,Z') \to (1,1')$. 
This means $(\xi\otimes \id_W)f = (\xi\otimes \id_W)g$ and $(\xi'\otimes \id_{W'})f' = (\xi'\otimes \id_{W'})g'$, for all $\xi:Z\to 1$ in $\cc$ and $\xi':Z'\to 1'$ in $\cc'$. Since $1$ and $1'$ are left $\otimes$-cogenerators of $\cc$ and $\cc'$, we have $f=g$ and $f'=g'$. It follows that $(f,f') = (g,g')$, i.e. $(1,1')$ is a left $\otimes$-cogenerator in $\cc\times \cc'$.

As an example, consider the category $\cc=\cc'=\vect_\Bbbk$ of vector spaces over a field $\Bbbk$. By Example \ref{Vec cogenerator}, we know that $1=\Bbbk$ is a left $\otimes$-cogenerator of $\vect_\Bbbk$. Then, $(1, 1)$ is a left $\otimes$-cogenerator of $\vect_\Bbbk \times\vect_\Bbbk$. Besides, it is easy to check that $\vect_\Bbbk \times \vect_\Bbbk$ has equalizers and every object is flat. Then, Theorem \ref{prop:semisep-coind-nocogen}, Proposition \ref{prop:sep-coind-nocogen}, and Proposition \ref{prop:natfull-coind} apply.
\end{es}

\begin{es}\label{es:set-coinduct}
    Consider the monoidal category $(\Set, \times, \{0\})$ of sets. Note that $\{0\}$ is not a left $\otimes$-cogenerator. Otherwise, wherever we consider $f,g:Y\to Z\times W$ in $\Set$ such that $(\xi\times\id_W)f=(\xi\times\id_W)g$, for every $\xi:Z\to\{0\}$, then $f=g$. Consider $Z=\{0,1\}$ and for a fixed $w\in W$, let $f:Y\to Z\times W$, $f(y)=(0,w)$, and $g: Y\to Z\times W$, $g(y)=(1,w)$. Then, $(\xi\times\id_W)f=(\xi\times\id_W)g$, but $f\neq g$.  In fact, a left $\otimes$-cogenerator can be chosen as a set with at least two distinct elements, see e.g. \cite[Proposition 4.7.1]{Bor94}. 
    \begin{invisible}
        Given distinct maps $f,g:Y\to Z\times W$, there exists $y\in Y$ such that $(z_y,w_y)=f(y)\neq g(y)=(z'_y,w'_y)$. Choosing any map $\xi:Z\to \{0,1\}$ such that $\xi(z_y)=0$, $\xi(z'_y)=1$, we get $(\xi\times W)(f(y))\neq (\xi\times W)(g(y))$.
    \end{invisible}
Besides, for any $S\in \Set$, $-\times S$ (resp., $S\times -$) preserves equalizers. In fact, 
$\mathrm{Eq}(f,g)\times S=\{(v,s)\in V\times  S\mid f(v)=g(v)\}=\{(v,s)\in V\times S\mid (f\times\id_S)(v,s)=(g\times\id_S)(v,s)\}=\mathrm{Eq}(f\times\id_S,g\times\id_S).$  
Recall that any set $S$ is a coalgebra in $\Set$ with comultiplication $\Delta_S:S\to S\times S$, $s\mapsto (s,s)$ and counit $S\to \{0\}$, $s\mapsto 0$; any map is a coalgebra morphism in $\Set$. Given a morphism $\psi:C\to D$ in $\Set$, we claim that the coinduction functor $\psi^*$ is separable if, and only if, $\psi$ is surjective. In fact, if $\psi$ is surjective, then by the Axiom of Choice $\psi$ is a split-epi, so there exists $\chi:D\to C$ such that $\psi\chi=\id_D$. We observe that $\chi$ is $D$-bicomodule morphism in $\Set$, because $(\chi\times\id_D)\Delta_D(d)=(\chi(d),d)=(\chi(d), \psi\chi(d))=(\id_C\times\psi)\Delta_C(\chi(d))$ and similarly $(\id_D\times\chi)\Delta_D=(\psi\times\id_C)\Delta_C\chi$. Then, $\psi^*$ is separable by Proposition \ref{prop:sep-coind-nocogen}. Conversely, if $\psi^*$ is separable, then by Remark \ref{rmk:coind-nocotens} we know that there exists  $\chi\in\Hom^D(D,C)$ such that $\psi\chi=\id_D$. Then, $\psi$ is surjective.

However, we do not have a similar claim for the naturally full case. More precisely, if $\psi:C\to D$ is injective, then $\psi^*$ is not necessarily naturally full. An example is given by $\psi:\{0\}\to S$, $\psi(0)=s_1$, for a non-empty set $S$, where $s_1$ is a fixed element in $S$. Then, the only  $\chi:S\to\{0\}$ is given by $\chi(s)=0$, which is not an $S$-bicomodule morphism in $\Set$ if $S$ has at least two elements. By Remark \ref{rmk:coind-nocotens}, $\psi^*$ is not naturally full.

Now, for a coalgebra morphism $\psi:C\to D$ in $\Set$, we claim that  $\psi^*$ is semiseparable if, and only if, $\psi^*$ is separable. Indeed, if $\psi^*$ is semiseparable, by Remark \ref{rmk:coind-nocotens}, there is a right $D$-comodule morphism $\chi:D\to C$ such that $\psi\chi\psi=\psi$. Since $\chi$ is a right $D$-comodule morphism, we know that $\psi\chi(d)=d$, for all $d\in D$, hence $\psi$ is surjective. Then, $\psi^*$ is separable. 
On the other hand, if $\psi^*$ is separable then it is semiseparable by Proposition \ref{prop:sep}. As a consequence, $\psi^*$ is naturally full if and only if it is fully faithful. Moreover, if $\psi^*$ is naturally full, then by Remark \ref{rmk:coind-nocotens} $\psi$ is bijective. Conversely, if $\psi$ is bijective then there exists a map $\chi:D\to C$ such that $\chi\psi=\id_C$ and $\psi\chi=\id_D$. From what observed above, $\chi$ is a $D$-bicomodule morphism, hence $\psi^*$ is naturally full by Proposition \ref{prop:natfull-coind}.



\end{es}


\begin{es}
Let $H$ be a finite dimensional bialgebra over a field $\Bbbk$, and let $({}_H\!\m, \otimes, \Bbbk)$ be the category of left $H$-modules considered in Example \ref{es:bialgmod}. It is well-known that, for any object $X$ in $_H\!\m$, the functors $X\otimes -$ and $-\otimes X$ are exact. 
\begin{invisible}
In fact, for a monomorphism $i$ in $_H\!\m$, it is a split-mono in $\vect$, i.e. there is a $j$ such that $ji = \id$. Thus, $(X\otimes j) (X\otimes i) =X\otimes (ji) = \id$, which means that $X\otimes i$ is a monomorphism in $\vect$. Hence, $X\otimes i$ is a monomorphism in $_H\!\m$.
\end{invisible} 
Recall that a coalgebra $C$ in ${}_H\m$ is an $H$-module coalgebra, i.e. a coalgebra $C$  with coproduct $\Delta_C$ and counit $\varepsilon_C$ such that
$\Delta_C(h\cdot c)=(h_1\cdot c_1)\otimes (h_2\cdot c_2)$, and $\varepsilon_C (h\cdot c)=\varepsilon_H(h)\varepsilon_C(c)$. 
We observe that $\Bbbk$ is not a left $\otimes$-cogenerator in general. For example, consider the group algebra $\Bbbk G $, where $G$ is a finite group. Define $f_g: \Bbbk G\to \Bbbk G$, $x\mapsto xg$, for $g\neq 1$. 
Clearly, $f_g$ and the identity $\id_{\Bbbk G}$ are left $\Bbbk G$-module morphisms. For any left $\Bbbk G$-module morphism $\xi: \Bbbk G\to \Bbbk$, the $\Bbbk G$-linearity means $\xi(h)=\varepsilon(h)\xi(1)$, for all $h\in \Bbbk G$. Consequently, $\xi f_g(x)=\xi(xg)=\varepsilon(x)\xi(1)=\xi(x)$, but $f_g\neq \id_{\Bbbk G}$. As a result, $\Bbbk$ is not a cogenerator, hence not a left $\otimes$-cogenerator. 

We also observe that any simple object cannot be a cogenerator in a semisimple abelian category, which contains at least two non-isomorphic objects. Hence, it cannot be a left $\otimes$-cogenerator in the corresponding monoidal case. To see this, suppose a simple object $X$ is a left cogenerator, and choose another simple object $Y$ such that $Y$ is not isomorphic to $X$. By Schur's Lemma (see e.g. \cite[Lemma 1.5.2]{Eti15}), the only morphism between $Y$ and $X$ is the zero morphism. Thus, we can choose $\id_Y:Y \to Y$ and zero morphism $0_Y:Y \to Y$. We always have $0_{Y,X}\id_Y = 0_{Y,X}0_Y$, but $\id_Y \not= 0_Y$.



Let $H'$ be a coalgebra in ${}_H\m$. It is clear that $\Bbbk\oplus H'$ is in ${}_H\m$ with componentwise $H$-module structure. We claim that $\Bbbk\oplus H'$ is a coalgebra in ${}_H\m$ with coalgebra structure given by $\Delta_{\Bbbk\oplus H'}((k,h)):=(k,0)\otimes (1,0)+(0,h_1)\otimes (0,h_2)$ and $\varepsilon_{\Bbbk\oplus H'}((k,h))=k+\varepsilon_{H'}(h)$. Indeed, this coalgebra structure follows from the fact that $_H\m$ is a duoidal category (see Definition \ref{defn:duoidalcategory}) with $\circ=\oplus$, $\bullet = \otimes$, see e.g. \cite[Example 6.19]{AM10}, and use \cite[Proposition 6.35]{AM10}.
Consider the injection $i_{\Bbbk}: \Bbbk\to \Bbbk\oplus H' $ and the projection $\pi: \Bbbk\oplus H'\to \Bbbk$, which are morphisms in ${}_H\m$.  We show that $\pi$ is $(\Bbbk\oplus H')$-bicolinear. In fact, $(i_\Bbbk\otimes \Bbbk)\Delta_\Bbbk\pi((k,h))=(k,0)\otimes 1=((\Bbbk\oplus H')\otimes\pi)((k,0)\otimes(1,0)+(0,h_1)\otimes (0,h_2))=((\Bbbk\oplus H')\otimes\pi)\Delta((k,h))$. Similarly, one can check that $\pi$ is right $(\Bbbk\oplus H')$-colinear. Moreover, $\pi\circ i_{\Bbbk}=\id_\Bbbk$, hence by Proposition \ref{prop:natfull-coind} the coinduction functor $i_\Bbbk^*: ({}_H\m)^{\Bbbk \oplus H'}\to ({}_H\m)^\Bbbk$ is naturally full. 
\end{es}

\begin{es}
Let $\cc$ be a monoidal category, $A$, $A'$ be algebras in $\cc$, and $\pi:A\to A'$, $\gamma:A'\to A$ be algebra morphisms in $\cc$ such that $\pi\gamma =\id_{A'}$. Since $\pi$ is an algebra morphism in $\cc$, we have that $\pi m_A(\gamma\otimes A)=m_{A'}(\pi\otimes \pi)(\gamma\otimes A)=m_{A'}(A'\otimes \pi)$, and similarly $\pi m_A(A\otimes \gamma)=m_{A'}(\pi\otimes \pi)(A\otimes \gamma)=m_{A'}(\pi\otimes A')$. By Theorem \ref{thm:sepmon} and Remark \ref{rmk:notens-gen}, the induction functor $\gamma^*:\cc_{A'}\to\cc_A$ is separable. Dually, let $C, C'$ be coalgebras in $\cc$, and let $\pi:C\to C'$, $\gamma: C'\to C$ be coalgebra morphisms in $\cc$ such that $\pi\gamma=\id_{C'}$.  
By Proposition \ref{prop:sep-coind-nocogen}, the coinduction functor $\pi^*:\cc^{C'}\to \cc^C$ is separable.

As an example, consider a bialgebra with a projection in a braided monoidal category $\cc$, i.e. a quadruple
$(A, H, \pi, \gamma)$ where $A$ and $H$ are bialgebras in $\cc$, and $\pi:A\to H$, $\gamma:H\to A$ are bialgebra morphisms in $\cc$ such that $\pi\gamma=\id_H$. 
Then, the induction functor $\gamma^*: \cc_H\to \cc_A$ and the coinduction functor $\pi^*:\cc^{H}\to \cc^A$ are separable. 
In particular, let $H$ be a Hopf algebra in $\cc$. Consider a bialgebra $R$ in ${}^H_H\!\mathcal{YD}(\mathcal{C})$, and the bosonization $R\# H$, see \cite[Definition 3.8.1]{HS20}. By \cite[Proposition 3.8.4]{HS20}, $R\# H$ is a bialgebra in $\mathcal{C}$, and by \cite[Lemma 3.8.2]{HS20} $\pi:=\varepsilon_R\otimes \id_H:R\# H\to H$ and $\gamma:=u_R\otimes \id_H: H\to R\#H$ are bialgebra morphisms in $\cc$ with $\pi\gamma =\id_H$. Then, the induction functor $\gamma^*:\cc_H\to \cc_{R\#H}$ and the coinduction functor $\pi^*:\cc^{H}\to\cc^{R\#H}$ are separable.
\end{es}

\subsection{Tensor functors of coalgebras in a monoidal category}

Let $(\cc, \otimes, 1)$ be a monoidal category and let $(C,\Delta, \varepsilon)$ be a coalgebra in $\cc$. Recall that for any object $M$ in $\cc$, $(M\otimes C, M\otimes\Delta)$ is an object in the category $\cc^C$ of right $C$-comodules in $\cc$, and the tensor functor $-\otimes C:\cc\to\cc^C$, $ M\mapsto (M\otimes C, M\otimes \Delta)$, $f\mapsto f\otimes C$, is a right adjoint of the forgetful functor $F:\cc^C\to\cc$. The unit $\eta$ is defined for every $(N,\rho_N)\in\cc^C$ by $\eta_{(N,\rho_N)}=\rho_N:N\to N\otimes C=F(N)\otimes C$; the counit $\epsilon$ is given for any $M\in\cc$ by $\epsilon_M= M\otimes\varepsilon :F(M\otimes C)=M\otimes C\to M$, see e.g. \cite[Section 2]{MT21}.

\begin{lem}\label{lem:coalg-coind}
    Let $(C,\Delta, \varepsilon)$ be a coalgebra in a monoidal category $\cc$. Then, the coinduction functor $\varepsilon^*=-\square_1 C: \cc^1\to\cc^C$ results to be $-\otimes C:\cc\to\cc^C$.
\end{lem}
\begin{proof}
 The right $1$-comodule structure on any object $V$ in $\cc$ is given by the unit constraint $r^{-1}_V$. 
 By the triangle axiom we have that $r^{-1}_{V}\otimes\id_C=\id_V\otimes l_C^{-1}$.
 Thus, we have the equalizer \[
    \xymatrix@C=1.5cm{
    V\otimes  C\ar[r]^-{\id_{V\otimes C}}& V\otimes C\ar@<-1ex>[r]_-{\id_V\otimes l^{-1}_C}\ar@<1ex>[r]^-{r^{-1}_{V}\otimes\id_W}&V\otimes 1\otimes C}
    \]
    This means that $V\square_1C=V\otimes C$. \qedhere
\end{proof}

If $\varepsilon_C:C\to 1$ is a regular morphism in $\cc$, by Theorem \ref{prop:semisep-coind-nocogen} $\varepsilon_C^*$ is semiseparable. Hence, so is $-\otimes C: \cc\to \cc^C$. Note that this theorem requires that $\cc$ has equalizers and every object is flat. However, we can still obtain the result when we get rid of these conditions.

\begin{prop}\label{prop:coalgebra-induct-functor}
    Let $C$ be a coalgebra in a monoidal category $\cc$. Then, the following assertions are equivalent.
\begin{enumerate}
  \item[$(i)$] The counit $\varepsilon_{C}:C\to 1$ is regular (resp., split-epi, split-mono) as a morphism in $\cc$. 
  \item[$(ii)$] The tensor functor $-\otimes C:\mathcal{C}\to \mathcal{C}^C$ is semiseparable (resp., separable, naturally full).

  \item[$(iii)$] The tensor functor $C\otimes -: \mathcal{C}\to {}^C\mathcal{C}$ is semiseparable (resp., separable, naturally full). 
\end{enumerate} 
\end{prop}
\begin{proof}
It is the dual of the proof of Proposition \ref{prop:algebra-induct-functor}.
\begin{invisible}
We just consider the semiseparable case.

$(i)\Rightarrow (ii)$. Assume $(i)$, i.e.\ there exists $\chi: 1\to C$ such that $\varepsilon_C\chi\varepsilon_C=\varepsilon_C$. For every $M\in \cc$ define $\gamma_M=M\otimes\chi:M\to M\otimes C$. Then, $\gamma$ is a natural transformation as for every $f:M\to M'$ in $\cc$, one has $(M'\otimes \chi)(f\otimes\id_1)=f\otimes\chi=(f\otimes C)(M\otimes\chi)$. Moreover, $\epsilon_M\gamma_M\epsilon_M=(M\otimes\varepsilon_C)(M\otimes\chi)(M\otimes \varepsilon_C)=M\otimes\varepsilon_C=\epsilon_M$, for every $M\in\cc$.


$(ii)\Rightarrow (i)$. By Theorem \ref{thm:rafael} there exists a natural transformation $\gamma:\id_\cc\to F(-\otimes C)$ such that $\epsilon\gamma\epsilon=\epsilon$. Define $\chi=\gamma_1:1\to C$. Since $\varepsilon_C\chi\varepsilon_C=(1\otimes\varepsilon_C)\gamma_1(1\otimes\varepsilon_C)=\epsilon_1\gamma_1\epsilon_1=\epsilon_1=1\otimes\varepsilon_C=\varepsilon_C$, we get $(i)$. 

$(i)\Leftrightarrow (iii)$. It follows similarly.
\end{invisible}
\end{proof}

\begin{rmk}
    We observe that $(i)$ in Proposition \ref{prop:coalgebra-induct-functor} is equivalent to the existence of a morphism $\chi:1\to C$ in $\cc$ such that $(\varepsilon_C\otimes C)(\chi\otimes C)=\id_C$, dually to Remark \ref{rmk:algmonoid-induct}.
    \begin{invisible}
    On one hand, $\varepsilon_C\chi\varepsilon_C=\varepsilon_C$ implies  $\varepsilon_C\otimes C=\varepsilon_C\chi\varepsilon_C\otimes C=(\varepsilon_C\chi\otimes C)(\varepsilon_C\otimes C)$. Since $\varepsilon_C\otimes C$ is an epimorphism, $\id_C=\varepsilon_C\chi\otimes C=(\varepsilon_C\otimes C)(\chi\otimes C)$.   On the other hand, $(\varepsilon_C\otimes C)(\chi\otimes C)=\id_C$ implies $\varepsilon_C=\id_1\otimes\varepsilon_C=(\id_1\otimes\varepsilon_C)(\varepsilon_C\otimes C)(\chi\otimes C)=(\varepsilon_C\otimes 1)(\chi\otimes 1)(1\otimes \varepsilon_C)=\varepsilon_C\chi\varepsilon_C$. 
    \end{invisible}
\end{rmk}

Now, we discuss some examples.

\begin{es}
    Let $C$ be a bialgebra in a braided monoidal category $\cc$. Since $\varepsilon_C u_C=\id_1$, we have that $\varepsilon_C$ is a split-epi. By Proposition \ref{prop:coalgebra-induct-functor}, $-\otimes C: \mathcal{C}\to \mathcal{C}^C$ is separable. 
\end{es}

\begin{es}
  Let $R$ be a ring. Consider the monoidal category  $({}_R\m_R, \otimes_R, R)$ of $R$-bimodules and let $C$ be an $R$-coring, i.e. a coalgebra in ${}_R\m_R$.  
  Then, Proposition \ref{prop:coalgebra-induct-functor} recovers Proposition \ref{prop:coring}, \cite[Theorem 3.3]{Brz02}, \cite[Proposition 3.13]{ACMM06} for the semiseparable, separable, naturally full cases, respectively.   
  

\end{es}

\begin{es}\label{v}
    Let $\Bbbk$ be a field and consider the monoidal category $(\vect, \otimes_\Bbbk, \Bbbk)$ of vector spaces over $\Bbbk$. Let $C\neq 0$ be a $\Bbbk$-coalgebra. By Proposition \ref{prop:coalgebra-induct-functor} the functor $-\otimes_\Bbbk C:\vect\to \mathcal{M}^C$ is semiseparable if, and only if, it is separable, since in this case the counit $\varepsilon_C$ is surjective. 
    Besides, by \cite[3.29]{BW03} $-\otimes_\Bbbk C$ is separable if, and only if, there exists $z\in C$ such that $\varepsilon_C(z)=1$. Since $\varepsilon_C$ is surjective, then $-\otimes_\Bbbk C$ is always separable. Moreover, $-\otimes_\Bbbk C$ is full if, and only if, it is naturally full if, and only if, it is fully faithful.

    Consider now the category $({}_H\m,\otimes_\Bbbk, \Bbbk)$ of left $H$-modules over a $\Bbbk$-bialgebra $H$.  Let $C$ be a coalgebra in ${}_H\m$, i.e. an $H$-module coalgebra. Since $\varepsilon_C$ is surjective, the functor $-\otimes_\Bbbk C:{}_H\m\to{{}_H\m}^C$ is semiseparable if, and only if, it is separable.

    More generally, let $\cc$ be an abelian monoidal category such that the unit object is simple, for instance a ring category with left duals, see \cite[Theorem 4.3.8]{Eti15}. The counit $\varepsilon_C$ of a non-zero coalgebra in $\cc$ is always an epimorphism. Consequently, $-\otimes C: \cc \to \cc^C$ is semiseparable if, and only if, it is separable.
\end{es}

\begin{es}
    Consider the monoidal category $(\Set, \times, \{0\})$ of sets. Any set $S$ is a coalgebra in $\Set$, see Example \ref{es:set-coinduct}, and the counit $\varepsilon_S:S\to \{0\}$ is surjective if $S$ is non-empty. Moreover, let $\overline{s}\in S$ and define $\chi:\{0\}\to S$, $0\mapsto \overline{s}$. We have $\varepsilon_S\chi(0)=0$, so $\varepsilon_S$ is a split-epi.  Then, for any non-empty set $S$, Proposition \ref{prop:coalgebra-induct-functor} implies that the functor $-\times S$ is separable. As a consequence, it is semiseparable. Therefore, $-\times S$ is full if, and only if, it is naturally full if, and only if, $\varepsilon_S$ is a bijection.
\end{es}


\subsection{Coinduction functors varied by a monoidal functor}

Let $(F:\cc\to\dd,\psi^2,\psi^0)$ be a colax monoidal functor, see e.g. \cite[Definition 3.2]{AM10}, i.e.\ a functor $F:\cc\to\dd$ between monoidal categories which is equipped with a natural transformation $\psi^2$, given on components by $\psi^2_{X,Y}:F(X\otimes Y)\to F(X)\otimes F(Y)$, and a morphism $\psi^0: F(1)\to 1'$ satisfying axioms dual to those for a lax monoidal functor (see Subsection \ref{subsect:monoidal-functor}). If  $(C, \Delta_C, \varepsilon_C)$ is a coalgebra in $\cc$, then $(F(C), \psi^2_{C,C}F(\Delta_C), \psi^0F(\varepsilon_C))$ is a coalgebra in $\dd$, see e.g. \cite[Proposition 3.29]{AM10}. Moreover, if $\varphi:C\to D$ is a coalgebra morphism in $\cc$, then the induced morphism $F(\varphi):F(C)\to F(D)$ is a coalgebra morphism in $\dd$.

The next result is the dual of Proposition \ref{prop:algebra-monoidalfunctor}. 

\begin{prop}\label{prop:coalgebra-monoidalfunctor}
 Let $(\cc, \otimes, 1)$, $(\dd, \otimes, 1')$ be monoidal categories. Assume that $1$ is a left $\otimes$-cogenerator in $\cc$. Let $(F:\cc\to\dd,\psi^2,\psi^0)$ be a colax monoidal functor and let $\varphi:C\to D$ be a coalgebra morphism in $\cc$ such that $\varphi^*=-\square_DC:\cc^D\to\cc^C$ is semiseparable (resp., separable, naturally full). Then, $F(\varphi)^*=-\square_{F(D)}F(C):\dd^{F(D)}\to\dd^{F(C)}$ is semiseparable (resp., separable, naturally full).

 Conversely, assume $1'$ is a left $\otimes$-cogenerator in $\dd$ and $F$ is a fully faithful lax monoidal functor such that $\psi^2$ is monomorphism on components. If $F(\varphi)^*=-\square_{F(D)}F(C):\dd^{F(D)}\to\dd^{F(C)}$ is semiseparable (resp., separable, naturally full), then $\varphi^*=-\square_DC:\cc^D\to\cc^C$ is semiseparable (resp., separable, naturally full).
\end{prop}
\begin{invisible}
\begin{proof}
If $\varphi^*$ is semiseparable, then by Theorem \ref{prop:semisep-coind-nocogen} there exists a $D$-bicomodule morphism $\chi:D\to C$ in $\cc$ such that $\varphi\circ\chi\circ\varphi=\varphi$. We obtain that $F(\varphi) \circ F(\chi) \circ F(\varphi)= F(\varphi \chi \varphi) =F(\varphi)$. We show that $F(\chi)$ is an $F(D)$-bicomodule morphism in $\dd$. Indeed, 
\[
\begin{split}
&\rho_{F(C)}^{F(D)} F(\chi)= (\id_{FC}\otimes F(\varphi))\psi^2_{C,C} F(\Delta_C) F(\chi) = \psi^2_{C,D} F(\id_C\otimes \varphi)  F(\Delta_C) F(\chi)\\
= &\psi^2_{C,D} F((\id_C\otimes \varphi)\Delta_C  \chi) = \psi^2_{C,D} F((\chi \otimes \id_D)  \Delta_D ) = \psi^2_{C,D} F(\chi \otimes \id_D)  F(\Delta_D )\\
= &(F(\chi)\otimes \id_{FD}) \psi^2_{D,D}F(\Delta_D) = (F(\chi)\otimes \id_{FD}) \rho_{F(D)}^{F(D)},
\end{split} 
\]
and similarly $\lambda_{F(C)}^{F(D)} F(\chi)= (\id_{FD} \otimes F(\chi)) \lambda_{F(D)}^{F(D)}$.
Thus, by Theorem \ref{prop:semisep-coind-nocogen} again, $F(\varphi)^*$ is semiseparable. Separable and naturally full cases follow in a similar way.

Conversely, suppose $1'$ is a $\otimes$-cogenerator in $\dd$ and $F$ is fully faithful lax monoidal with $\psi^2$ monomorphism on components. If $F(\varphi)^*=-\square_{F(D)}F(C):\dd^{F(D)}\to\dd^{F(C)}$ is semiseparable, by Theorem \ref{prop:semisep-coind-nocogen} there exists a $F(D)$-bicomodule morphism $\chi': F(D)\to F(C)$ in $\cc$ such that $F(\varphi)\chi'F(\varphi)=F(\varphi)$. Since $F$ is full, $\chi'=F(\chi)$ for some $\chi:D\to C$ in $\cc$. Hence, $F(\varphi \chi \varphi)=F(\varphi)$, which implies $\varphi \chi\varphi=\varphi$ because $F$ is faithful. It remains to show $\chi$ is a $D$-bicomodule map. Because $\chi'=F(\chi)$ is a right $F(D)$-comodule morphism, which means $(\id_{F(C)}\otimes F(\varphi))\Delta_{F(C)} F(\chi) = (F(\chi)\otimes \id_{F(D)}) \Delta_{F(D)}$, we have $\psi^2_{C,D} F(\id_C\otimes \varphi)  F(\Delta_C) F(\chi) = \psi^2_{C,D} F(\chi \otimes \id_D)  F(\Delta_D )$. Since $F$ is faithful and  $\psi^2$ mono, we have $(\id_C\otimes \varphi) \Delta_C  \chi = (\chi \otimes \id_D) \Delta_D$. This means that $\chi$ is a right $D$-comodule morphism in $\cc$. Similarly, $\chi$ is also a left $D$-comodule morphism in $\cc$. Separable and naturally full cases follow in a similar way.
\end{proof}
\end{invisible}

For the particular case where the coalgebra morphism is given by the counit $\varepsilon:C\to 1$ of a coalgebra $C$ in $\cc$, the monoidal unit is not required to be a left $\otimes$-cogenerator. In fact, we get the following duals of Proposition \ref{lax monoidal variation unit} and Proposition \ref{lax-monoidal-separable}.

\begin{prop}\label{colax-monoidal-counit}
Let $(\cc, \otimes, 1)$, $(\dd, \otimes, 1')$ be monoidal categories, $C$ be a coalgebra in $\cc$ with counit $\varepsilon_C:C\to 1$. Let $(F:\cc\to\dd,\psi^2,\psi^0)$ be a colax monoidal functor. 
\begin{itemize}
\item[$i)$] If  $\psi^0$ is a split-mono and $-\otimes C:\cc\to\cc^C$ is semiseparable (resp., naturally full), then $-\otimes F(C):\dd \to \dd^{F(C)}$ is semiseparable (resp.,  naturally full).

\item[$ii)$] If $\psi^0$ is a split-epi and $-\otimes C:\cc\to\cc^C$ is separable, then $-\otimes F(C):\dd \to \dd^{F(C)}$ is separable. 
\end{itemize}
\end{prop}

\begin{invisible}
\begin{proof}
    $i)$. If $-\otimes C$ is semiseparable, by Proposition \ref{prop:coalgebra-induct-functor}, there is a morphism $\chi: 1\to C$ in $\cc$ such that $\varepsilon_C\chi \varepsilon_C =\varepsilon_C$. It follows that $F(\varepsilon_C) F(\chi) F(\varepsilon_C)=F(\varepsilon_C)$. Since there is a morphism $\phi^0:1'\to F(1)$ in $\dd$ satisfying $\phi^0 \psi^0 =\id_{F(1)}$, we get $\psi^0 F(\varepsilon_C) F(\chi) \phi^0 \psi^0 F(\varepsilon_C)=\psi^0 F(\varepsilon_C)$, which means $\varepsilon_{F(C)}  F(\chi)\phi^0 \varepsilon_{F(C)} = \varepsilon_{F(C)}$. By Proposition \ref{prop:coalgebra-induct-functor} again, we obtain that $-\otimes F(C)$ is semiseparable. The naturally full case follows similarly.

$ii)$.   If $-\otimes C$ is separable, then by Proposition \ref{prop:coalgebra-induct-functor} there is a morphism $\chi: 1\to C$ in $\cc$ such that $\varepsilon_C\chi=\id_1$. It follows that $ F(\varepsilon_C) F(\chi)=\id_{F(1)}$. Since there is a morphism $\phi^0:1'\to F(1)$ in $\dd$ satisfying $\psi^0 \phi^0 =\id_{1'}$, we have $\varepsilon_{F(C)} F(\chi)\phi^0= \psi^0 F(\varepsilon_C)F(\chi)\phi^0=\psi^0\phi^0 = \id_{1'}$. By Proposition \ref{prop:coalgebra-induct-functor} $-\otimes F(C)$ is separable.
 
    \end{proof}
\end{invisible}

\begin{prop}
 \label{colax-monoidal-separable}
Let $(\cc, \otimes, 1)$, $(\dd, \otimes, 1')$ be monoidal categories, $C$ be a coalgebra in $\cc$ with counit $\varepsilon_C:C\to 1$. Let $(F:\cc\to\dd,\psi^2,\psi^0)$ be a fully faithful colax monoidal functor. If $\psi^0$ is a monomorphism and $-\otimes F(C):\dd \to \dd^{F(C)}$ is semiseparable (resp., separable, naturally full), then $-\otimes C:\cc\to\cc^C$ is semiseparable (resp., separable, naturally full).
\end{prop}
\begin{invisible}
    \begin{proof} If $-\otimes F(C)$ is semiseparable, by Proposition \ref{prop:coalgebra-induct-functor}, we know that $\varepsilon_{F(C)} \chi' \varepsilon_{F(C)} = \varepsilon_{F(C)}$ for some morphism $\chi':1'\to F(C) $ in $\dd$. Since $F$ is full, there is a morphism $\chi:1\to C$ such that $F(\chi) = \chi'\psi^0$. This implies $\psi^0 F(\varepsilon_C) F(\chi) F(\varepsilon_C) =\psi^0 F(\varepsilon_C)$. Because $\phi^0 \psi^0 =\id_{F(1)}$ and $F$ is faithful, we obtain $\varepsilon_C\chi \varepsilon_C=\varepsilon_C$. As a result, by Proposition \ref{prop:coalgebra-induct-functor}, $-\otimes C$ is semiseparable. The naturally full case follows similarly.

If $-\otimes F(C)$ is separable, by Proposition \ref{prop:coalgebra-induct-functor}, we know that $\psi^0F(\varepsilon_C)\chi' = \varepsilon_{F(C)}\chi' = \id_{1'}$ for some morphism $\chi':1'\to F(C) $ in $\dd$. 
    Since $F$ is full, there is a morphism $\chi: 1\to C$ in $\cc$ such that $F(\chi) = \chi'\psi^0$. Since $\psi^0$ is a split-mono, there is a morphism $\phi^0:1'\to F(1)$ in $\dd$ such that $\id_{F(1)}=\phi^0\psi^0=\phi^0 \psi^0 F(\varepsilon_C)\chi'\psi^0= F(\varepsilon_C)F(\chi)$. Because $F$ is faithful, we obtain that $ \varepsilon_C\chi=\id_1$. As a result, by Proposition \ref{prop:coalgebra-induct-functor}, $-\otimes C$ is separable.
    \end{proof}
\end{invisible}

We now exhibit an example in which the counit $\varepsilon_C$ of a coalgebra in a monoidal category is not an epimorphism.

\begin{es}
Let $\Bbbk$ be a field and $\vect_f$ be the category of finite dimensional $\Bbbk$-vector spaces. Consider the category $\mathrm{Mat}_n(\vect_f)$ of $n$-by-$n$ matrices, for $n\geq 2$, with entries in $\vect_f$. By \cite[Example 4.1.3]{Eti15}, we know that $\mathrm{Mat}_n(\vect_f)$ is a multitensor category with tensor product given by
\[
(V\otimes W)_{il}:=\bigoplus_{j=1}^nV_{ij}\otimes W_{jl},
\]
and unit $\Bbbk \mathrm{I}_n$, where $\mathrm{I}_n$ is the identity matrix.


Consider the inclusion functor $F:\vect_f\to \mathrm{Mat}_n(\vect_f)$, given by $F(V)_{ij}=\delta_{i1}\delta_{j1}V$. It is clear that 
$$
(F(V)\otimes F(W))_{il}=\bigoplus_{j} F(V)_{ij}\otimes F(W)_{jl}=\delta_{i1}\delta_{l1}V\otimes W=F(V\otimes W)_{il}.
$$
Hence, let $\psi^2_{V,W}:=\id:F(V\otimes W)\to F(V)\otimes F(W)$, $ \psi^0: F(\Bbbk)=(\delta_{i1}\delta_{j1}\Bbbk)_{ij}\to \Bbbk\mathrm{I}_n,$ given by ${(\psi^0)}_{11}=\id_\Bbbk$ and ${(\psi^0)}_{ij}=0$ otherwise. 
One can check  that $(F,\psi^2, \psi^0)$ is a colax monoidal functor. 
\begin{invisible}
It is obvious that $(\id\otimes\phi^2(Y,Z))\phi^2(X,Y\otimes Z)=(\phi^2(X,Y)\otimes\id)\phi^2(X\otimes Y, Z)$. For every $ M=(M_{ij})_{ij}$, the left unit constraint in $\mathrm{Mat}_n(\vect_f)$is given by $\lambda'_M: \Bbbk\mathrm{I}_n\otimes M\to M$, $\lambda'_M=(\lambda_{M_{ij}})_{ij}$, where $\lambda_{M_{ij}}:\Bbbk\otimes M_{ij}\to M_{ij}$ is the left constraint for $M_{ij}$ in $\vect_f$. Then, for every $X\in\vect_f$, it follows that $\lambda'_{F(X)}(\psi^0\otimes \id)\psi^2(\Bbbk, X)=F(\lambda_X)$. The axiom for the right unit constraints is similar. 
\end{invisible}
Recall that if $(C, \Delta_C, \varepsilon_C)$ is a coalgebra in $\vect_f$, then $F(C)$ is a coalgebra in $\mathrm{Mat}_n(\vect_f)$ with comultiplication $\psi^2_{C,C} F(\Delta_C)$ and counit $\psi^0 F(\varepsilon_C)$. It is clear that the counit of $F(C)$ is not surjective. 


Let $C$ be a nonzero coalgebra in $\vect_f$.  By Proposition \ref{prop:coalgebra-induct-functor}, we know that 
$-\otimes_\Bbbk C:\vect_f\to \mathcal{M}^C$ is always separable, because $\varepsilon_C$ is surjective.  Since $\psi^0$ is a split-mono through the canonical projection $\phi^0: \Bbbk\mathrm{I}_n\to F(\Bbbk)$, then by Proposition  \ref{colax-monoidal-counit} $i)$ we get that $-\otimes F(C)$ is semiseparable.  However, $-\otimes F(C)$ is  not separable as the counit of $F(C)$ is not surjective. 

\end{es}

\section{Semiseparability results in a duoidal category}\label{sect:duoidal}

As seen in Proposition \ref{prop:algebra-monoidalfunctor} and Proposition \ref{prop:coalgebra-monoidalfunctor}, one of the most natural variations of (co)induction functors is given through a monoidal functor. One may consider if there are other variations which preserve the semiseparability of (co)induction functors. Indeed, such a behavior can be fulfilled by the direct sum in an abelian monoidal category or the tensor product in a pre-braided monoidal category. More generally, we now prove these results within the realm of duoidal categories.

We recall
the notion of duoidal category \cite[Definition 6.1]{AM10}, also known as
2-monoidal category. 

\begin{defn}\label{defn:duoidalcategory}
    A \emph{duoidal category} is a category $\cc$ equipped with two monoidal structures $\cc^\circ= (\cc, \circ, I, a^\circ, l^\circ, r^\circ )$ and $\cc^\bullet = (\cc, \bullet, J, a^\bullet, l^\bullet, r^\bullet )$,  related through a natural transformation 
\[
\zeta=(\zeta_{A,B,C,D}: (A\bullet B)\circ (C\bullet D)\to (A\circ C)\bullet (B\circ D))_{A,B,C,D\in\cc},
\]
called the \emph{interchange law}, and morphisms $\delta:I\to I\bullet I$, $\varpi: J\circ J\to J$, $\tau:I\to J$ 
such that 
\begin{itemize}
    \item[$i)$] $(J,\varpi, \tau)$ is an algebra in $\cc^\circ$ and $(I, \delta, \tau)$ is a coalgebra in $\cc^\bullet$;
    \item[$ii)$] the following associativity conditions hold, for all $A,B,C,D,E,F\in\cc$,
    \begin{equation}\label{eq:assoc1}
    \begin{split}
\zeta_{A,B, C\circ E, D\circ F}(\id_{A\bullet B}&\circ\zeta_{C,D,E,F})a^\circ_{A\bullet B, C\bullet D, E\bullet F}\\&=(a^\circ_{A,C,E}\bullet a^\circ_{B,D,F})\zeta_{A\circ C, B\circ D,E,F}(\zeta_{A,B,C,D}\circ\id_{E\bullet F})
 \end{split}   \end{equation}
 \begin{equation}\label{eq:assoc2}
    \begin{split}
(\id_{A\circ D}\bullet \zeta_{B,C,E,F})&\zeta_{A,B\bullet C, D, E\bullet F}(a^\bullet_{A,B,C}\circ a^\bullet_{D,E,F})\\&=a^\bullet_{A\circ D, B\circ E, C\circ F}(\zeta_{A,B,D, E}\bullet \id_{C\circ F})\zeta_{A\bullet B, C, D\bullet E, F};
    \end{split} 
 \end{equation}
 \item[$iii)$] the following unitality conditions hold, for all $A,B\in\cc$
 \begin{equation}\label{eq:unit1}
         (l^\circ_{A}\bullet l^\circ_B)\zeta_{I,I, A,B}(\delta\circ\id_{A\bullet B})=l^\circ_{A\bullet B};\quad (r^\circ_A\bullet r^\circ_B)\zeta_{A,B,I,I}(\id_{A\bullet B}\circ\delta)=r^\circ_{A\bullet B};
 \end{equation}
 \begin{equation}\label{eq:unit2}
    l^\bullet_{A\circ B}(\varpi\bullet \id_{A\circ B})\zeta_{J,A,J,B}=l^\bullet_A\circ l^\bullet_B;\quad r^\bullet_{A\circ B}(\id_{A\circ B}\bullet \varpi)\zeta_{A,J,B,J}=r^\bullet_A\circ r^\bullet_B. 
 \end{equation}
\end{itemize}
\end{defn}

In this section, we always consider a duoidal category $\cc$ with structure $\circ, I,\bullet,  J,\zeta, \delta, \varpi, \tau$ defined as above. 
In \cite[Theorem 2.16]{BM12} it is shown that every duoidal category is duoidal equivalent to a strict duoidal category, i.e.\ one where both monoidal categories are strict. In view of this, we will omit the associativity constraints. 

We denote by $\mathrm{Alg}(\cc^{\circ})$ the category of algebras and algebra morphisms in $\cc^{\circ}$ and by $\mathrm{Coalg}(\cc^{\bullet})$ the category of coalgebras and coalgebra morphisms in $\cc^{\bullet}$. Recall from \cite[Proposition 6.35]{AM10} that $\bullet$ induces a monoidal structure on $\mathrm{Alg}(\cc^{\circ})$, and $\circ$ induces a monoidal structure on $\mathrm{Coalg}(\cc^{\bullet})$. Explicitly, for $(A, m_A, u_A)$, $(B, m_B, u_B)$ in $\mathrm{Alg}(\cc^{\circ})$, $A\bullet B$ is in $\mathrm{Alg}(\cc^{\circ})$ with multiplication $m_{A\bullet B}=(m_A\bullet m_B)\zeta_{A,B,A,B}$ and unit $u_{A\bullet B}=(u_A\bullet u_B)\delta$. For $(C, \Delta_C, \varepsilon_C)$, $(D, \Delta_D, \varepsilon_D)$ in $\mathrm{Coalg}(\cc^{\bullet})$,  $C\circ D$ is in $\mathrm{Coalg}(\cc^{\bullet})$ with comultiplication $\Delta_{C\circ D}=\zeta_{C,C,D,D}(\Delta_C\circ\Delta_D)$ and counit $\varepsilon_{C\circ D}=\varpi (\varepsilon_C\circ \varepsilon_D)$.

\subsection{Combination of induction functors}\label{subsect:induct-duoidal} Throughout this subsection, let $\cc$ be a duoidal category with coequalizers such that any object in $\cc^\circ$ is coflat. We consider the induction functor attached to an algebra morphism in $\cc^\circ$, defined as in Subsection \ref{subsect:induct}. 

\begin{prop}\label{prop:duoidal}
Suppose $I$ is a left $\circ$-generator of $(\cc,\circ,I)$. Let $f_1:R_1\to S_1$, $f_2:R_2\to S_2$ be morphisms in $\mathrm{Alg}(\cc^{\circ})$ such that $f_1^*=-\circ_{R_1} S_1:\cc_{R_1}\to \cc_{S_1}$ and $f_2^*=-\circ_{R_2} S_2:\cc_{R_2}\to \cc_{S_2}$ are semiseparable (resp., separable, naturally full). Then, $(f_1\bullet f_2)^*= -\circ_{R_1\bullet R_2} (S_1\bullet S_2):\cc_{R_1\bullet R_2}\to \cc_{S_1\bullet S_2}$ is also semiseparable (resp., separable, naturally full).
\end{prop}

\begin{proof}
Suppose $f_1^*=-\circ_{R_1} S_1$ and $f_2^*=-\circ_{R_2} S_2$ are semiseparable (resp., separable, resp., naturally full). By Theorem \ref{thm:phi*semisep} (resp., Theorem \ref{thm:sepmon}, resp., Proposition \ref{prop:phi*natfull}), there exist an $R_1$-bilinear morphism $E_1:S_1\to R_1$ and an $R_2$-bilinear morphism $E_2:S_2\to R_2$ in $\cc$ such that $f_1E_1f_1=f_1$, $f_2E_2f_2=f_2$ (resp., $E_1f_1=\id_{R_1}$, $E_2f_2=\id_{R_2}$, resp., $f_1E_1=\id_{S_1}$, $f_2E_2=\id_{S_2}$). Due to the naturality of $\zeta$, we have $((f_1 \circ \id_{S_1})\bullet (f_2 \circ \id_{S_2})) \zeta_{R_1,R_2,S_1,S_2} = \zeta_{S_1,S_2,S_1,S_2}((f_1\bullet f_2)\circ \id_{S_1\bullet S_2})$ and $\zeta_{R_1,R_2,R_1,R_2}(\id_{R_1 \bullet R_2} \circ (E_1\bullet E_2))= ((\id_{R_1} \circ E_1)\bullet (\id_{R_2} \circ E_2)) \zeta_{R_1,R_2,S_1,S_2}$. Note that $f_1\bullet f_2$ is an algebra morphism in $\cc^\circ$, so it induces an $R_1\bullet R_2$-bimodule structure on $S_1\bullet S_2$. Thus, 
$$
\begin{aligned}
(E_1\bullet E_2)m_{S_1\bullet S_2}((f_1\bullet f_2)\circ \id_{S_1\bullet S_2}) &= (E_1\bullet E_2)(m_{S_1}\bullet m_{S_2})\zeta_{S_1,S_2,S_1,S_2}((f_1\bullet f_2)\circ \id_{S_1\bullet S_2})\\ 
&= (E_1\bullet E_2)(m_{S_1}\bullet m_{S_2})((f_1 \circ \id_{S_1})\bullet (f_2 \circ \id_{S_2})) \zeta_{R_1,R_2,S_1,S_2}\\
&= ((E_1 m_{S_1}(f_1 \circ \id_{S_1}))\bullet (E_2 m_{S_2}(f_2 \circ \id_{S_2})))\zeta_{R_1,R_2,S_1,S_2}\\
&= (m_{R_1}(\id_{R_1}\circ E_1))\bullet (m_{R_2}(\id_{R_2}\circ E_2))\zeta_{R_1,R_2,S_1,S_2}\\ 
&=(m_{R_1}\bullet m_{R_2})\zeta_{R_1,R_2,R_1,R_2}(\id_{R_1 \bullet R_2} \circ (E_1\bullet E_2))\\
&=m_{R_1\bullet R_2}(\id_{R_1 \bullet R_2} \circ (E_1\bullet E_2)).
\end{aligned}
$$
As a result, $E_1\bullet E_2$ is left $R_1 \bullet R_2$-linear. Similarly, one can check that it is also right $R_1 \bullet R_2$-linear. \begin{invisible}
 Indeed, \[\begin{aligned}
(E_1\bullet E_2)m_{S_1\bullet S_2}(\id_{S_1\bullet S_2}\circ (f_1\bullet f_2)) &= (E_1\bullet E_2)(m_{S_1}\bullet m_{S_2})\zeta_{S_1,S_2,S_1,S_2}( \id_{S_1\bullet S_2}\circ (f_1\bullet f_2))\\ 
&= (E_1\bullet E_2)(m_{S_1}\bullet m_{S_2})(( \id_{S_1}\circ f_1)\bullet (\id_{S_2}\circ f_2)) \zeta_{S_1,S_2,R_1,R_2}\\
&= ((E_1 m_{S_1}( \id_{S_1}\circ f_1))\bullet (E_2 m_{S_2}( \id_{S_2}\circ f_2)))\zeta_{S_1,S_2,R_1,R_2}\\
&= (m_{R_1}(E_1\circ\id_{R_1}))\bullet (m_{R_2}(E_2\circ\id_{R_2}))\zeta_{S_1,S_2,R_1,R_2}\\ 
&=(m_{R_1}\bullet m_{R_2})\zeta_{R_1,R_2,R_1,R_2}( (E_1\bullet E_2)\circ\id_{R_1 \bullet R_2} )\\
&=m_{R_1\bullet R_2}( (E_1\bullet E_2)\circ \id_{R_1 \bullet R_2}).
\end{aligned}   \]
\end{invisible}
Note that $(f_1\bullet f_2)(E_1\bullet E_2)(f_1\bullet f_2)=f_1\bullet f_2$ (resp., $(E_1\bullet E_2)(f_1\bullet f_2)=\id_{R_1}\bullet \id_{R_2}$, resp., $(f_1\bullet f_2)(E_1\bullet E_2)= \id_{S_1}\bullet \id_{S_2}$). 
It follows from Theorem \ref{thm:phi*semisep} (resp., Theorem \ref{thm:sepmon},  Proposition \ref{prop:phi*natfull}) that $(f_1\bullet f_2)^*= -\circ_{R_1\bullet R_2} (S_1\bullet S_2)$ is also semiseparable (resp., separable,  naturally full).
\end{proof}

The following corollaries are direct applications of the previous result.

\begin{cor}\label{additivemonoidal}
Let $(\cc,\otimes,1)$ be a monoidal category with finite products such that the unit $1$ is a left $\otimes$-generator. Suppose $f_1:R_1\to S_1$, $f_2:R_2\to S_2$ are algebra morphisms in $\cc$ such that $f_1^*=-\otimes_{R_1} S_1$ and $f_2^*=-\otimes_{R_2} S_2$ are semiseparable (resp., separable, naturally full), then $(f_1\times f_2)^*=-\otimes_{R_1\times R_2} (S_1\times S_2):\cc_{R_1\times R_2}\to \cc_{S_1\times S_2}$ is also semiseparable (resp., separable, naturally full).
\end{cor}

\begin{proof}
It follows from Proposition \ref{prop:duoidal}, since we have the cartesian
monoidal structure $(C, \times, 0)$ and $\cc$ is a duoidal category with $\circ = \otimes$, $\bullet=\times$, in which $I=1$, $J=0$, $\delta:1\to 1\times 1$, $\varpi:0\otimes0\to 0$, $\tau=0_{1,0}:1 \to 0$. In this case, $\zeta_{A,B,C,D}:(A\times B)\otimes(C\times D) \to (A\otimes C)\times (B\otimes D)$ is defined by $\zeta_{A,B,C,D}=\langle
    p_A\otimes p_C,
    p_B\otimes p_D
\rangle$, see \cite[Example 6.19]{AM10}. 
\end{proof}


\begin{cor}\label{cor:prebraid}
Let $\cc$ be a pre-braided monoidal category satisfying the unit $1$ is a left $\otimes$-generator for $\cc$. Suppose $f_1:R_1\to S_1$, $f_2:R_2\to S_2$ are algebra morphisms in $\cc$ such that $f_1^*=-\otimes_{R_1} S_1$ and $f_2^*=-\otimes_{R_2} S_2$ are semiseparable (resp., separable, naturally full), then $(f_1\otimes f_2)^*=-\otimes_{R_1\otimes R_2} (S_1\otimes S_2)$ is also semiseparable (resp., separable, naturally full).
\end{cor}
\begin{proof}
   It follows from Proposition \ref{prop:duoidal} since a pre-braided monoidal category is a duoidal category with $\circ = \bullet = \otimes$, see \cite[Proposition 6.14]{AM10}.
\end{proof}

\begin{es}\label{es:duoidal}
Let $H$ be a bialgebra over a field $\Bbbk$. By \cite[Remark 1.6]{Sar21}, see also  \cite[Proposition 2.7]{BPT25}, the category $\cc={}_H\m_H$ of $H$-bimodules is duoidal with monoidal structures $\cc^\circ =(\cc, \circ=\otimes_H, I=H)$ and $\cc^\bullet=(\cc, \bullet=\otimes_\Bbbk, J=\Bbbk )$. Recall from \cite[Section 2]{BPT25} that, if $H$ is an algebra over a field $\Bbbk$, an algebra in $\cc$ is a pair $(A,i)$ consisting of a $\Bbbk$-algebra $A$ and a $\Bbbk$-algebra morphism $i:H\to A$. An algebra morphism $f:(A,i)\to (B,j)$ in $\cc$ is a $\Bbbk$-algebra morphism $f:A\to B$ such that $fi=j$. By \cite[Lemma 2.1]{BPT25}, an algebra morphism in $\cc$ is a $\Bbbk$-algebra morphism that is at the same time left and right $H$-linear.

Assume that $H$ is a left $\circ$-generator. By Proposition \ref{prop:duoidal}, if $f$ and $g$ are algebra morphisms in $\cc^\circ=(\cc, \otimes_H, H)$ such that the induction functors $f^*$, $g^*$ are semiseparable (resp., separable, naturally full), then so is $(f\otimes_\Bbbk g)^{*}$. 
For instance, consider $H$ to be the group algebra $\Bbbk G$, where $G$ is a finite group such that $\mathrm{char}(\Bbbk)\nmid \vert G\vert$. By Example \ref{es:bimodules}, $\Bbbk G$ is a left $\otimes_{\Bbbk G}$-generator, and we know that $\psi:\Bbbk G\oplus \Bbbk G\to\Bbbk G$, $(a,b)\mapsto a$, and $\varphi: \Bbbk G\to \Bbbk G\otimes_\Bbbk \Bbbk G$, $\sum_{i=0}^{n-1} k_ig_i\mapsto \sum_{i=0}^{n-1}k_ig_i\otimes_\Bbbk g_i$ are algebra morphisms in $\cc^\circ$, such that the induction functor $\varphi^*$ is separable and the induction functor $\psi^*$ is naturally full. In particular, $\varphi^*$, $\psi^*$ are semiseparable. By Proposition \ref{prop:duoidal}, we obtain that $(\varphi\otimes_\Bbbk \psi)^*$ is semiseparable.
\end{es}
    

Now, we turn our attention to the cases where the unit object $I$ is not necessarily a left $\circ$-generator.

\begin{prop}\label{prop:duoidal-delta}
Let $\cc$ be a duoidal category. Let $R$, $S$ be algebras in $\cc^\circ$ such that the functors 
$-\circ R: \cc\to \cc_R$, $-\circ S:\cc\to\cc_S$ are semiseparable (resp., naturally full). Suppose $\delta:I \to I \bullet I$ is split-epi. Then, the functor 
$-\circ (R\bullet S):\cc\to\cc_{R\bullet S}$ is semiseparable (resp., naturally full).
\end{prop}

\begin{proof}
By Proposition \ref{prop:algebra-induct-functor}, there are morphisms $\chi_R$, $\chi_S$ such that $u_R \chi_R u_R = u_R$, $u_S \chi_S u_S = u_S$ (resp., $u_R \chi_R = \id_R$, $u_S \chi_S = \id_S$). It follows that $(u_R \bullet u_S) (\chi_R \bullet \chi_S) (u_R \bullet u_S) = u_R \bullet u_S$ (resp., $(u_R \bullet u_S) (\chi_R \bullet \chi_S) = \id_R \bullet \id_S$). Since $\delta \lambda = \id_{I \bullet I}$ for some morphism $\lambda:I\bullet I\to I$, we get $(u_R \bullet u_S) \delta \lambda (\chi_R \bullet \chi_S) (u_R \bullet u_S) \delta = (u_R \bullet u_S) \delta$ (resp., $(u_R \bullet u_S) \delta \lambda (\chi_R \bullet \chi_S) = \id_R \bullet \id_S$). 
This means $u_{R\bullet S} \lambda (\chi_R \bullet \chi_S) u_{R\bullet S} = u_{R\bullet S}$ (resp., $u_{R\bullet S} \lambda (\chi_R \bullet \chi_S) = \id_R\bullet \id_S$), i.e. $u_{R\bullet S}$ is regular (resp., split-epi). By Proposition \ref{prop:algebra-induct-functor} again, the functor $-\circ (R\bullet S)$ is semiseparable (resp., naturally full).
\end{proof}

\begin{prop}
Let $\cc$ be a duoidal category. Let $R$, $S$ be algebras in $\cc^\circ$ such that the functors $-\circ R:\cc\to\cc_R$, $-\circ S:\cc\to\cc_S$ are separable. Then, $-\circ (R\bullet S):\cc\to\cc_{R\bullet S}$ is separable.
\end{prop}

\begin{proof}
By Proposition \ref{prop:algebra-induct-functor}, there are morphisms $\chi_R: R\to I$, $\chi_S:S\to I$ in $\cc$ such that $\chi_R u_R = \id_I$, $\chi_S u_S = \id_I$, hence we have that $r^\bullet_I(\id_I \bullet \tau) (\chi_R \bullet \chi_S) u_{R\bullet S}= r^\bullet_I(\id_I \bullet \tau)(\chi_R \bullet \chi_S) (u_R \bullet u_S) \delta = r^\bullet_I(\id_I \bullet \tau) \delta = r^\bullet_I(r^\bullet_I)^{-1}=\id_{I}$, where $\tau: I \to J$ is the counit of $I$ in $\cc^{\bullet}$.
\end{proof}

\begin{es}
    Let $X$ be a set. We consider the duoidal category $\mathsf{span}(X)$ of spans over $X$, see \cite[4.2]{BCZ13}, which was introduced in \cite[Example 6.17]{AM10} under the name of \emph{category of directed graphs with vertex set $X$}. The objects of $\mathsf{span}(X)$ are triples $(A,s,t)$, where $A$ is a set and $s$, $t$ are maps $A\to X$, called the \emph{source} and \emph{target maps}, respectively. The morphisms in $\mathsf{span}(X)$ are maps $f: (A,s,t)\to (A',s',t')$ such that $s'f=s$ and $t'f=t$.

For any spans $A,B$ over $X$, one monoidal structure is 
\[
A\circ B= \{ (a,b)\in A\times B\,\vert\, s(a)=t(b)\}\quad \text{and}\quad I=(X, \id,\id),
\]
where $s(a,b):=s(b)$ and $t(a,b):=t(a)$, 
while the other is 
\[
A\bullet B= \{(a,b)\in A\times B\,\vert\, s(a)=s(b), t(a)=t(b)\}\quad \text{and}\quad J=(X\times X, p_2, p_1),
\]
where $p_1$, $p_2$ are the canonical projections and $s(a,b):=s(a)=s(b)$, $t(a,b):=t(a)=t(b)$. 
The interchange law is
\[
\zeta := (A\bullet B)\circ (A'\bullet B')\to (A\circ A')\bullet (B\circ B'),\quad (a,b,a',b')\mapsto (a,a',b,b'). 
\]
The $\circ$-monoidal unit $I$ is a comonoid with respect to $\bullet$ via the comultiplication 
\[
\delta : I\to I \bullet I = \{ (x,y)\in X\times X\,\vert\, x=y\}\cong I,\quad x\mapsto (x,x)\cong x.
\]
The $\bullet$-monoidal unit $J$ is a monoid with respect to $\circ$ via the multiplication
\[
\varpi :J\circ J =\{ (x,y,x',y')\,\vert\, y=x'\}\to J, \quad (x,y,x', y')\mapsto (x,y').
\]
The counit of the comonoid $I$ and the unit of the monoid $J$ are both given by 
$\tau:I\to J$, $x\mapsto (x,x)$.

We adopt some notations here. Let $A$ be an object in $\mathsf{span}(X)$. For any pair of vertices $i,j$ in $X$, the set of arrows from $i$ to $j$ in $A$, i.e. of elements $a\in A$ such that $s(a)=i$ and $t(a)=j$,  is denoted by $\mathrm{Arr}_{A}(i,j)$, which is a subset of $A$. For a morphism $u:A \to B$ in $\mathsf{span}(X)$, we denote the restriction of $u$ to the arrow set $\mathrm{Arr}_{A}(i,j)$ by $u_{ij}: \mathrm{Arr}_{A}(i,j) \to \mathrm{Arr}_{B}(i,j)$.

For two morphisms $u,v: A_1 \to A_2$ in $\mathsf{span}(X)$, we claim that their coequalizer exists. For any pair of vertices $i,j$ in $X$, we define  $\mathrm{Coeq}(u,v)$ by the directed graph such that $\mathrm{Arr}_{\mathrm{Coeq}(u,v)}(i,j):= \mathrm{Coeq}(u_{ij},v_{ij})$, where the latter is the coequalizer of $u_{ij}$, $v_{ij}$ in $\Set$. For any morphism $h: A_2 \to Z$ in $\mathsf{span}(X)$ such that $hu=hv$, we know $h_{ij}u_{ij}=h_{ij}v_{ij}$. Therefore, by the universal property of the coequalizer $q_{ij}: \mathrm{Arr}_{A_2}(i,j) \to \mathrm{Arr}_{\mathrm{Coeq}(u,v)}(i,j)$ in $\Set$, we obtain a unique map $l_{ij}:\mathrm{Arr}_{\mathrm{Coeq}(u,v)}(i,j) \to \mathrm{Arr}_{Z}(i,j)$ such that $l_{ij} q_{ij} = h_{ij}$ for any pair of vertices $i$, $j$ in $X$, and hence a morphism $l:\mathrm{Coeq}(u,v) \to Z$ such that $l q = h$, where $q: A_2 \to \mathrm{Coeq}(u,v)$ is induced by $q_{ij}$ for all vertices $i,j$ in $X$. This means $\mathrm{Coeq}(u,v)$ is the coequalizer of $u,v$ in $\mathsf{span}(X)$.  

Next, we claim that every object in $(\mathsf{span}(X), \circ)$ is coflat. Let $A_1$, $A_2$, $B$ be objects in $\mathsf{span}(X)$. For two morphisms $u,v: A_1 \to A_2$ in $\mathsf{span}(X)$, we want to show $\id_{B} \circ\mathrm{Coeq}(u, v)$ is the coequalizer of $\id_{B} \circ u$ and $\id_{B} \circ v$. Suppose there is a morphism $h: B\circ A_2 \to Z$ in $\mathsf{span}(X)$ such that $h(\id_{B} \circ u) = h (\id_{B} \circ v)$, i.e.\ for any element $b\circ a$ in $B\circ A_1$, we have $h(b\circ u(a))= h(b\circ v(a))$. We claim that $h_{s(a),t(b)}(\id \times u_{s(a),t(a)}) = h_{s(a),t(b)}(\id \times v_{s(a),t(a)})$. Indeed, for any $(b', a')$ in $\mathrm{Arr}_B(t(a),t(b)) \times \mathrm{Arr}_{A_1}(s(a),t(a)) = \mathrm{Arr}_{B\circ A_1}(s(a), t(b))$, we have 
\begin{align*}
h_{s(a),t(b)}(\id \times u_{s(a),t(a)})(b',a') = &h_{s(a),t(b)}(b',u_{s(a),t(a)}(a')) = h_{s(a),t(b)}(b'\circ u(a')) = h(b'\circ u(a')) \\ = &h(b'\circ v(a')) = h_{s(a),t(b)}(\id \times v_{s(a),t(a)})(b',a').
\end{align*}
Since every object in $\Set$ is coflat, we know that $\mathrm{Arr}_B(t(a), t(b))\times \mathrm{Arr}_{\mathrm{Coeq(u,v)}}(s(a), t(a))$ is the coequalizer of $\id \times u_{s(a),t(a)}$ and $\id \times v_{s(a),t(a)}$. Consequently, by its universal property, there exists a unique  map $l_{s(a),t(b)}$ such that the diagram 
\[
\xymatrixcolsep{1.2cm}\xymatrix{
\mathrm{Arr}_B(t(a), t(b))\times \mathrm{Arr}_{A_2}(s(a), t(a))\ar[r]^-{\id\times  q_{s(a),t(a)}}\ar[d]_-{h_{s(a),t(b)}}&\mathrm{Arr}_B(t(a), t(b))\times \mathrm{Arr}_{\mathrm{Coeq(u,v)}}(s(a), t(a)) \ar@{.>}[ld]^-{l_{s(a),t(b)} }\\\mathrm{Arr}_Z(s(a), t(b)) &}
\]
commutes. Since $a$ is in $A_1$, the maps $l_{s(a),t(b)}$ are not sufficient to determine a morphism from $B\circ \mathrm{Coeq(u,v)}$ to $Z$ directly. To overcome this, we discuss two different cases.

For any element $b\circ a'$ in $B\circ A_2$, $(b, a')$ is in $\mathrm{Arr}_B(t(a'), t(b))\times \mathrm{Arr}_{A_2}(s(a'), t(a'))$. In the first case, suppose there is an arrow $a$ in $\mathrm{Arr}_{A_1}(s(a), t(a))$ such that $s(a) = s(a')$, $t(a) = t(a')$. This means $(b, a')$ is in $\mathrm{Arr}_B(t(a), t(b))\times \mathrm{Arr}_{A_2}(s(a), t(a))$. As said above, there is a unique  $l_{s(a'),t(b)}$ such that $h_{s(a'),t(b)} = l_{s(a'),t(b)}(\id\times  q_{s(a'),t(a')})$. In the second case, there is no arrow in $\mathrm{Arr}_{A_1}(s(a'),t(a'))$ i.e. $\mathrm{Arr}_{A_1}(s(a'),t(a'))$ is an empty set. By definition, 
$$
\mathrm{Arr}_{\mathrm{Coeq(u,v)}}(s(a'), t(a'))= \mathrm{Coeq}(u_{s(a')t(a')},v_{s(a')t(a')}) = \mathrm{Arr}_{A_2}(s(a'),t(a')),
$$
since the domains of $u_{s(a')t(a')}$ and $v_{s(a')t(a')}$ are empty sets. Also note that $q_{s(a'),t(a')} = \id$. Hence, we define $l_{s(a'),t(a')} = h_{s(a'),t(a')}$. It is clear that $h_{s(a'),t(b)} = l_{s(a'),t(b)}(\id\times  q_{s(a'),t(a')})$. In summary, there is a unique morphism $l: B\circ \mathrm{Coeq(u,v)} \to Z$ such that $h= l (\id_B\circ q)$.

Note that Proposition \ref{prop:duoidal-delta} applies as $\delta$ is an isomorphism. Moreover, an algebra in $(\mathsf{span}(X), \circ )$ is precisely a small category with object set $X$, see \cite[Example 6.43]{AM10}.


\end{es}

\subsection{Combination of coinduction functors} Now, let $\cc$ be a duoidal category with equalizers and we assume that any object in $\cc^\bullet$ is flat. We consider the coinduction functor attached to a coalgebra morphism in $\cc^\bullet$, as defined in Section \ref{sect:coinduct}. Here, given a coalgebra morphism $\psi:C\to D$ in $\cc^\bullet$, for every $M$ in $\cc^D$, we denote  $M\blacksquare_D C :=\mathrm{Eq}(\rho^D_M\bullet\id_C, \id_M\bullet \lambda^D_C)$. 

\begin{prop}\label{prop:duoidal-coind}
Suppose $J$ is a left $\bullet$-cogenerator of $(\cc,\bullet,J)$. Let $f_1:C_1\to D_1$, $f_2:C_2\to D_2$ be morphisms in $\Coalg(\cc^{\bullet})$ such that the coinduction functors $f_1^*=-\blacksquare_{D_1} C_1:\cc^{D_1}\to\cc^{C_1}$ and $f_2^*=-\blacksquare_{D_2} C_2:\cc^{D_2}\to\cc^{C_2}$ are semiseparable (resp., separable, naturally full). Then, $(f_1\circ f_2)^*= -\blacksquare_{D_1\circ D_2} (C_1\circ C_2):\cc^{D_1\circ D_2}\to \cc^{C_1\circ C_2}$ is also semiseparable (resp., separable, naturally full).
\end{prop}
\begin{proof}
    Suppose $f_1^*=-\blacksquare_{D_1}C_1$ and $f_2^*=-\blacksquare_{D_2}C_2$ are semiseparable (resp., separable, naturally full). By Theorem \ref{prop:semisep-coind-nocogen} (resp., Proposition \ref{prop:sep-coind-nocogen}, Proposition \ref{prop:natfull-coind}), there exist a $D_1$-bicolinear morphism $\chi_1: D_1\to C_1$ such that $f_1 \chi_1 f_1= f_1$ (resp., $f_1\chi_1=\id_{D_1}$, $\chi_1 f_1=\id_{C_1}$)
    and a $D_2$-bicolinear morphism $\chi_2: D_2\to C_2$ such that $f_2 \chi_2 f_2= f_2$   
    (resp., $f_2\chi_2=\id_{D_2}$, $\chi_2 f_2=\id_{C_2}$). We check that $\chi_1\circ\chi_2$ is $(D_1\circ D_2)$-colinear.     Note that $(f_1\circ f_2)(\chi_1\circ\chi_2)(f_1\circ f_2)=f_1\circ f_2$ (resp., $(f_1\circ f_2)(\chi_1\circ \chi_2)=\id_{D_1\circ D_2}$, $(\chi_1\circ \chi_2)(f_1\circ f_2)=\id_{C_1\circ C_2}$). We check that $\chi_1\circ\chi_2$ is a $(D_1\circ D_2)$-bicolinear morphism in $\cc$. We have 
    \[
    \begin{split}
    (\id_{D_1\circ D_2}\bullet (\chi_1\circ\chi_2))\Delta_{D_1\circ D_2}&= (\id_{D_1\circ D_2}\bullet (\chi_1\circ\chi_2))\zeta_{D_1,D_1, D_2, D_2}(\Delta_{D_1}\circ\Delta_{D_2})\\&=\zeta_{D_1, C_1, D_2, C_2}((\id_{D_1}\bullet\chi_1)\circ (\id_{D_2}\bullet\chi_2))(\Delta_{D_1}\circ\Delta_{D_2})\\&=\zeta_{D_1, C_1, D_2, C_2}((\id_{D_1}\bullet\chi_1)\Delta_{D_1}\circ (\id_{D_2}\bullet\chi_2)\Delta_{D_2})\\&=\zeta_{D_1, C_1, D_2, C_2}(((f_1\bullet C_1)\Delta_{C_1}\chi_1)\circ ((f_2\bullet C_2)\Delta_{C_2}\chi_2))\\&=\zeta_{D_1, C_1, D_2, C_2}((f_1\bullet C_1)\circ (f_2\bullet C_2))(\Delta_{C_1}\circ \Delta_{C_2})(\chi_1\circ\chi_2)
    \\&=((f_1\circ f_2)\bullet \id_{C_1\circ C_2})\zeta_{C_1,C_1,C_2,C_2}(\Delta_{C_1}\circ\Delta_{C_2})(\chi_1\circ\chi_2)\\&=\lambda^{D_1\circ D_2}_{C_1\circ C_2}(\chi_1\circ\chi_2)
   \end{split} \]
Thus, $\chi_1\circ\chi_2$ is left $(D_1\circ D_2)$-colinear in $\cc$. Similarly, one can check that it is also right  $(D_1\circ D_2)$-colinear. By Theorem \ref{prop:semisep-coind-nocogen} (resp., Proposition \ref{prop:sep-coind-nocogen}, Proposition \ref{prop:natfull-coind}) $(f_1\circ f_2)^*$ is semiseparable (resp, separable, naturally full).  
\end{proof}

Finally, we present direct applications of Proposition \ref{prop:coalgebra-induct-functor} where the unit object $J$ is not necessarily a left $\bullet$-cogenerator.

\begin{prop}\label{prop:duoidal-coalgebras}
Let $\cc$ be a duoidal category. Let $C$, $D$ be coalgebras in $\cc^\bullet$ such that the functors $-\bullet C:\cc\to\cc^C$ and $-\bullet D:\cc\to\cc^D$ are semiseparable (resp., naturally full). Suppose $\varpi:J\circ J \to J$ is split-mono. Then, $-\bullet (C\circ D):\cc\to\cc^{C\circ D}$ is semiseparable (resp., naturally full).
\end{prop}

\begin{proof}
By Proposition \ref{prop:coalgebra-induct-functor}, there are morphisms $\chi_C:J\to C$, $\chi_D:J\to D$ in $\cc$ such that $\varepsilon_C \chi_C \varepsilon_C = \varepsilon_C$, $\varepsilon_D \chi_D \varepsilon_D = \varepsilon_D$ (resp., $\chi_C \varepsilon_C = \id_C$, $\chi_D \varepsilon_D= \id_D$). It follows that $(\varepsilon_C \circ \varepsilon_D) (\chi_C \circ \chi_D) (\varepsilon_C \circ \varepsilon_D) = \varepsilon_C \circ \varepsilon_D$ (resp., $(\chi_C \circ \chi_D) (\varepsilon_C \circ \varepsilon_D) = \id_C \circ \id_D$). Since $\varpi$ is split-mono, there exists a morphism $\lambda$ in $\cc$ such that $\lambda \varpi = \id_{J \circ J}$. Then, $\varepsilon_{C\circ D} (\chi_C \circ \chi_D) \lambda\varepsilon_{C\circ D}=\varpi(\varepsilon_C \circ \varepsilon_D) (\chi_C \circ \chi_D) \lambda \varpi (\varepsilon_C \circ \varepsilon_D) = \varpi(\varepsilon_C \circ \varepsilon_D) = \varepsilon_{C\circ D}$ (resp., $(\chi_C \circ \chi_D) \lambda \varepsilon_{C\circ D}=(\chi_C \circ \chi_D) \lambda \varpi(\varepsilon_C \circ \varepsilon_D) = \id_C \circ \id_D$). Therefore, $\varepsilon_{C\circ D}$ is regular (resp., split-mono). By Proposition \ref{prop:coalgebra-induct-functor} again, we obtain that $-\bullet (C\circ D)$ is semiseparable (resp., naturally full).
\end{proof}

\begin{prop}\label{prop:duoidal-coalgebras-sep}
Let $\cc$ be a duoidal category. Let $C$, $D$ be coalgebras in $\cc^\bullet$ such that $-\bullet C$, $-\bullet D$, are separable. Then, $-\bullet (C\circ D)$ is separable.
\end{prop}

\begin{proof}
By Proposition \ref{prop:coalgebra-induct-functor}, there are morphisms $\chi_C:J\to C$, $\chi_D:J\to D$ in $\cc$ such that $\varepsilon_{C} \chi_C  = \id_J$, $\varepsilon_D \chi_D = \id_J$. It follows that $\varepsilon_{C\circ D}(\chi_C \circ \chi_D)(\id_J \circ \tau) (r_J^{\circ})^{-1} = \varpi( \varepsilon_C \circ \varepsilon_D) (\chi_C \circ \chi_D)(\id_J \circ \tau) (r_J^{\circ})^{-1} = \varpi(\id_J \circ \tau) (r_J^{\circ})^{-1} = r_J^{\circ} (r_J^{\circ})^{-1} = \id_J$, where $\tau: I \to J$ is the unit of $J$ in $\cc^{\circ}$.
\end{proof}

\begin{es}
    Let $C$, $D$ be  coalgebras in a pre-braided monoidal category $\cc$ such that $-\otimes C$, $-\otimes D$ are semiseparable (resp., separable, naturally full). Since $\varpi:1\otimes 1\to 1$ is an isomorphism, by Proposition \ref{prop:duoidal-coalgebras} (resp., Proposition \ref{prop:duoidal-coalgebras-sep}, Proposition \ref{prop:duoidal-coalgebras}), we obtain that $-\otimes (C\otimes D)$ is also semiseparable (resp., separable, naturally full).  
\end{es}    
    

\noindent\textbf{Acknowledgements.} The authors would like to thank A.\@ Ardizzoni for meaningful suggestions and comments. This paper was written while the authors were members of the ``National Group for Algebraic and Geometric Structures and their Applications'' (GNSAGA-INdAM). L. Bottegoni was supported by a postdoctoral fellowship at the Department of Mathematics of the Universit\'e Libre de Bruxelles, within the framework of the ARC project ``From algebra to combinatorics, and back''. Z. Zuo sincerely acknowledges the support provided by CSC (China Scholarship Council) through a PhD
student fellowship (No. 202406190047). This work was partially supported by the project funded by the European Union - NextGenerationEU under NRRP, Mission 4 Component 2 CUP D53D23005960006 - Call PRIN 2022 No.\, 104 of February 2, 2022 of Italian Ministry of University and Research; Project 2022S97PMY \textit{Structures for Quivers, Algebras and Representations (SQUARE).}

\end{document}